\documentclass[12pt]{amsart}

\usepackage{amssymb}
\usepackage{stmaryrd}
\usepackage{tikz}
\usetikzlibrary{decorations.markings,decorations.pathreplacing,matrix,intersections}
\usepackage{hyperref,hyphenat}

\usepackage[a4paper,left=2.25cm,right=2.25cm,top=3cm,bottom=3cm]{geometry}
\allowdisplaybreaks[1]

\newcommand\newtheoremlink[2]{\newtheorem{raw#1}{#2}%
\newenvironment{#1}[1]{%
\def\thmhead####1####2####3{{\href{##1}{\thmname{####1} \thmnumber{####2}}\thmnote{ (####3)}}}
\csname raw#1\endcsname}{\csname endraw#1\endcsname}
}

\theoremstyle{plain}
\newtheorem{lem}{Lemma}
\newtheorem{prop}{Proposition}
\newtheorem{thm}{Theorem}
\newtheorem{thmd}{Theorem}
\newtheorem{thmdd}{Theorem}
\theoremstyle{remark}
\newtheorem*{rmk}{Remark}
\newtheoremlink{exlink}{Example}

\def\ie{{\it i.e.}, }
\def\eg{{\it e.g.}\/}
\def\cf{{\it cf.}\/}
\newcommand\bra[1]{\left<#1\right|}
\newcommand\ket[1]{\left|#1\right>}

\renewcommand\ss{\scriptstyle}
\newcommand\sss{\scriptscriptstyle}
\newdimen{\cellsize}
\newcommand\medboxes{\setlength{\cellsize}{14.22pt}\def\boxformat{}}

\newcommand\tinyboxes{\setlength{\cellsize}{6pt}\def\boxformat{\scriptscriptstyle}}
\medboxes
\tikzset{tableaubox/.style={draw=black,thin,sharp corners,solid,minimum size=\cellsize,inner sep=0pt}}
\tikzset{tableau/.style={matrix,name=tab,matrix anchor=tab-1-1.south west,inner sep=1pt,matrix of math nodes,cells={anchor=center,draw=black,thin,solid,arrows=-},nodes={tableaubox,execute at begin node=\boxformat},nodes in empty cells,row sep={\cellsize,between origins},column sep={\cellsize,between origins}}}
\makeatletter
\newcommand\cellextra[1]{#1\expandafter\tikz@lib@matrix@start@cell}
\makeatother
\def\activate#1{\begingroup
  \lccode`\~=`#1%
  \lowercase{\endgroup \let~#1}%
  \catcode`#1=13\relax}
\activate &
\newcommand\tableau[1]{\tikz[baseline=0]
\node[tableau]{#1};}

\newcommand\fund{{\tinyboxes\tableau{\\}}}

\long\def\remove#1{}
\newcommand\cev{\overset{\vbox to 0.7pt{\hbox{$\ss\shortleftarrow$}}}}
\renewcommand\vec{\overset{\vbox to 0.7pt{\hbox{$\ss\shortrightarrow$}}}}
\renewcommand\bar{\overset{\vbox to 2.5pt{\hbox{$\ss-$}}}}
\newcommand\x{{\bar x}}
\newcommand\y{{\vec y}}
\newcommand\z{{\vec z}}
\renewcommand\t{{\vec t}}
\newcommand\yy{{\cev y}}
\newcommand\zz{{\cev z}}

\renewcommand{\leq}{\leqslant}
\renewcommand{\geq}{\geqslant}
\renewcommand\ss{\scriptstyle}

\newcommand\ra{R^{\tikz[scale=0.3]{\dtile{{4,{0,0},0},{4,{0,0},1}}}}}
\newcommand\rb{R^{\tikz[scale=0.3]{\dtileb{{4,{0,1},0},{4,{0,0},1}}}}}
\newcommand\rc{R^{\tikz[scale=0.3]{\dtilec{{4,{1,0},0},{4,{0,0},1}}}}}

\tikzset{arrow/.style={postaction={decorate,thick,decoration={markings,mark = at position #1 with {\arrow{>}}}}},arrow/.default=0.5}
\tikzset{invarrow/.style={postaction={decorate,thick,decoration={markings,mark = at position #1 with {\arrow{<}}}}},invarrow/.default=0.5}
\tikzset{vertex/.style={circle,thin,draw=black,fill=#1,inner sep=1pt}}
\tikzset{edge/.style={very thick,draw=#1!50!black}}
\tikzset{bgline/.style={gray,very thin}}
\tikzset{bgplaq/.style={fill=#1!15!white}}

\def\colors {{0.7/ 1.0/ 0.87/ 1}, {1.0/ 0.8/ 0.87/ 1}, {1.0/ 0.87/ 0.7/
     1}, {0.87/ 1.0/ 0.7/ 1}, {0.82/ 0.82/ 1.0/ 1}, {0.9/ 1.0/ 0.9/ 
    1}, {1.0/ 0.9/ 0.9/ 1}, {0.94/ 0.94/ 0.94/ 1}, {1.0/ 1.0/ 0.65/ 
    1}}
\globalcolorstrue
\foreach[count=\i from 0] \r/\g/\b/\a in \colors {\definecolor{c\i}{rgb}{\r,\g,\b}}
\definecolor{lightblue}{rgb}{0.5,0.5,1}
\colorlet{nw0}{green}\colorlet{nw1}{lightblue}\colorlet{nw3}{green}\colorlet{nw4}{lightblue}\colorlet{nw5}{lightblue}\colorlet{nw7}{green}
\colorlet{ne0}{lightblue}\colorlet{ne1}{red}\colorlet{ne2}{red}\colorlet{ne4}{lightblue}\colorlet{ne5}{green}\colorlet{ne7}{red}
\colorlet{s0}{green}\colorlet{s1}{red}\colorlet{s2}{green}\colorlet{s3}{red}
\def\rawpuzzle#1{
\foreach\p in {#1} {
\pgfmathtruncatemacro\c{{\p}[0]}
\pgfmathtruncatemacro\x{{\p}[1][1]}
\pgfmathtruncatemacro\y{{\p}[1][0]}
\pgfmathtruncatemacro\t{{\p}[2]}
\pgfmathtruncatemacro\xx{\x+\t}
\pgfmathtruncatemacro\yy{\y+\t}
\fill[bgplaq=c\c] (\xx,\yy) -- (\x+1,\y) -- (\x,\y+1);
\ifcase\c
\draw[edge=green] (\x+0.5,\y+0.5) -- (\xx,\y+0.5);
\or
\draw[edge=red] (\x+0.5,\y+0.5) -- (\x+0.5,\yy);
\or
\draw[edge=green] (\x+0.5,\y+0.5) -- (\xx,\y+0.5);
\draw[edge=red] (\xx,\y+0.5) -- (\x+0.5,\yy);
\or
\draw[edge=red] (\x+0.5,\y+0.5) -- (\x+0.5,\yy);
\draw[edge=green] (\x+0.5,\yy) -- (\xx,\y+0.5);
\or
\or
\draw[edge=green] (\x+0.5,\y+0.5) -- (\x+0.5,\yy);
\or
\draw[edge=red] (\x+0.5,\y+0.5) -- (\xx,\y+0.5);
\or
\draw[edge=red] (\x+0.5,\y+0.5) -- (\x+0.5,\yy);
\draw[edge=green] (\x+0.5,\y+0.5) -- (\xx,\y+0.5);
\or
\draw[edge=green] (\x+0.48,\yy) -- (\xx,\y+0.48);
\draw[edge=red] (\x+0.52,\yy) -- (\xx,\y+0.52);
\fi
}
}
\def\rawpuzzlevertices#1{
\foreach\p in {#1}{
\pgfmathtruncatemacro\c{{\p}[0]}
\pgfmathtruncatemacro\x{{\p}[1][1]}
\pgfmathtruncatemacro\y{{\p}[1][0]}
\pgfmathtruncatemacro\t{{\p}[2]}
\pgfmathtruncatemacro\l{\size-1-\x-\y}
\ifnum\t=0
\ifnum\x=0 \node[vertex=nw\c] at (0,\y+0.5) {};\fi
\ifnum\y=0 \node[vertex=ne\c] at (\x+0.5,0) {};\fi
\ifnum\l=0 \node[vertex=s\c] at (\x+0.5,\y+0.5) {};\fi
\fi
}
}
\def\rawdoublepuzzlevertices#1{
\foreach\p in {#1}{
\pgfmathtruncatemacro\c{{\p}[0]}
\pgfmathtruncatemacro\x{{\p}[1][1]}
\pgfmathtruncatemacro\y{{\p}[1][0]}
\pgfmathtruncatemacro\t{{\p}[2]}
\pgfmathtruncatemacro\lx{\size-1-\x}
\pgfmathtruncatemacro\ly{\size-1-\y}
\ifnum\t=0
\ifnum\x=0 \node[vertex=nw\c] at (0,\y+0.5) {};\fi
\ifnum\y=0 \node[vertex=ne\c] at (\x+0.5,0) {};\fi
\else
\ifnum\lx=0\node[vertex=nw\c] at (\size,\y+0.5) {};\fi
\ifnum\ly=0\node[vertex=ne\c] at (\x+0.5,\size) {};\fi
\fi
}
}
\def\puzzlescale{0.7}
\newif\ifstarttikz
\tikzset{puz/.style={yscale=-1.732,scale=0.707,rotate=45,scale=\puzzlescale,line cap=round}}
\newcommand\tile[1]{
\tikzifinpicture{\starttikzfalse}{\starttikztrue\begin{tikzpicture}}
\begin{scope}[puz]
\begin{scope}
\clip (0,0) -- (1,0) -- (0,1) -- cycle;
\rawpuzzle{#1}
\end{scope}
\draw[bgline] (0,0) -- (1,0) -- (0,1) -- cycle;
\end{scope}
\ifstarttikz\end{tikzpicture}\fi%
}
\newcommand\tilei[1]{
\tikzifinpicture{\starttikzfalse}{\starttikztrue\begin{tikzpicture}}
\begin{scope}[puz]
\begin{scope}
\clip (1,1) -- (1,0) -- (0,1) -- cycle;
\rawpuzzle{#1}
\end{scope}
\draw[bgline] (1,1) -- (1,0) -- (0,1) -- cycle;
\end{scope}
\ifstarttikz\end{tikzpicture}\fi%
}
\newcommand\dtile[1]{
\tikzifinpicture{\starttikzfalse}{\starttikztrue\begin{tikzpicture}}
\begin{scope}[puz]
\begin{scope}
\clip (0,0) -- (1,0) -- (1,1) -- (0,1) -- cycle;
\rawpuzzle{#1}
\end{scope}
\draw[bgline] (0,0) -- (1,0) -- (1,1) -- (0,1) -- cycle;
\end{scope}
\ifstarttikz\end{tikzpicture}\fi%
}
\newcommand\dtileb[1]{
\tikzifinpicture{\starttikzfalse}{\starttikztrue\begin{tikzpicture}}
\begin{scope}[puz]
\begin{scope}
\clip (0,1) -- (1,1) -- (2,0) -- (1,0) -- cycle;
\rawpuzzle{#1}
\end{scope}
\draw[bgline] (0,1) -- (1,1) -- (2,0) -- (1,0) -- cycle;
\end{scope}
\ifstarttikz\end{tikzpicture}\fi%
}
\newcommand\dtilec[1]{
\tikzifinpicture{\starttikzfalse}{\starttikztrue\begin{tikzpicture}}
\begin{scope}[puz]
\begin{scope}
\clip (0,1) -- (0,2) -- (1,1) -- (1,0) -- cycle;
\rawpuzzle{#1}
\end{scope}
\draw[bgline] (0,1) -- (0,2) -- (1,1) -- (1,0) -- cycle;
\end{scope}
\ifstarttikz\end{tikzpicture}\fi%
}
\newcommand\puzzle[2][]{%
\tikzifinpicture{\starttikzfalse}{\starttikztrue\begin{tikzpicture}}
\begin{scope}[puz]
\rawpuzzle{#2}
\draw (0,0) -- (\size,0) -- (0,\size) -- cycle;
\begin{scope}[bgline]
\foreach\x in {1,...,\size} \draw (\x,0) -- (\x,\size-\x);
\foreach\x in {1,...,\size} \draw (0,\x) -- (\size-\x,\x);
\foreach\x in {1,...,\size} \draw (\size-\x,0) -- (0,\size-\x);
\end{scope}
\rawpuzzlevertices{#2}
#1
\end{scope}
\ifstarttikz\end{tikzpicture}\fi%
}
\newcommand\equivpuzzle[2][]{
\tikzifinpicture{\starttikzfalse}{\starttikztrue\begin{tikzpicture}}
\begin{scope}[puz]
\rawpuzzle{#2}
\draw (0,0) -- (\size,0) -- (0,\size) -- cycle;
\begin{scope}[bgline]
\foreach\x in {1,...,\size} \draw (\x,0) -- (\x,\size-\x);
\foreach\x in {1,...,\size} \draw (0,\x) -- (\size-\x,\x);
\end{scope}
\rawpuzzlevertices{#2}
#1
\end{scope}
\ifstarttikz\end{tikzpicture}\fi%
}
\newcommand\doublepuzzle[2][]{%
\tikzifinpicture{\starttikzfalse}{\starttikztrue\begin{tikzpicture}}
\begin{scope}[puz]
\rawpuzzle{#2}
\draw (0,0) -- (\size,0) -- (\size,\size) -- (0,\size) -- cycle;
\begin{scope}[bgline]
\foreach\x in {1,...,\size} \draw (\x,0) -- (\x,\size);
\foreach\x in {1,...,\size} \draw (0,\x) -- (\size,\x);
\foreach\x in {0,...,\size} \draw (\size-\x,0) -- (0,\size-\x);
\foreach\x in {1,...,\size} \draw (\x,\size) -- (\size,\x);
\end{scope}
\rawdoublepuzzlevertices{#2}
#1
\end{scope}
\ifstarttikz\end{tikzpicture}\fi%
}
\newcommand\equivdoublepuzzle[2][]{%
\tikzifinpicture{\starttikzfalse}{\starttikztrue\begin{tikzpicture}}
\begin{scope}[puz]
\rawpuzzle{#2}
\draw (0,0) -- (\size,0) -- (\size,\size) -- (0,\size) -- cycle;
\begin{scope}[bgline]
\foreach\x in {1,...,\size} \draw (\x,0) -- (\x,\size);
\foreach\x in {1,...,\size} \draw (0,\x) -- (\size,\x);
\end{scope}
\rawdoublepuzzlevertices{#2}
#1
\end{scope}
\ifstarttikz\end{tikzpicture}\fi%
}
\def\anglesloz{{60,300,240,120}}
\newcommand\loz[3]{%
\begin{tikzpicture}[baseline=0]
\path (-0.5,0) coordinate (A);
\foreach[count=\i] \name/\ori in {#1}
{
\pgfmathsetmacro\a{\anglesloz[\i-1]}
\draw[\ori] (A) ++(\a+90:0.3) -- node[shift={(\a+90:0.25)}] {$\ss\name$} ++(\a:1);
\draw (A) -- ++(\a:1) coordinate (A);
}
\node[rotate=-60] at (-0.17,-0.2) {$\ss#2$};
\node[rotate=60] at (0.17,-0.2) {$\ss#3$};
\end{tikzpicture}%
}
\def\anglestriup{{60,300,180}}
\newcommand\triup[2]{%
\begin{tikzpicture}[baseline=0]
\path (-0.5,-0.3) coordinate (A);
\foreach[count=\i] \name/\ori in {#1}
{
\pgfmathsetmacro\a{\anglestriup[\i-1]}
\draw[\ori] (A) ++(\a+90:0.3) -- node[shift={(\a+90:0.25)}] {$\ss\name$} ++(\a:1);
\draw (A) -- ++(\a:1) coordinate (A);
}
\node {$\ss#2$};
\end{tikzpicture}%
}
\def\anglestridown{{0,240,120}}
\newcommand\tridown[2]{%
\begin{tikzpicture}[baseline=-0.3cm]
\path (-0.5,0.3) coordinate (A);
\foreach[count=\i] \name/\ori in {#1}
{
\pgfmathsetmacro\a{\anglestridown[\i-1]}
\draw[\ori] (A) ++(\a+90:0.3) -- node[shift={(\a+90:0.25)}] {$\ss\name$} ++(\a:1);
\draw (A) -- ++(\a:1) coordinate (A);
}
\node {$\ss#2$};
\end{tikzpicture}%
}
\long\def\sectionremove#1\section{\section}
\title[Littlewood--Richardson coefficients for Grothendieck polynomials]
{Littlewood--Richardson coefficients for Grothendieck polynomials from integrability}

\author{M.~Wheeler}
\address{Michael Wheeler, School of Mathematics and Statistics, University of Melbourne, Parkville, Victoria 3010, Australia}
\email{wheelerm@unimelb.edu.au}

\author{P.~Zinn-Justin}
\address{Paul Zinn-Justin, Laboratoire de Physique Th\'eorique et Hautes \'Energies, CNRS UMR 7589 and Universit\'e Pierre et Marie Curie (Paris 6), 4 place Jussieu, 75252 Paris cedex 05, France}
\email{pzinn@lpthe.jussieu.fr}

\begin{document}

\begin{abstract}
We study the Littlewood--Richardson coefficients of double Grothendieck polynomials indexed by Grassmannian permutations. Geometrically, these are the structure constants of the equivariant $K$-theory ring of Grassmannians. Representing the double Grothendieck polynomials as partition functions of an integrable vertex model, we use its Yang--Baxter equation to derive a series of product rules for the former polynomials and their duals. The Littlewood--Richardson coefficients that arise can all be expressed in terms of puzzles without gashes, which generalize previous puzzles obtained by Knutson--Tao and Vakil.
\end{abstract}

\maketitle

\section{Introduction}
\subsection{Background}
Grothendieck polynomials were introduced by Lascoux and Sch\"utzen\hyp{}berger \cite{LS}, as polynomial representatives of Schubert classes in the Grothendieck ring of the flag manifold. They are inhomogeneous, multivariable polynomials which generalize Schubert polynomials (the latter are recovered by extracting all lowest degree monomials) and are indexed by permutations. In this work we shall focus solely on the case of Grassmannian permutations, \ie\ permutations $\sigma$ with a unique descent $\sigma(k) > \sigma(k+1)$, when the Grothendieck polynomials become symmetric in their variables and are more naturally indexed by partitions.  We will continue to refer to them simply as {\it Grothendieck polynomials,} with the implicit understanding that this always refers to Grassmannian permutations, or equivalently to
the $K$-theory of the Grassmannian. 

In this paper we consider the {\it structure constants} of the ring of symmetric functions with respect to the Grothendieck polynomial basis. They are the unique expansion coefficients obtained by taking a product of two Grothendieck polynomials:
\begin{equation}
\label{structure}
G^{\lambda} G^{\mu}
=
\sum_{\nu}
c^{\lambda,\mu}_{\nu}
G^{\nu},
\qquad
c^{\lambda,\mu}_{\nu} \in \mathbb{Z}.
\end{equation}
Since Grassmannian Grothendieck polynomials are generalizations of Schur polynomials, the coefficients $c^{\lambda,\mu}_{\nu}$ can be regarded as $K$-theoretic analogues of the famous Littlewood--Richardson coefficients \cite{lit-ric,rob,tho,sch}. Indeed, in the special case $|\lambda| + |\mu| = |\nu|$, the $c^{\lambda,\mu}_{\nu}$ reproduce exactly the Littlewood--Richardson coefficients. For this reason, we shall refer to $c^{\lambda,\mu}_{\nu}$ as Littlewood--Richardson coefficients for the Grothendieck polynomials.

The first explicit rule for calculating $c^{\lambda,\mu}_{\nu}$ was obtained by Buch \cite{Buch}, essentially by extending the notion of Littlewood--Richardson tableaux \cite{lit-ric} to the case of {\it set-valued} tableaux. An important corollary of this combinatorial expression for $c^{\lambda,\mu}_{\nu}$ is the fact that the sum \eqref{structure} is finite, which is not otherwise obvious (in view of the inhomogeneity of $G^{\lambda}$).

Not long after this, Vakil \cite{Vakil} obtained an elegant combinatorial expression for $c^{\lambda,\mu}_{\nu}$ by introducing a $K$-theoretic version of Knutson--Tao puzzles \cite{KT,KTW}. The generalization of the Knutson--Tao puzzles was achieved by adding an extra (non-rotatable) {\it puzzle piece} to the set of existing pieces, which gives rise to a minus sign every time it appears in a puzzle. Subsequently, an equivariant version of these $K$-puzzles was conjectured by Knutson and Vakil \cite{CV}, and proved recently (modulo a correction to the original formula) by Pechenik and Yong \cite{PY}. 

In the present paper we shall also consider the equivariant $K$-theory of the Grassmannian, for which the corresponding polynomial representatives are the {\it double} Grothendieck polynomials \cite{LS}. One of our principal aims is to recover the $K$-puzzles of \cite{Vakil} using techniques from integrability, and then to generalize them to the equivariant case. It is important to point out that our puzzle expression for the equivariant Littlewood--Richardson coefficients (Theorem \ref{thm:equiv}) is {\it not} obviously the same as the one obtained in \cite{PY}, while giving the same answer; indeed, our formulation seems to be simpler, since our tiling rules are purely local, and we require fewer puzzles to compute coefficients than in \cite{PY}. In addition to this, we present new rules for the multiplication of {\it dual} double Grothendieck polynomials, and $K$-theoretic versions of the Molev--Sagan product rule \cite{mol-sag}.

It is worth noting that our two main theorems provide rules that are positive in the sense of \cite{AGM-Kpos}:
our Theorem~\ref{thm:equiv} displays positivity as in \cite[Cor.~5.3]{AGM-Kpos} (except we use Schubert classes, not opposite ones, so the secondary alphabet is reversed), and our Theorem~\ref{thm:equivdual} displays positivity as in \cite[Cor.~5.2]{AGM-Kpos}.

\subsection{Methodology}
In \cite{z-j}, a new approach for calculating (Schur) Littlewood--Richardson coefficients was developed. One of the key ideas in \cite{z-j} was to realise the product rule for the Schur polynomials in the setting of integrable vertex models. Namely, both sides of the product identity $s^{\lambda} s^{\mu} = \sum_{\nu} c^{\lambda,\mu}_{\nu} s^{\nu}$ were expressed as partition functions in the square-triangle-rhombus model \cite{d-g-nie}, with their equality arising as a consequence of the Yang--Baxter equation of the model. This led to a novel formulation of the puzzles that had previously been obtained in \cite{KT,KTW}, and also permitted a solution of the more complicated Molev--Sagan problem \cite{mol-sag}. In this work, we shall extend the methods of \cite{z-j} to the $K$-theoretic setting. There are four steps in this program:

{\bf 1.} It is first necessary to express the Grothendieck polynomials (and their duals) as partition functions in an appropriate vertex model. In the case of ordinary Grothendieck polynomials, such a construction was recently given by Motegi and Sakai \cite{ms1,ms2}, using an integrable five-vertex model. Here we extend the setup of \cite{ms1,ms2} to the equivariant case (double Grothendieck polynomials), which can be done without any further effort, since the equivariant parameters are nothing but the vertical spectral parameters of the lattice model.

{\bf 2.} At the second stage, we embed the five-vertex model used at step {\bf 1} in a rank-two solution of the Yang--Baxter equation. This is a technical requirement which allows our subsequent calculations to succeed (they would not, if we performed all calculations within the simpler, rank-one model), although it is difficult to motivate {\it a priori.} As a loose heuristic, one can view each Grothendieck polynomial appearing in \eqref{structure} as corresponding to one of the three possible embeddings of $\mathfrak{sl}(2)$ in $\mathfrak{sl}(3)$. While each polynomial can be constructed in its own right using only an $\mathfrak{sl}(2)$ vertex model, we must pass to $\mathfrak{sl}(3)$ in order to obtain $c^{\lambda,\mu}_{\nu}$ as a partition function. Recently we adopted similar techniques in calculating the structure constants of the Hall--Littlewood polynomials via a rank-two bosonic model \cite{wz-j2}.

{\bf 3.} Next, we write down a specific partition function in the model of point {\bf 2} and evaluate it in two equivalent ways, using the Yang--Baxter equation. This leads to the following equation (shown here schematically):
\begin{align}
\label{schem-master}
\sum_{\nu}
\begin{tikzpicture}[scale=0.4,baseline=0]
\draw (-5*0.25-4*0.5,-5*1.299) coordinate (A) -- ++(120:5) coordinate (B) -- ++(60:2*5) coordinate (C) -- ++(0:4) coordinate (D) -- ++(-60:5) coordinate (E) -- ++(-120:2*5) coordinate (F) -- cycle;
\draw (B) -- ++(0:4) coordinate (G) -- (D); \draw (F) -- (G);
\filldraw[fill=white!75!black] (A) -- (B) -- (G) -- (F) -- cycle;
\filldraw[fill=white!90!black] (G) ++(60:5) -- ++(-60:5) -- (F) -- (G) -- cycle;
\draw[ultra thick,arrow=0.5]  (A) -- node[below left] {$\ss\lambda$} (B);
\draw[ultra thick,arrow=0.5] (B) -- node[above left] {$\ss\mu$} ++(60:5) coordinate (H);
\draw[ultra thick,arrow=0.5] (H) ++(0:4) -- node[above left] {$\ss\nu$} (D);
\draw[ultra thick,arrow=0.5] (F) ++(60:5) -- node[below right] {$\ss\tilde\mu$} (E);
\draw[ultra thick,arrow=0.5] (E) -- node[above right] {$\ss\tilde\lambda$} (D);
\draw[dotted] (B) ++(60:5) -- ++(0:4) -- ++(-60:5);
\path (D) ++(-90:4) node {$c^{\nu,\tilde\lambda}_{\lambda,\tilde\mu}$}; 
\path (C) ++(-75:2) node {$G_\nu$}; 
\path (C) ++(240:5)  ++(-75:2) node {$G^\mu$}; 
\end{tikzpicture}
=
\sum_{\tilde\nu}
\begin{tikzpicture}[scale=-0.4,baseline=0]
\draw (-5*0.25-4*0.5,-5*1.299) coordinate (A) -- ++(120:5) coordinate (B) -- ++(60:2*5) coordinate (C) -- ++(0:4) coordinate (D) -- ++(-60:5) coordinate (E) -- ++(-120:2*5) coordinate (F) -- cycle;
\draw (B) -- ++(0:4) coordinate (G) -- (D); \draw (F) -- (G);
\filldraw[fill=white!75!black] (A) -- (B) -- (G) -- (F) -- cycle;
\filldraw[fill=white!90!black] (G) ++(60:5) -- ++(-60:5) -- (F) -- (G) -- cycle;
\draw[ultra thick,arrow=0.5]  (B) -- node[above right] {$\ss\tilde\lambda$} (A);
\draw[ultra thick,arrow=0.5] (C) ++(-120:5) -- node[below right] {$\ss\tilde\mu$} (B);
\draw[ultra thick,arrow=0.5] (D) -- node[below right] {$\ss\tilde\nu$} ++(-120:5);
\draw[ultra thick,arrow=0.5] (E) -- node[above left] {$\ss\mu$} ++(-120:5);
\draw[ultra thick,arrow=0.5] (D) -- node[below left] {$\ss\lambda$} (E);
\draw[dotted] (B) ++(60:5) -- ++(0:4) -- ++(-60:5);
\path (D) ++(-90:4) node {$c^{\mu,\tilde\lambda}_{\lambda,\tilde\nu}$}; 
\path (C) ++(-75:2) node {$G^{\tilde\nu}$}; 
\path (C) ++(240:5)  ++(-75:2) node {$G_{\tilde\mu}$}; 
\end{tikzpicture}
\end{align}
The boundary conditions of this partition function are chosen in such a way that they encode four partitions $\lambda,\mu,\tilde\lambda,\tilde\mu$ (as shown), while a sum is taken on all internal lattice sites. It is easy to show that certain regions of the partition function (those shown in grey) are in fact ``frozen'', meaning that there are only three non-trivial regions on either side of the identity. One of these regions can be recognized to produce a double Grothendieck polynomial, another a dual polynomial, while the third region is a ``lozenge'' bordered by four partitions. The resulting identity, when written algebraically, is the {\it master equation} \eqref{eq:master} (note that we have written the partitions $\tilde\lambda$ and $\tilde\mu$ in the opposite orientation to their appearance in \eqref{eq:master}, to avoid a slight notational difficulty which we will only address later).

{\bf 4.} Finally, it is possible to specialize \eqref{eq:master} in a number of ways, leading to various product rules. The freedom to specialize comes from the four partitions $\lambda,\mu,\tilde\lambda,\tilde\mu$, of which any subset can be chosen to be $\varnothing$ (the empty partition), as well as the sets of spectral parameters running through the lattice (not shown in the above figure), which can be identified. By exhausting all possible specializations, we arrive at a comprehensive list of product rules which replicates the known puzzles of \cite{Vakil}, while producing many new ones.

\subsection{Layout of paper}

In Section \ref{sec:results} we recall the recursive definition of double Grothendieck polynomials via the action of Demazure divided difference operators. We define equivariant puzzles as certain tilings of the triangular lattice by equilateral triangles, before listing each of the theorems proved in this work. In Section \ref{sec:model} we review the construction of (double) Grothendieck polynomials using an integrable five-vertex model (a trivial degeneration of the stochastic six-vertex model), before emedding it in a more general eleven-vertex model. We give the graphical representation of the Yang--Baxter equation for the latter model, which is ultimately used in Section \ref{sec:proofs} to derive the master identity \eqref{eq:master}. Having derived \eqref{eq:master}, we then take appropriate specializations of it to prove each of the theorems in Section \ref{sec:results}.

\subsection{Young diagrams and other notation}

Fix two integers $0\leq k\leq n$. We will consider Young diagrams contained in the $k \times (n-k)$ rectangle, and encode them by their {\it frame.} The frame of a Young diagram $\lambda$ is the subset of $\{1,\dots,n\}$ obtained by segmenting the border of $\lambda$ into up/right steps, and recording an $i$ if the $i$-th step is up, for all $1 \leq i \leq n$. For example, for $k=5$, $n=10$ and $\lambda$ given by
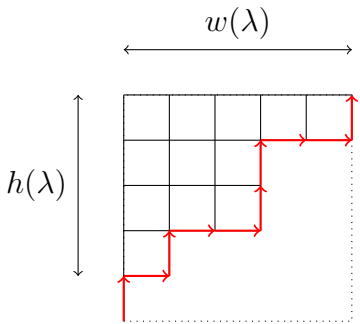
\begin{center}
\begin{tikzpicture}[scale=0.6]
\draw[<->] (-1,0) -- node[left] {$h(\lambda)$} (-1,4);
\draw[<->] (0,5) -- node[above] {$w(\lambda)$} (5,5);
\draw[dotted] (0,-1) -- (5,-1) -- (5,4) -- (0,4) -- (0,-1);
\draw (0,4) -- (5,4);
\draw (0,3) -- (5,3);
\draw (0,2) -- (3,2);
\draw (0,1) -- (3,1);
\draw (0,0) -- (1,0);
\draw (0,-1) -- (0,4);
\draw (1,0) -- (1,4);
\draw (2,1) -- (2,4);
\draw (3,1) -- (3,4);
\draw (4,3) -- (4,4);
\draw (5,3) -- (5,4);
\draw[thick,red,->] (0,-1) -- (0,0);
\draw[thick,red,->] (0,0) -- (1,0);
\draw[thick,red,->] (1,0) -- (1,1);
\draw[thick,red,->] (1,1) -- (2,1);
\draw[thick,red,->] (2,1) -- (3,1);
\draw[thick,red,->] (3,1) -- (3,2);
\draw[thick,red,->] (3,2) -- (3,3);
\draw[thick,red,->] (3,3) -- (4,3);
\draw[thick,red,->] (4,3) -- (5,3);
\draw[thick,red,->] (5,3) -- (5,4);
\end{tikzpicture}
\end{center}
we obtain the frame $\{1,3,6,7,10\}$ by tracing the boundary as indicated, and recording up steps. The frame associated to $\lambda$ should not be confused with the {\it partition} associated to $\lambda$ (lengths of rows of $\lambda$), a concept which will not be used in the paper. We shall write $i \in \lambda$ to indicate that $i$ appears in the frame of $\lambda$. We define the {\it width} $w(\lambda)$ ({\it height} $h(\lambda)$) of a Young diagram $\lambda$ as the number of boxes in its first row (column), respectively. In the example above, $w(\lambda) = 5$ and $h(\lambda) = 4$.

Let $\lambda$ and $\mu$ be two Young diagrams contained in the $k \times (n-k)$ rectangle, with frames $\{\ell_1,\dots,\ell_k\}$ and $\{m_1,\dots,m_k\}$, respectively. We write $\lambda \vartriangleright \mu$ if $0 \leq \ell_i-m_i \leq 1$ for all $1 \leq i \leq k$, and $(\ell_i-m_i = 1 \implies \ell_{i+1} = m_{i+1})$ for all $1\leq i \leq k-1$. In more traditional terminology, this is equivalent to saying that the skew diagram $\lambda-\mu$ forms a horizontal {\it and} vertical strip. Returning to the previous example,  
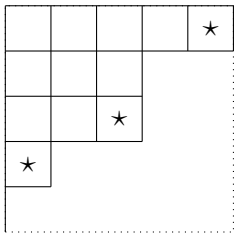
\begin{center}
\begin{tikzpicture}[scale=0.6]
\draw[dotted] (0,-1) -- (5,-1) -- (5,4) -- (0,4) -- (0,-1);
\draw (0,4) -- (5,4);
\draw (0,3) -- (5,3);
\draw (0,2) -- (3,2);
\draw (0,1) -- (3,1);
\draw (0,0) -- (1,0);
\draw (0,-1) -- (0,4);
\draw (1,0) -- (1,4);
\draw (2,1) -- (2,4);
\draw (3,1) -- (3,4);
\draw (4,3) -- (4,4);
\draw (5,3) -- (5,4);
\node[text centered] at (0.5,0.5) {$\star$};
\node[text centered] at (2.5,1.5) {$\star$};
\node[text centered] at (4.5,3.5) {$\star$};
\end{tikzpicture}
\end{center}
deleting any or all of the marked boxes in $\lambda$ produces a Young diagram $\mu$ such that $\lambda \vartriangleright \mu$.

We abbreviate sets of variables by a letter with a bar placed over it: $\x = (x_1,\dots,x_k)$. Ordered sets are specified by placing an arrow over a letter: $\y = (y_1,\dots,y_n)$ and $\yy = (y_n,\dots,y_1)$. In all such cases, the cardinality of the set will be clear from context.

\section{Main results}
\label{sec:results}

This section contains an overview of our results, with all proofs deferred to Section \ref{sec:proofs}. Our numbering system for the theorems is as follows: Theorem $m$ is a product rule for Grothendieck polynomials $G^{\lambda}$, Theorem $m'$ is the accompanying rule for dual Grothendieck polynomials $G_{\lambda}$, and Theorem $m''$ is a Molev--Sagan type product rule for $G_{\lambda}$. The case $m=1$ applies to ordinary Grothendieck polynomials (without a secondary alphabet), $m=2$ applies to double Grothendieck polynomials, and $m=3$ applies to double Grothendieck polynomials with a reversed secondary alphabet. We have illustrated most theorems by explicit examples; these were generated by the website \cite{puzzle} and can be viewed {\it in situ} by clicking on the hyperlinks provided. 

\subsection{Grothendieck polynomials}
\label{ssec:groth}

To each Young diagram $\lambda$ in the $k\times (n-k)$ rectangle we associate
the {\em double Grothendieck polynomial}\/ $G^\lambda(\x;\y)$. It
can for example be defined inductively \cite{LS} by letting
\begin{align}
\label{ind-0}
G^{k\times(n-k)}(\x;\y)=G^{\{n-k+1,\ldots,n\}}(\x;\y)=\prod_{i=1}^k \prod_{j=1}^{n-k} \left(1-\frac{x_i}{y_j}\right),
\end{align}
which defines $G^\lambda(\x;\y)$ for the largest possible Young diagram contained in the rectangle, and
\begin{align}
\label{ind-i}
G^\mu(\x;\y)=D_i G^\lambda(\x;\y)
\end{align}
if the frame of $\mu$ can be obtained from that of $\lambda$ by replacing $i+1$ with $i$.
Here $D_i$ is the Demazure divided difference operator acting on the variables $\y$:
\begin{align*}
D_i f(\y)=\frac{y_i f(\ldots,y_i,y_{i+1},\ldots)-y_{i+1} f(\ldots,y_{i+1},y_i,\ldots)}{y_i-y_{i+1}}.
\end{align*}
$G^\lambda(\x;\y)$ is easily shown to be a symmetric polynomial in the variables $\x$, and a polynomial in the variables $\y^{-1}=(y_1^{-1},\dots,y_n^{-1})$. When there is no risk of confusion, we sometimes write $G^\lambda \equiv G^\lambda(\x;\y)$.

These polynomials have the following geometric interpretation.
Consider the Grassmannian
$Gr(k,n)$. It has a natural $GL(n)$ action, and in particular the Cartan
torus $T$ acts on it, hence we can define its equivariant $K$-theory ring $K_T(Gr(k,n))$.
The latter is a quotient of $\mathbb Z[\x{}^{\pm1},\y{}^{\pm1}]^{\mathcal S_k}$ 
(where $\mathcal S_k$ acts by permutation of the $\x$;
the elementary symmetric polynomials of the $\x$ are classes of exterior powers of
the tautological vector bundle, whereas the $\y$ are equivariant parameters)
by the ideal of polynomials that vanish when the $\x$ are specialized
to a subset of the $\y$. Then $G^\lambda$ is a representative
of the $K_T$-class 
of the (structure sheaf of) the {\em Schubert variety}\/ $X^\lambda$
associated to the Young diagram $\lambda$. Here the convention is that
the number of boxes of $\lambda$, denoted $|\lambda|$, is the codimension
of $X^\lambda$.


We can also define {\em dual double Grothendieck polynomials}\/ $G_\lambda(\x;\y)$ to be
\begin{align*}
G_\lambda(\x;\y)=\prod_{i=1}^k x_i \prod_{i\in\lambda} y_i^{-1}
\,
G^{\lambda^\ast}(\x;\yy)
\end{align*}
where $\lambda^\ast$ is obtained by rotating by 180 degrees and complementing the Young
diagram $\lambda$, or equivalently whose frame is $n+1$ minus the frame of $\lambda$, and we recall that
$\yy=(y_n,\ldots,y_1)$. Geometrically,
$G_\lambda$ is a representative of the class of the ideal sheaf
of $\bigcup_{\mu\subsetneq\lambda}X_\mu$ inside $X_\lambda$, where the $X_\lambda$ are
opposite Schubert varieties (obtained from $X^{\lambda^\ast}$ by acting with the longest
element of the Weyl group;
note that $|\lambda|$ is
the dimension of $X_\lambda$, rather than its codimension).

It is known that the $G^\lambda$ and $G_\mu$ form dual bases of $K_T(Gr(k,n))$ (see \eg~\cite{AGM-Kpos}):
\begin{align*}
\left< G^\lambda G_\mu \right> = \delta^\lambda_\mu
\end{align*}
where $\left<\cdot\right>$ denotes pushforward to $K_T(\cdot)\cong \mathbb Z[\y{}^{\pm1}]$.

\subsection{Puzzles}
Consider the following triangular tiles:
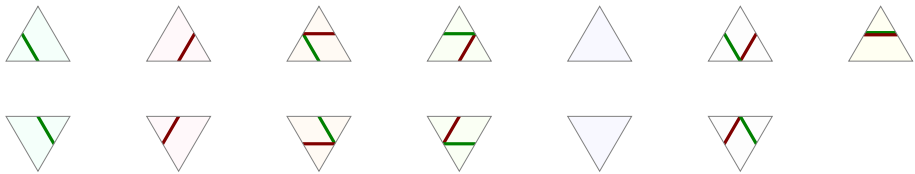
\begin{center}
\begin{tikzpicture}
\matrix[column sep=1cm,cells={scale=1.2}]
{
\tile{{0,{0,0},0}}
&
\tile{{1,{0,0},0}}
&
\tile{{2,{0,0},0}}
&
\tile{{3,{0,0},0}}
&
\tile{{4,{0,0},0}}
&
\tile{{7,{0,0},0}}
&
\tile{{8,{0,0},0}}
\\
\tilei{{0,{0,0},1}}
&
\tilei{{1,{0,0},1}}
&
\tilei{{2,{0,0},1}}
&
\tilei{{3,{0,0},1}}
&
\tilei{{4,{0,0},1}}
&
\tilei{{7,{0,0},1}}
&
\\
};
\end{tikzpicture}
\end{center}
The tile in the last column is called the $K$-tile, while the two tiles in the previous column are called equivariant tiles.

An {\em equivariant puzzle}\/
is a filling of a domain of the triangular lattice with the tiles above,
with the rule that red and green lines must be continuous across edges of neighboring triangles (\ie red and green lines can only start/end at the boundary of the domain).
An (ordinary) puzzle is an equivariant puzzle in which the equivariant tiles are excluded.\footnote{Note that in all our puzzles, the $K$-tile is allowed,
in contradistinction with the puzzles of \cite{KT,KTW}, which describe
the cohomology of $Gr(k,n)$. 
}

The {\em state}\/ of an edge of the domain for a particular puzzle is the list of colors of lines going through that edge; it can be empty, red, green or both.

We shall consider three types of domains:
\begin{enumerate}
\item Equilateral triangles pointing up.
\item Equilateral triangles pointing down.
\item Lozenges with short horizontal diagonal.
\end{enumerate}

In each case, we impose the following {\em boundary conditions}.
Horizontal boundary edges can have red lines passing through, or green lines, but not neither or both.
Similarly, 60/240 degree boundary edges can exist in two states, green or empty;
and 120/300 degree boundary edges can be either red or empty. Here is an example of a puzzle for each domain:
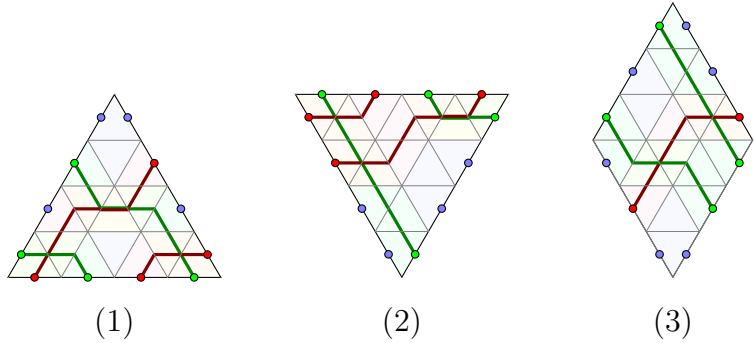
\begin{center}
\def\size{4}
\puzzle[\node at (2.5,2.5) {(1)};]{{4, {0, 0}, 0}, {4, {0, 0}, 1}, {1, {0, 1}, 0}, {3, {0, 1}, 
  1}, {0, {0, 2}, 0}, {2, {0, 2}, 1}, {2, {0, 3}, 0}, {0, {1, 0}, 
  0}, {2, {1, 0}, 1}, {8, {1, 1}, 0}, {4, {1, 1}, 1}, {1, {1, 2}, 
  0}, {1, {2, 0}, 0}, {3, {2, 0}, 1}, {0, {2, 1}, 0}, {3, {3, 0}, 0}}
\qquad
\begin{tikzpicture}[scale=-1]
\puzzle[\node at (-0.5,-0.5) {(2)};]{{0, {0, 0}, 0}, {0, {0, 0}, 1}, {0, {0, 1}, 0}, {2, {0, 1}, 
  1}, {2, {0, 2}, 0}, {2, {0, 2}, 1}, {2, {0, 3}, 0}, {4, {1, 0}, 
  0}, {4, {1, 0}, 1}, {1, {1, 1}, 0}, {1, {1, 1}, 1}, {1, {1, 2}, 
  0}, {4, {2, 0}, 0}, {8, {2, 0}, 1}, {2, {2, 1}, 0}, {3, {3, 0}, 0}}
\end{tikzpicture}
\qquad
\def\size{3}
\doublepuzzle[\node at (3.5,3.5) {(3)};]{{0, {0, 0}, 0}, {0, {0, 0}, 1}, {0, {0, 1}, 0}, {2, {0, 1}, 
  1}, {2, {0, 2}, 0}, {0, {0, 2}, 1}, {4, {1, 0}, 0}, {4, {1, 0}, 
  1}, {1, {1, 1}, 0}, {3, {1, 1}, 1}, {0, {1, 2}, 0}, {0, {1, 2}, 
  1}, {0, {2, 0}, 0}, {0, {2, 0}, 1}, {3, {2, 1}, 0}, {1, {2, 1}, 
  1}, {4, {2, 2}, 0}, {4, {2, 2}, 1}}
\end{center}
in which we have marked the states of boundary edges using small circles (empty being represented by
a 
blue color).

\tinyboxes
Call $n$ the {\em size}\/ of the puzzle, \ie the length in lattice units of any side of its boundary.
A side of the boundary, once we fix an orientation of it (depicted by an arrow), 
forms a binary string of length $n$ in the two allowed edge states; 
equivalently, it encodes a subset of $\{1,\ldots,n\}$ where we record the green edges on the two orientations
of sides where they are allowed, and blue edges on the third orientation:
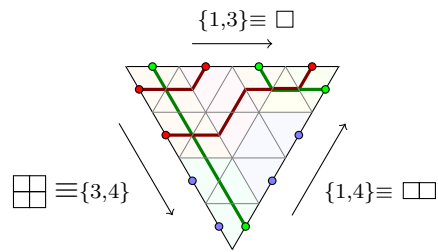
\begin{center}
\def\size{4}
\begin{tikzpicture}[scale=-1]
\def\u{\draw[->] (3.5,-0.75) -- (1.5,-0.75);
\node at (3.25,-2) {$\equiv\ss \{3,4\}$}; \node[tableau] at (4.75,-3.5) {&\\&\\};
\draw[->] (3,1.5) -- (1.5,3);
\node at (2.5,2.5) {$\ss \{1,3\}\equiv$}; \node[tableau] at (1.6,3.25) {\\};
\draw[->] (-0.75,1.5) -- (-0.75,3.5);
\node at (-1.75,3) {$\ss \{1,4\}\equiv$}; \node[tableau] at (-2.65,3.75) {&\\};}
\puzzle[\u]{{0, {0, 0}, 0}, {0, {0, 0}, 1}, {0, {0, 1}, 0}, {2, {0, 1}, 
  1}, {2, {0, 2}, 0}, {2, {0, 2}, 1}, {2, {0, 3}, 0}, {4, {1, 0}, 
  0}, {4, {1, 0}, 1}, {1, {1, 1}, 0}, {1, {1, 1}, 1}, {1, {1, 2}, 
  0}, {4, {2, 0}, 0}, {8, {2, 0}, 1}, {2, {2, 1}, 0}, {3, {3, 0}, 0}}
\end{tikzpicture}
\end{center}
Observe that these subsets all have the same cardinality, say $k$, so that we can associate to each
sequence a Young diagram inside the $k\times(n-k)$ rectangle, as indicated on the picture.

\subsection{Non-equivariant theorems}
Fix two integers $0\leq k\leq n$.
The first two theorems are special cases of our results corresponding to non-equivariant $K$-theory $K(Gr(k,n))$
(which formally amounts to setting all $y_i$ to $1$).
\begin{thm}[Vakil \cite{Vakil}]\label{thm:nonequiv}
Denote $G^\lambda(\x)=G^\lambda(\x;1,\ldots,1)$.
Given two Young diagrams $\lambda$, $\mu$ inside the $k\times(n-k)$ rectangle,
write, as an identity in $K(Gr(k,n))$,
\begin{align*}
G^\lambda(\x)G^\mu(\x)=\sum_\nu c^{\lambda,\mu}_{\,\,\nu} G^\nu(\x)
\end{align*}
where $\nu$ also runs over Young diagrams inside the $k\times(n-k)$ rectangle.
Then $c^{\lambda,\mu}_{\,\,\nu}$ is $(-1)^{|\nu|-|\lambda|-|\mu|}$ 
times the number of puzzles in a triangle pointing up,
with sides labelled $\lambda,\mu,\nu$ all oriented to the right:
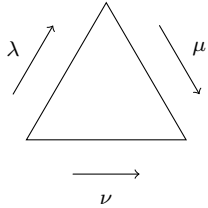
\begin{center}
\def\size{3}
\begin{tikzpicture}[puz]
\draw (0,0) -- (\size,0) -- (0,\size) -- cycle;
\draw[<-] (2.75,-0.75) -- (1.25,-0.75);
\node at (2.25,-1.25) {$\ss\mu$};
\draw[<-] (2.5,1.25) -- (1.25,2.5);
\node at (2.15,2.15) {$\ss\nu$};
\draw[<-] (-0.75,1.25) -- (-0.75,2.75);
\node at (-1.25,2.25) {$\ss\lambda$};
\end{tikzpicture}
\end{center}
\end{thm}
\begin{rmk}
One can actually consider the expansion above for the polynomials themselves (and not their classes in $K(Gr(k,n))$),
the only difference being that the sum over $\nu$ is now over all partitions with at most $k$ rows.
Then the theorem still applies as written, because of a {\em stability}\/ property of Grothendieck polynomials
and of Littlewood--Richardson coefficients that will be discussed in Sect.~\ref{sec:stab}.
\end{rmk}

\begin{thmd}\label{thm:nonequivdual}
Denote $G_\lambda(\x)=G_\lambda(\x;1,\ldots,1)$.
Write, as an identity in $K(Gr(k,n))$,
\begin{align*}
G_\lambda(\x)G_\mu(\x)=\sum_\nu c_{\lambda,\mu}^{\,\,\nu} G_\nu(\x).
\end{align*}
Given three Young diagrams $\lambda$, $\mu$, $\nu$ inside the $k\times(n-k)$ rectangle,
$c_{\lambda,\mu}^{\,\,\nu}$ is 
$(-1)^{|\lambda|+|\mu|-|\nu|-k(n-k)}$ 
times the number of puzzles in a triangle pointing down,
with sides labelled $\lambda,\mu,\nu$ all oriented to the right:
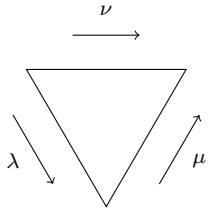
\begin{center}
\def\size{3}
\begin{tikzpicture}[puz,scale=-1]
\draw (0,0) -- (\size,0) -- (0,\size) -- cycle;
\draw[->] (2.75,-0.75) -- (1.25,-0.75);
\node at (2.25,-1.25) {$\ss\lambda$};
\draw[->] (2.5,1.25) -- (1.25,2.5);
\node at (2.15,2.15) {$\ss\nu$};
\draw[->] (-0.75,1.25) -- (-0.75,2.75);
\node at (-1.25,2.25) {$\ss\mu$};
\end{tikzpicture}
\end{center}
\end{thmd}

\begin{rmk}
We can equivalently formulate the theorem in terms of $\tilde G^\lambda(\x):=\prod_{i=1}^k x_i\, G^\lambda(\x)$ as follows: the
coefficient of the expansion of $\tilde G^\lambda(\x)\tilde G^\mu(\x)$ in the $\tilde G^\nu(\x)$ is given
by $(-1)^{|\nu|-|\lambda|-|\mu|}$ times the number of puzzles with boundary conditions as above, except all orientations are reversed. It is tempting
at this stage to rotate the triangle 180 degrees. However the price to pay is that one tile is not invariant under this transformation: the $K$-tile
\tile{{8,{0,0},0}}
gets replaced with the inverted $K$-tile
\tilei{{8,{0,0},1}}.
In the end we obtain boundary conditions which look similar to the ones
in Theorem~\ref{thm:nonequiv},
\begin{center}
\def\size{3}
\begin{tikzpicture}[puz]
\draw (0,0) -- (\size,0) -- (0,\size) -- cycle;
\draw[<-] (2.75,-0.75) -- (1.25,-0.75);
\node at (2.25,-1.25) {$\ss\mu$};
\draw[<-] (2.5,1.25) -- (1.25,2.5);
\node at (2.15,2.15) {$\ss\nu$};
\draw[<-] (-0.75,1.25) -- (-0.75,2.75);
\node at (-1.25,2.25) {$\ss\lambda$};
\node at (1,1) {$\circlearrowleft$};
\end{tikzpicture}
\end{center}
but in which the symbol $\circlearrowleft$ is meant to remind us of the
transformation of the $K$-tile.

One advantage of this transformation is that it makes it clear that in the cohomology limit in which the $K$-tile is forbidden,
Theorems~\ref{thm:nonequiv} and \ref{thm:nonequivdual} become identical and in fact correspond to the standard Littlewood--Richardson rule, as formulated in \cite{KT,KTW}.

However, in what follows, we shall not consider this 180 degree rotation and inverted $K$-tile; the reason is that
the generalization to the equivariant setting would become rather confusing,
since the weights we shall assign to the various
tiles are not 180 degree rotationally invariant.
\end{rmk}

\newcommand\stableau[1]{\tikz[baseline=0]{
\setlength{\cellsize}{4.5pt}
\node[tableau]{#1};
\path 
(current bounding box.north west) coordinate (A) 
([yshift=-\cellsize]current bounding box.north east) coordinate (B)
([yshift=\cellsize]current bounding box.south east) coordinate (C);
\pgfresetboundingbox
\useasboundingbox (A) (B) (C);
}}

\begin{exlink}{http://www.lpthe.jussieu.fr/~pzinn/puzzles/?height=2&width=2&y1=2,2&y1comp&y2=2,1&y2comp&Kinv&mask=35&view=9&process}
Set $n=4$, $k=2$. Consider $G_{\stableau{&\\&\\}}(\x)G_{\stableau{&\\\\}}(\x)$.
Using the explicit formulae
\begin{align*}
G^{\varnothing}(\x)&=1,
\qquad
G^{\stableau{\\}}(\x)=1-\prod_{i=1}^kx_i,
\qquad
G^{\stableau{&\\}}(\x)=
1+
\left(-(k+1)+\sum_{i=1}^k x_i\right)
\prod_{i=1}^k x_i,
\\
G^{\stableau{\\\\}}(\x)&=
1+
\left(k-1-\sum_{i=1}^k x_i^{-1}\right)
\prod_{i=1}^k x_i,
\qquad
G^{\stableau{&\\\\}}(\x)=1-\left(
\prod_{i=1}^k x_i
+\sum_{i=1}^k x_i^{-1}
-\sum_{i=1}^k x_i
\right)
\prod_{i=1}^kx_i,
\end{align*}
and $G_\lambda(\x)=\prod_{i=1}^k x_i\ G^{\lambda^\ast}(\x)$,
we find
\begin{align*}
G_{\stableau{&\\&\\}}(\x)G_{\stableau{&\\\\}}(\x)
=
G_{\stableau{&\\\\}}(\x)
-
G_{\stableau{&\\}}(\x)
-
G_{\stableau{\\\\}}(\x)
+
G_{\stableau{\\}}(\x).
\end{align*}
There are indeed four puzzles: 
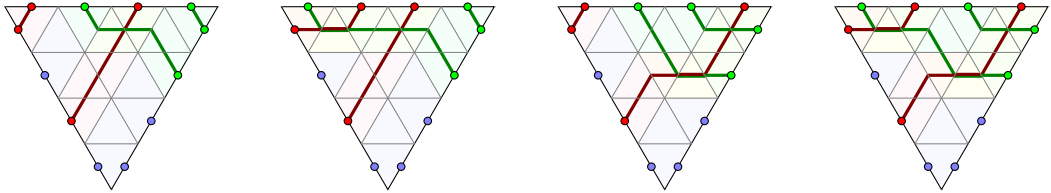
\begin{center}
\def\size{4}
\tikz[scale=-1]{\puzzle{{4,{0,0},0},{4,{0,0},1},{1,{0,1},0},{1,{0,1},1},{4,{0,2},0},{4,{0,2},1},{1,{0,3},0},{4,{1,0},0},{4,{1,0},1},{1,{1,1},0},{3,{1,1},1},{0,{1,2},0},{0,{2,0},0},{0,{2,0},1},{3,{2,1},0},{0,{3,0},0}}}\qquad
\tikz[scale=-1]{\puzzle{{4,{0,0},0},{4,{0,0},1},{1,{0,1},0},{1,{0,1},1},{4,{0,2},0},{8,{0,2},1},{2,{0,3},0},{4,{1,0},0},{4,{1,0},1},{1,{1,1},0},{3,{1,1},1},{3,{1,2},0},{0,{2,0},0},{0,{2,0},1},{3,{2,1},0},{0,{3,0},0}}}\qquad
\tikz[scale=-1]{\puzzle{{4,{0,0},0},{4,{0,0},1},{1,{0,1},0},{1,{0,1},1},{4,{0,2},0},{4,{0,2},1},{1,{0,3},0},{4,{1,0},0},{8,{1,0},1},{2,{1,1},0},{0,{1,1},1},{0,{1,2},0},{3,{2,0},0},{3,{2,0},1},{0,{2,1},0},{3,{3,0},0}}}\qquad
\tikz[scale=-1]{\puzzle{{4,{0,0},0},{4,{0,0},1},{1,{0,1},0},{1,{0,1},1},{4,{0,2},0},{8,{0,2},1},{2,{0,3},0},{4,{1,0},0},{8,{1,0},1},{2,{1,1},0},{0,{1,1},1},{3,{1,2},0},{3,{2,0},0},{3,{2,0},1},{0,{2,1},0},{3,{3,0},0}}}\qquad
\end{center}
\end{exlink}

\subsection{Equivariant theorems}
Now we provide equivariant versions of these theorems. This requires some extra setup. We glue together adjacent triangles that share a horizontal edge to form elementary lozenges. 
Each such lozenge acquires a weight given by the following table: 
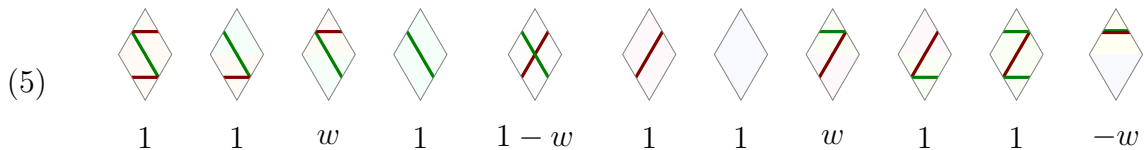
\begin{equation}\label{eq:wei}
\vcenter{\hbox{\begin{tikzpicture}
\matrix[column sep=0.5cm,row sep=0.2cm]
{
\dtile{{2,{0,0},0},{2,{0,0},1}}
&
\dtile{{0,{0,0},0},{2,{0,0},1}}
&
\dtile{{2,{0,0},0},{0,{0,0},1}}
&
\dtile{{0,{0,0},0},{0,{0,0},1}}
&
\dtile{{7,{0,0},0},{7,{0,0},1}}
&
\dtile{{1,{0,0},0},{1,{0,0},1}}
&
\dtile{{4,{0,0},0},{4,{0,0},1}}
&
\dtile{{3,{0,0},0},{1,{0,0},1}}
&
\dtile{{1,{0,0},0},{3,{0,0},1}}
&
\dtile{{3,{0,0},0},{3,{0,0},1}}
&
\dtile{{8,{0,0},0},{4,{0,0},1}}
\\
\node{$1$};
&
\node{$1$};
&
\node{$w$};
&
\node{$1$};
&
\node{$1-w$};
&
\node{$1$};
&
\node{$1$};
&
\node{$w$};
&
\node{$1$};
&
\node{$1$};
&
\node{$-w$};
\\
};
\end{tikzpicture}}}
\end{equation}
Here $w$ is a ratio of equivariant parameters, depending on the location of the lozenge, as explained
in the statement of the theorems below.
The total weight of an equivariant puzzle is the product of the weights of all its elementary lozenges.

We should remark that, as in the non-equivariant case, all expansions written below are considered as identities in $K_T(Gr(k,n))$; we postpone the discussion of stability to Sect.~\ref{sec:stab}.

\begin{thm}\label{thm:equiv}
Write
\begin{align*}
G^\lambda(\x;\y)G^\mu(\x;\y)=\sum_\nu c^{\lambda,\mu}_{\,\,\nu}(\y) G^\nu(\x;\y).
\end{align*}
Given three Young diagrams $\lambda$, $\mu$, $\nu$ inside the $k\times(n-k)$ rectangle,
$c^{\lambda,\mu}_{\,\,\nu}(\y)$ is the sum of weights of equivariant
puzzles in a triangle pointing up,
with sides labelled $\lambda,\mu,\nu$ all oriented to the right, the parameter in the weights
being given by $w=y_i/y_j$ for an elementary lozenge whose coordinates $(i,j)$ are defined as follows:
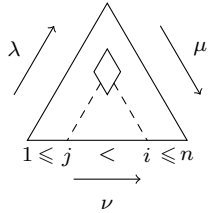
\begin{center}
\def\size{3}
\begin{tikzpicture}[puz]
\draw (0,0) -- (\size,0) -- (0,\size) -- cycle;
\draw (0.5,0.5) -- (0.5,1) -- (1,1) -- (1,0.5) -- cycle;
\node[below=-0.5mm] at (0,3) {$\ss 1$};
\node[below=-0.5mm] at (0.35,2.65) {$\ss \leq$};
\node[below=-0.5mm] at (1.5,1.5) {$\ss <$};
\node[below=-0.5mm] at (2.65,0.35) {$\ss \leq$};
\node[below=-0.1mm] at (3,0) {$\ss n$};
\draw[dashed] (0.75,1) -- (0.75,2.25) node[below=-0.5mm] {$\ss j$};
\draw[dashed] (1,0.75) -- (2.25,0.75) node[below=-0.5mm] {$\ss i$};
\draw[<-] (2.75,-0.75) -- (1.25,-0.75);
\node at (2.25,-1.25) {$\ss\mu$};
\draw[<-] (2.55,1.3) -- (1.3,2.55);
\node at (2.2,2.2) {$\ss\nu$};
\draw[<-] (-0.75,1.25) -- (-0.75,2.75);
\node at (-1.25,2.25) {$\ss\lambda$};
\end{tikzpicture}
\end{center}
\end{thm}
Equivalently, the theorem gives a combinatorial formula
for $\left<G^\lambda G^\mu G_\nu\right>$ in $K_T(Gr(k,n))$.

\begin{exlink}{http://www.lpthe.jussieu.fr/~pzinn/puzzles/?height=2&width=3&y1=2&y2=1&y3=3,1&y3comp&K&equiv&mask=35&process}
Consider $n=5$, $k=2$ and
the coefficient of $G^{\stableau{&&\\\\}}(\x;\y)$ in $G^{\stableau{&\\}}(\x;\y)G^{\stableau{\\}}(\x;\y)$.
We obtain a single equivariant puzzle:
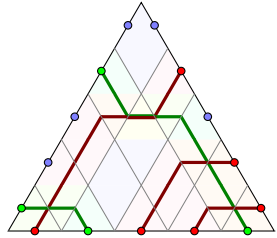
\begin{center}
\def\size{5}
\equivpuzzle{{4, {0, 0}, 0}, {4, {0, 0}, 1}, {1, {0, 1}, 0}, {3, {0, 1}, 
  1}, {0, {0, 2}, 0}, {2, {0, 2}, 1}, {2, {0, 3}, 0}, {2, {0, 3}, 
  1}, {2, {0, 4}, 0}, {0, {1, 0}, 0}, {2, {1, 0}, 1}, {8, {1, 1}, 
  0}, {4, {1, 1}, 1}, {1, {1, 2}, 0}, {1, {1, 2}, 1}, {1, {1, 3}, 
  0}, {1, {2, 0}, 0}, {1, {2, 0}, 1}, {4, {2, 1}, 0}, {4, {2, 1}, 
  1}, {1, {2, 2}, 0}, {1, {3, 0}, 0}, {3, {3, 0}, 1}, {0, {3, 1}, 
  0}, {3, {4, 0}, 0}}
\end{center}
and a single nontrivial weight where the $K$-tile lies, resulting in the coefficient $-y_4/y_2$.
Interestingly, we can instead compute
$G^{\stableau{\\}}(\x;\y)G^{\stableau{&\\}}(\x;\y)$ (\ie switch $\lambda$ and $\mu$) and also obtain a single puzzle:
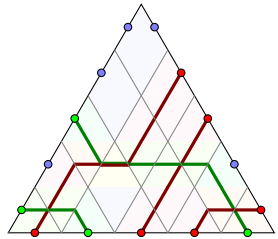
\begin{center}
\def\size{5}
\equivpuzzle{{4, {0, 0}, 0}, {4, {0, 0}, 1}, {1, {0, 1}, 0}, {1, {0, 1}, 
  1}, {1, {0, 2}, 0}, {3, {0, 2}, 1}, {0, {0, 3}, 0}, {2, {0, 3}, 
  1}, {2, {0, 4}, 0}, {4, {1, 0}, 0}, {4, {1, 0}, 1}, {1, {1, 1}, 
  0}, {3, {1, 1}, 1}, {3, {1, 2}, 0}, {1, {1, 2}, 1}, {1, {1, 3}, 
  0}, {0, {2, 0}, 0}, {2, {2, 0}, 1}, {8, {2, 1}, 0}, {4, {2, 1}, 
  1}, {1, {2, 2}, 0}, {1, {3, 0}, 0}, {3, {3, 0}, 1}, {0, {3, 1}, 
  0}, {3, {4, 0}, 0}}
\end{center}
This time, besides the $K$-tile giving a weight of $-y_3/y_2$, we have to its right a tile giving a weight of $y_4/y_3$,
hence the same result $-y_4/y_2$.

The same example was treated in \cite{PY}, but there several diagrams were required to perform the calculation.\footnote{Note that there is an error in the example on page 3 of the current version of \cite{PY}. The sum wt$(P_2)$+wt$(P_3)$+wt$(P_5)$+wt$(P_6)$ simplifies to $-t_2/t_4$, rather than $t_2/t_4-1$ as written. This then agrees with our example above, since the equivariant parameters in \cite{PY} are ordered in reverse with respect to ours.}
\end{exlink}

\begin{exlink}{http://www.lpthe.jussieu.fr/~pzinn/puzzles/?height=2&width=3&y1=2&y2=2,1&y3=3,2&y3comp&K&equiv&mask=35&process}
Consider $n=5$, $k=2$ and
the coefficient of $G^{\stableau{&&\\&\\}}(\x;\y)$ in $G^{\stableau{&\\}}(\x;\y)G^{\stableau{&\\\\}}(\x;\y)$.
We obtain two equivariant puzzles, one of which actually contains an equivariant tile:
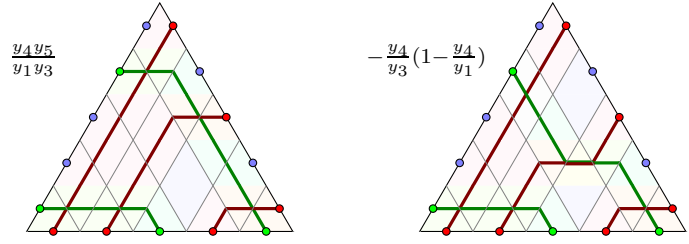
\begin{center}
\def\size{5}
\equivpuzzle[\node at (-\size*0.35,\size*0.6) {$\ss \frac{y_4y_5}{y_1y_3}$};]{{1,{0,0},0},{3,{0,0},1},{0,{0,1},0},{2,{0,1},1},{2,{0,2},0},{0,{0,2},1},{0,{0,3},0},{2,{0,3},1},{2,{0,4},0},{3,{1,0},0},{1,{1,0},1},{1,{1,1},0},{1,{1,1},1},{4,{1,2},0},{4,{1,2},1},{1,{1,3},0},{1,{2,0},0},{1,{2,0},1},{1,{2,1},0},{3,{2,1},1},{0,{2,2},0},{1,{3,0},0},{3,{3,0},1},{3,{3,1},0},{3,{4,0},0}}\qquad
\equivpuzzle[\node at (-\size*0.35,\size*0.6) {$\ss -\frac{y_4}{y_3}(1-\frac{y_4}{y_1})$};]{{1,{0,0},0},{1,{0,0},1},{4,{0,1},0},{4,{0,1},1},{1,{0,2},0},{3,{0,2},1},{0,{0,3},0},{2,{0,3},1},{2,{0,4},0},{7,{1,0},0},{7,{1,0},1},{0,{1,1},0},{2,{1,1},1},{8,{1,2},0},{4,{1,2},1},{1,{1,3},0},{1,{2,0},0},{1,{2,0},1},{1,{2,1},0},{3,{2,1},1},{0,{2,2},0},{1,{3,0},0},{3,{3,0},1},{3,{3,1},0},{3,{4,0},0}}
\end{center}
The sum of weights is $\frac{y_4}{y_1y_3}(-y_1+y_4+y_5)$, which agrees with direct computation
of $\big\langle G^{\stableau{&\\}} G^{\stableau{&\\\\}}G_{\stableau{&&\\&\\}}\big\rangle$.
\end{exlink}

\begin{thmd}\label{thm:equivdual}
Write
\begin{align*}
G_\lambda(\x;\y)G_\mu(\x;\y)=\sum_\nu c_{\lambda,\mu}^{\,\,\nu}(\y) G_\nu(\x;\y).
\end{align*}
Given three Young diagrams $\lambda$, $\mu$, $\nu$ inside the $k\times(n-k)$ rectangle,
$c_{\lambda,\mu}^{\,\,\nu}(\y)$ is the sum of weights of equivariant
puzzles in a triangle pointing down,
with sides labelled $\lambda,\mu,\nu$ all oriented to the right, the parameter in the weights being given by $w=y_i/y_j$ for an elementary lozenge whose coordinates $(i,j)$ are defined as follows:
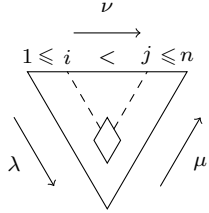
\begin{center}
\def\size{3}
\begin{tikzpicture}[puz,scale=-1]
\draw (0,0) -- (\size,0) -- (0,\size) -- cycle;
\draw (0.5,0.5) -- (0.5,1) -- (1,1) -- (1,0.5) -- cycle;
\node[above=-0.1mm] at (0,3) {$\ss n$};
\node[above=-0.5mm] at (0.35,2.65) {$\ss \leq$};
\node[above=-0.5mm] at (1.5,1.5) {$\ss <$};
\node[above=-0.5mm] at (2.65,0.35) {$\ss \leq$};
\node[above=-0.25mm] at (3,0) {$\ss 1$};
\draw[dashed] (0.75,1) -- (0.75,2.25) node[above=-0.5mm] {$\ss j$};
\draw[dashed] (1,0.75) -- (2.25,0.75) node[above=-0.5mm] {$\ss i$};
\draw[->] (2.75,-0.75) -- (1.25,-0.75);
\node at (2.25,-1.25) {$\ss\lambda$};
\draw[->] (2.55,1.3) -- (1.3,2.55);
\node at (2.2,2.2) {$\ss\nu$};
\draw[->] (-0.75,1.25) -- (-0.75,2.75);
\node at (-1.25,2.25) {$\ss\mu$};
\end{tikzpicture}
\end{center}
\end{thmd}
The theorem gives this time a combinatorial formula
for $\left<G_\lambda G_\mu G^\nu\right>$ in $K_T(Gr(k,n))$.

\begin{exlink}{http://www.lpthe.jussieu.fr/~pzinn/puzzles/?height=2&width=3&y1=3,2&y2=3,1&y3=2&y1comp&y2comp&Kinv&equiv&mask=35&view=9&process}
In $Gr(2,5)$, calculate the coefficient of $G_{\stableau{&\\}}$ in $G_{\stableau{&&\\\\}}G_{\stableau{&&\\&\\}}$:
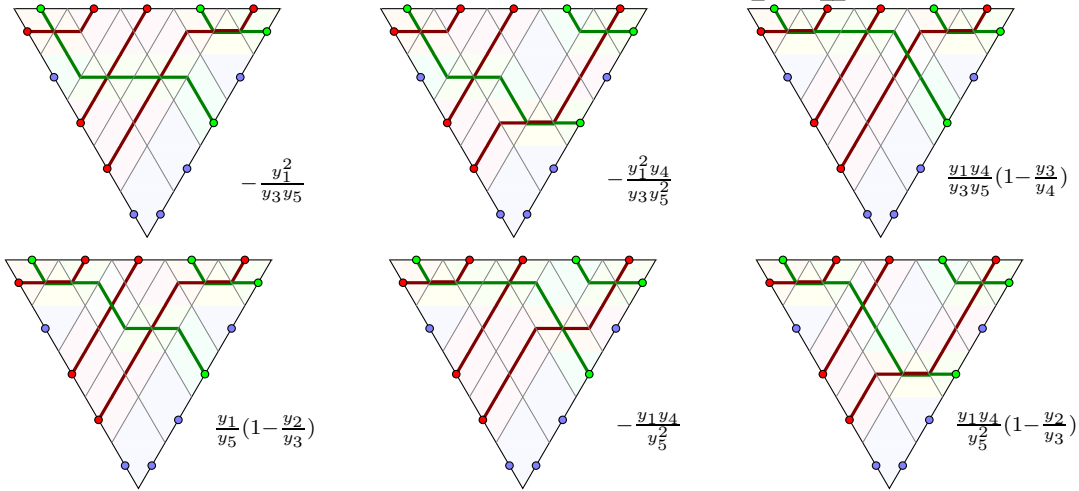
\begin{center}
\def\size{5}
\tikz[scale=-1]{\equivpuzzle[\node at (-\size*0.35,\size*0.6) {$\ss -\frac{y_1^2}{y_3y_5}$};]{{4,{0,0},0},{4,{0,0},1},{1,{0,1},0},{1,{0,1},1},{1,{0,2},0},{3,{0,2},1},{0,{0,3},0},{2,{0,3},1},{2,{0,4},0},{4,{1,0},0},{4,{1,0},1},{1,{1,1},0},{3,{1,1},1},{3,{1,2},0},{1,{1,2},1},{1,{1,3},0},{0,{2,0},0},{0,{2,0},1},{3,{2,1},0},{1,{2,1},1},{1,{2,2},0},{4,{3,0},0},{8,{3,0},1},{2,{3,1},0},{3,{4,0},0}}}\qquad
\tikz[scale=-1]{\equivpuzzle[\node at (-\size*0.35,\size*0.6) {$\ss -\frac{y_1^2y_4}{y_3y_5^2}$};]{{4,{0,0},0},{4,{0,0},1},{1,{0,1},0},{1,{0,1},1},{1,{0,2},0},{3,{0,2},1},{0,{0,3},0},{2,{0,3},1},{2,{0,4},0},{4,{1,0},0},{8,{1,0},1},{2,{1,1},0},{0,{1,1},1},{3,{1,2},0},{1,{1,2},1},{1,{1,3},0},{3,{2,0},0},{1,{2,0},1},{4,{2,1},0},{4,{2,1},1},{1,{2,2},0},{1,{3,0},0},{3,{3,0},1},{0,{3,1},0},{3,{4,0},0}}}\qquad
\tikz[scale=-1]{\equivpuzzle[\node at (-\size*0.35,\size*0.6) {$\ss \frac{y_1y_4}{y_3y_5}(1-\frac{y_3}{y_4})$};]{{4,{0,0},0},{4,{0,0},1},{1,{0,1},0},{1,{0,1},1},{1,{0,2},0},{1,{0,2},1},{4,{0,3},0},{8,{0,3},1},{2,{0,4},0},{4,{1,0},0},{4,{1,0},1},{1,{1,1},0},{1,{1,1},1},{1,{1,2},0},{3,{1,2},1},{3,{1,3},0},{0,{2,0},0},{0,{2,0},1},{7,{2,1},0},{7,{2,1},1},{3,{2,2},0},{4,{3,0},0},{8,{3,0},1},{2,{3,1},0},{3,{4,0},0}}}\\[2mm]
\tikz[scale=-1]{\equivpuzzle[\node at (-\size*0.35,\size*0.6) {$\ss \frac{y_1}{y_5}(1-\frac{y_2}{y_3})$};]{{4,{0,0},0},{4,{0,0},1},{1,{0,1},0},{1,{0,1},1},{1,{0,2},0},{1,{0,2},1},{4,{0,3},0},{8,{0,3},1},{2,{0,4},0},{4,{1,0},0},{4,{1,0},1},{1,{1,1},0},{3,{1,1},1},{7,{1,2},0},{7,{1,2},1},{3,{1,3},0},{0,{2,0},0},{0,{2,0},1},{3,{2,1},0},{1,{2,1},1},{1,{2,2},0},{4,{3,0},0},{8,{3,0},1},{2,{3,1},0},{3,{4,0},0}}}\qquad
\tikz[scale=-1]{\equivpuzzle[\node at (-\size*0.35,\size*0.6) {$\ss -\frac{y_1y_4}{y_5^2}$};]{{4,{0,0},0},{4,{0,0},1},{1,{0,1},0},{1,{0,1},1},{1,{0,2},0},{1,{0,2},1},{4,{0,3},0},{8,{0,3},1},{2,{0,4},0},{4,{1,0},0},{4,{1,0},1},{1,{1,1},0},{1,{1,1},1},{1,{1,2},0},{3,{1,2},1},{3,{1,3},0},{0,{2,0},0},{2,{2,0},1},{2,{2,1},0},{0,{2,1},1},{3,{2,2},0},{1,{3,0},0},{3,{3,0},1},{0,{3,1},0},{3,{4,0},0}}}\qquad
\tikz[scale=-1]{\equivpuzzle[\node at (-\size*0.35,\size*0.6) {$\ss \frac{y_1y_4}{y_5^2}(1-\frac{y_2}{y_3})$};]{{4,{0,0},0},{4,{0,0},1},{1,{0,1},0},{1,{0,1},1},{1,{0,2},0},{1,{0,2},1},{4,{0,3},0},{8,{0,3},1},{2,{0,4},0},{4,{1,0},0},{8,{1,0},1},{2,{1,1},0},{0,{1,1},1},{7,{1,2},0},{7,{1,2},1},{3,{1,3},0},{3,{2,0},0},{1,{2,0},1},{4,{2,1},0},{4,{2,1},1},{1,{2,2},0},{1,{3,0},0},{3,{3,0},1},{0,{3,1},0},{3,{4,0},0}}}
\end{center}
We find
$-\frac{y_1}{y_3 y_5^2} \left(y_1 y_4+y_2 y_4-y_5 y_4+y_1 y_5+y_2 y_5\right)$, which is confirmed by direct calculation
of $\big\langle
G_{\stableau{&&\\\\}}
G_{\stableau{&&\\&\\}}
G^{\stableau{&\\}}\big\rangle$.
\end{exlink}

This last theorem can be further generalized to a product where secondary alphabets differ:
\setcounter{thmdd}{1}
\begin{thmdd}\label{thm:MS}
Write
\begin{align*}
G_\lambda(\x;\z)G_\mu(\x;\y)=\sum_\nu c_{\lambda,\mu}^\nu(\z;\y) G_\nu(\x;\y).
\end{align*}
Given three Young diagrams $\lambda$, $\mu$, $\nu$ inside the $k\times(n-k)$ rectangle,
$c_{\lambda,\mu}^\nu(\z;\y)$ is the sum of weights of equivariant puzzles in a lozenge,
with sides labelled by Young diagrams $\lambda,\mu,\varnothing,\nu$ as indicated on the diagram, and the parameter $w$
in the weights being given by $w=y_i/z_j$ for an elementary lozenge at location $(i,j)$:
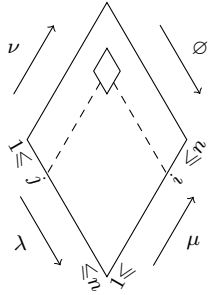
\begin{center}
\def\size{3}
\begin{tikzpicture}[puz]
\draw (0,0) -- (\size,0) -- (\size,\size) -- (0,\size) -- cycle;
\draw (0.5,0.5) -- (0.5,1) -- (1,1) -- (1,0.5) -- cycle;
\node[rotate=-60] at (0,3.3) {$\ss 1\leq$};
\node[rotate=-60] at (2.7,3.3) {$\ss \leq n$};
\draw[dashed] (0.75,1) -- (0.75,3) node[rotate=-60,below=-1mm] {$\ss j$};
\node[rotate=60] at (3.3,0) {$\ss \leq n$};
\node[rotate=60] at (3.3,2.7) {$\ss 1\leq$};
\draw[dashed] (1,0.75) -- (3,0.75) node[rotate=60,below=-1mm] {$\ss i$};
\draw[<-] (2.75,-0.75) -- (1.25,-0.75);
\node at (2.25,-1.25) {$\ss\varnothing$};
\draw[<-] (3.75,0.5) -- (3.75,2);
\node at (4.25,1) {$\ss\mu$};
\draw[->] (0.5,3.75) -- (2,3.75);
\node at (1,4.25) {$\ss\lambda$};
\draw[<-] (-0.75,1.25) -- (-0.75,2.75);
\node at (-1.25,2.25) {$\ss\nu$};
\end{tikzpicture}
\end{center}
\end{thmdd}
The geometric interpretation of this last result is more complicated, \cf\ \cite[Sect. 6]{KT} for the case
of cohomology.

\begin{exlink}{http://www.lpthe.jussieu.fr/~pzinn/puzzles/?height=2&width=3&y1=1,1&y2&y3=2,2&y4=3,1&y3comp&y4comp&K&equiv&mask=35&process}
$n=5$, $k=2$, coefficient of $G_{\stableau{\\\\}}(\x;\y)$ in $G_{\stableau{&&\\\\}}(\x;\z)G_{\stableau{&\\&\\}}(\x;\y)$. There are 5 puzzles:
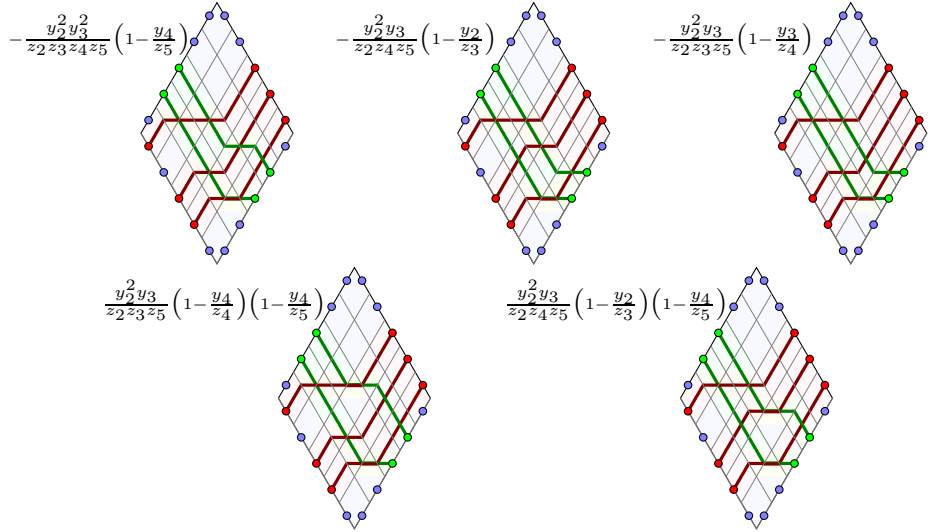
\begin{center}
\def\size{5}\def\puzzlescale{0.4}
\equivdoublepuzzle[\node at (-\size*0.65,\size*0.9) {$\sss-\frac{y_2^2 y_3^2}{z_2 z_3 z_4 z_5} \left(1-\frac{y_4}{z_5}\right)$};]{{4,{0,0},0},{4,{0,0},1},{4,{0,1},0},{4,{0,1},1},{1,{0,2},0},{1,{0,2},1},{1,{0,3},0},{1,{0,3},1},{1,{0,4},0},{1,{0,4},1},{4,{1,0},0},{4,{1,0},1},{4,{1,1},0},{4,{1,1},1},{1,{1,2},0},{1,{1,2},1},{1,{1,3},0},{3,{1,3},1},{7,{1,4},0},{7,{1,4},1},{0,{2,0},0},{0,{2,0},1},{0,{2,1},0},{2,{2,1},1},{2,{2,2},0},{0,{2,2},1},{3,{2,3},0},{1,{2,3},1},{1,{2,4},0},{3,{2,4},1},{0,{3,0},0},{2,{3,0},1},{2,{3,1},0},{0,{3,1},1},{0,{3,2},0},{2,{3,2},1},{2,{3,3},0},{2,{3,3},1},{8,{3,4},0},{4,{3,4},1},{1,{4,0},0},{1,{4,0},1},{4,{4,1},0},{4,{4,1},1},{1,{4,2},0},{1,{4,2},1},{1,{4,3},0},{1,{4,3},1},{4,{4,4},0},{4,{4,4},1}}\quad
\equivdoublepuzzle[\node at (-\size*0.65,\size*0.9) {$\sss-\frac{y_2^2 y_3}{z_2 z_4 z_5} \left(1-\frac{y_2}{z_3}\right)$};]{{4,{0,0},0},{4,{0,0},1},{4,{0,1},0},{4,{0,1},1},{1,{0,2},0},{1,{0,2},1},{1,{0,3},0},{1,{0,3},1},{1,{0,4},0},{1,{0,4},1},{4,{1,0},0},{4,{1,0},1},{4,{1,1},0},{4,{1,1},1},{1,{1,2},0},{1,{1,2},1},{1,{1,3},0},{1,{1,3},1},{1,{1,4},0},{3,{1,4},1},{0,{2,0},0},{0,{2,0},1},{0,{2,1},0},{2,{2,1},1},{2,{2,2},0},{2,{2,2},1},{2,{2,3},0},{0,{2,3},1},{3,{2,4},0},{3,{2,4},1},{0,{3,0},0},{2,{3,0},1},{2,{3,1},0},{0,{3,1},1},{7,{3,2},0},{7,{3,2},1},{0,{3,3},0},{2,{3,3},1},{8,{3,4},0},{4,{3,4},1},{1,{4,0},0},{1,{4,0},1},{4,{4,1},0},{4,{4,1},1},{1,{4,2},0},{1,{4,2},1},{1,{4,3},0},{1,{4,3},1},{4,{4,4},0},{4,{4,4},1}}\quad
\equivdoublepuzzle[\node at (-\size*0.65,\size*0.9) {$\sss-\frac{y_2^2 y_3}{z_2 z_3 z_5} \left(1-\frac{y_3}{z_4}\right)$};]{{4,{0,0},0},{4,{0,0},1},{4,{0,1},0},{4,{0,1},1},{1,{0,2},0},{1,{0,2},1},{1,{0,3},0},{1,{0,3},1},{1,{0,4},0},{1,{0,4},1},{4,{1,0},0},{4,{1,0},1},{4,{1,1},0},{4,{1,1},1},{1,{1,2},0},{1,{1,2},1},{1,{1,3},0},{1,{1,3},1},{1,{1,4},0},{3,{1,4},1},{0,{2,0},0},{0,{2,0},1},{0,{2,1},0},{2,{2,1},1},{2,{2,2},0},{0,{2,2},1},{7,{2,3},0},{7,{2,3},1},{3,{2,4},0},{3,{2,4},1},{0,{3,0},0},{2,{3,0},1},{2,{3,1},0},{0,{3,1},1},{0,{3,2},0},{2,{3,2},1},{2,{3,3},0},{2,{3,3},1},{8,{3,4},0},{4,{3,4},1},{1,{4,0},0},{1,{4,0},1},{4,{4,1},0},{4,{4,1},1},{1,{4,2},0},{1,{4,2},1},{1,{4,3},0},{1,{4,3},1},{4,{4,4},0},{4,{4,4},1}}\\
\equivdoublepuzzle[\node at (-\size*0.8,\size*1.05) {$\sss\frac{y_2^2 y_3}{z_2 z_3 z_5} \left(1-\frac{y_4}{z_4}\right) \left(1-\frac{y_4}{z_5}\right)$};]{{4,{0,0},0},{4,{0,0},1},{4,{0,1},0},{4,{0,1},1},{1,{0,2},0},{1,{0,2},1},{1,{0,3},0},{1,{0,3},1},{1,{0,4},0},{1,{0,4},1},{4,{1,0},0},{4,{1,0},1},{4,{1,1},0},{4,{1,1},1},{1,{1,2},0},{3,{1,2},1},{7,{1,3},0},{7,{1,3},1},{7,{1,4},0},{7,{1,4},1},{0,{2,0},0},{0,{2,0},1},{0,{2,1},0},{2,{2,1},1},{8,{2,2},0},{4,{2,2},1},{1,{2,3},0},{1,{2,3},1},{1,{2,4},0},{3,{2,4},1},{0,{3,0},0},{2,{3,0},1},{2,{3,1},0},{0,{3,1},1},{0,{3,2},0},{2,{3,2},1},{2,{3,3},0},{2,{3,3},1},{8,{3,4},0},{4,{3,4},1},{1,{4,0},0},{1,{4,0},1},{4,{4,1},0},{4,{4,1},1},{1,{4,2},0},{1,{4,2},1},{1,{4,3},0},{1,{4,3},1},{4,{4,4},0},{4,{4,4},1}}\qquad
\equivdoublepuzzle[\node at (-\size*0.8,\size*1.05) {$\sss \frac{y_2^2 y_3}{z_2 z_4 z_5} \left(1-\frac{y_2}{z_3}\right)\left(1-\frac{y_4}{z_5}\right)$};]{{4,{0,0},0},{4,{0,0},1},{4,{0,1},0},{4,{0,1},1},{1,{0,2},0},{1,{0,2},1},{1,{0,3},0},{1,{0,3},1},{1,{0,4},0},{1,{0,4},1},{4,{1,0},0},{4,{1,0},1},{4,{1,1},0},{4,{1,1},1},{1,{1,2},0},{1,{1,2},1},{1,{1,3},0},{3,{1,3},1},{7,{1,4},0},{7,{1,4},1},{0,{2,0},0},{0,{2,0},1},{0,{2,1},0},{2,{2,1},1},{2,{2,2},0},{2,{2,2},1},{8,{2,3},0},{4,{2,3},1},{1,{2,4},0},{3,{2,4},1},{0,{3,0},0},{2,{3,0},1},{2,{3,1},0},{0,{3,1},1},{7,{3,2},0},{7,{3,2},1},{0,{3,3},0},{2,{3,3},1},{8,{3,4},0},{4,{3,4},1},{1,{4,0},0},{1,{4,0},1},{4,{4,1},0},{4,{4,1},1},{1,{4,2},0},{1,{4,2},1},{1,{4,3},0},{1,{4,3},1},{4,{4,4},0},{4,{4,4},1}}
\end{center}
Summing these contributions produces
the correct coefficient $\frac{y_2^2 y_3 y_4}{z_2 z_3 z_4 z_5^2} \left(y_2+y_3+y_4-z_3-z_4-z_5\right)$.
\end{exlink}

\subsection{Variations of equivariant theorems}
Finally, we present a series of ``variations'' on these results. In principle, they can be obtained
from the theorems above by applying duality arguments, but for pedagogical reasons (and because
they will appear as necessary intermediate results), we shall
derive them using integrability only. They involve appropriate modifications of the weights:
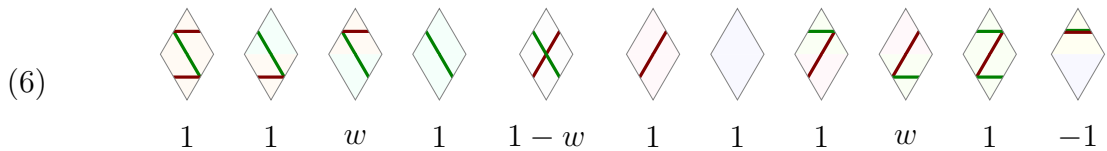
\begin{equation}\label{eq:modwei}
\vcenter{\hbox{\begin{tikzpicture}
\matrix[column sep=0.4cm,row sep=0.2cm]
{
\dtile{{2,{0,0},0},{2,{0,0},1}}
&
\dtile{{0,{0,0},0},{2,{0,0},1}}
&
\dtile{{2,{0,0},0},{0,{0,0},1}}
&
\dtile{{0,{0,0},0},{0,{0,0},1}}
&
\dtile{{7,{0,0},0},{7,{0,0},1}}
&
\dtile{{1,{0,0},0},{1,{0,0},1}}
&
\dtile{{4,{0,0},0},{4,{0,0},1}}
&
\dtile{{3,{0,0},0},{1,{0,0},1}}
&
\dtile{{1,{0,0},0},{3,{0,0},1}}
&
\dtile{{3,{0,0},0},{3,{0,0},1}}
&
\dtile{{8,{0,0},0},{4,{0,0},1}}
\\
\node{$1$};
&
\node{$1$};
&
\node{$w$};
&
\node{$1$};
&
\node{$1-w$};
&
\node{$1$};
&
\node{$1$};
&
\node{$1$};
&
\node{$w$};
&
\node{$1$};
&
\node{$-1$};
\\
};
\end{tikzpicture}}}
\end{equation}

\begin{thm}\label{thm:equivalt}
Expand
\begin{align*}
G^\lambda(\x;\y)
G^\mu(\x;\yy)
=
\sum_\nu
d^{\lambda,\mu}_\nu(\y)
G^\nu(\x;\yy).
\end{align*}
Then $d^{\lambda,\mu}_\nu(\y)$ is the sum of modified weights of equivariant puzzles
in an up-pointing triangle with weights given in terms of $w=y_i/y_j$ and boundary conditions:
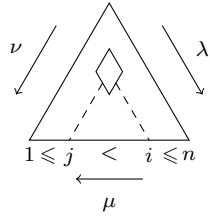
\begin{center}
\def\size{3}
\begin{tikzpicture}[puz]
\draw (0,0) -- (\size,0) -- (0,\size) -- cycle;
\draw (0.5,0.5) -- (0.5,1) -- (1,1) -- (1,0.5) -- cycle;
\node[below=-0.5mm] at (0,3) {$\ss 1$};
\node[below=-0.5mm] at (0.35,2.65) {$\ss \leq$};
\node[below=-0.5mm] at (1.5,1.5) {$\ss <$};
\node[below=-0.5mm] at (2.65,0.35) {$\ss \leq$};
\node[below=-0.1mm] at (3,0) {$\ss n$};
\draw[dashed] (0.75,1) -- (0.75,2.25) node[below=-0.5mm] {$\ss j$};
\draw[dashed] (1,0.75) -- (2.25,0.75) node[below=-0.5mm] {$\ss i$};
\draw[<-] (2.75,-0.75) -- (1.25,-0.75);
\node at (2.25,-1.25) {$\ss\lambda$};
\draw[->] (2.55,1.3) -- (1.3,2.55);
\node at (2.2,2.2) {$\ss\mu$};
\draw[->] (-0.75,1.25) -- (-0.75,2.75);
\node at (-1.25,2.25) {$\ss\nu$};
\end{tikzpicture}
\end{center}
\end{thm}

\begin{exlink}{http://www.lpthe.jussieu.fr/~pzinn/puzzles/?height=2&width=2&y1=1&y3=1&y2comp&K&equiv&mask=35&process}
$n=4$, $k=2$, expansion of $G^{\stableau{\\}}(\x;\y)G^{\stableau{\\}}(\x;\yy)$. We find the following puzzles:

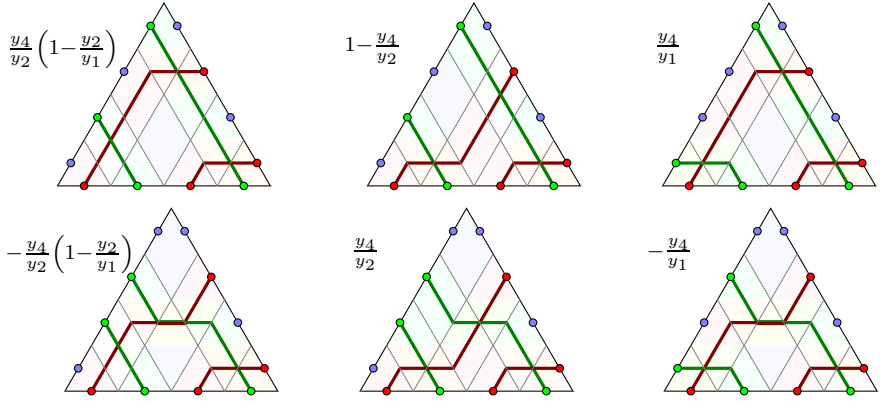
\begin{center}
\def\size{4}
\equivpuzzle[\node at (-\size*0.35,\size*0.6) {$\ss \frac{y_4}{y_2}\left(1-\frac{y_2}{y_1}\right)$};]{{0,{0,0},0},{2,{0,0},1},{2,{0,1},0},{0,{0,1},1},{0,{0,2},0},{2,{0,2},1},{2,{0,3},0},{1,{1,0},0},{1,{1,0},1},{4,{1,1},0},{4,{1,1},1},{1,{1,2},0},{7,{2,0},0},{7,{2,0},1},{0,{2,1},0},{1,{3,0},0}}\qquad
\equivpuzzle[\node at (-\size*0.35,\size*0.6) {$\ss 1-\frac{y_4}{y_2}$};]{{0,{0,0},0},{0,{0,0},1},{7,{0,1},0},{7,{0,1},1},{0,{0,2},0},{2,{0,2},1},{2,{0,3},0},{4,{1,0},0},{4,{1,0},1},{1,{1,1},0},{1,{1,1},1},{1,{1,2},0},{0,{2,0},0},{2,{2,0},1},{2,{2,1},0},{1,{3,0},0}}\qquad
\equivpuzzle[\node at (-\size*0.35,\size*0.6) {$\ss \frac{y_4}{y_1}$};]{{0,{0,0},0},{2,{0,0},1},{2,{0,1},0},{0,{0,1},1},{0,{0,2},0},{2,{0,2},1},{2,{0,3},0},{1,{1,0},0},{1,{1,0},1},{4,{1,1},0},{4,{1,1},1},{1,{1,2},0},{1,{2,0},0},{3,{2,0},1},{0,{2,1},0},{3,{3,0},0}}\\[2mm]
\equivpuzzle[\node at (-\size*0.35,\size*0.6) {$\ss -\frac{y_4}{y_2}\left(1-\frac{y_2}{y_1}\right)$};]{{4,{0,0},0},{4,{0,0},1},{1,{0,1},0},{3,{0,1},1},{0,{0,2},0},{2,{0,2},1},{2,{0,3},0},{0,{1,0},0},{2,{1,0},1},{8,{1,1},0},{4,{1,1},1},{1,{1,2},0},{7,{2,0},0},{7,{2,0},1},{0,{2,1},0},{1,{3,0},0}}\qquad
\equivpuzzle[\node at (-\size*0.35,\size*0.6) {$\ss \frac{y_4}{y_2}$};]{{4,{0,0},0},{4,{0,0},1},{1,{0,1},0},{3,{0,1},1},{0,{0,2},0},{2,{0,2},1},{2,{0,3},0},{0,{1,0},0},{0,{1,0},1},{3,{1,1},0},{1,{1,1},1},{1,{1,2},0},{0,{2,0},0},{2,{2,0},1},{2,{2,1},0},{1,{3,0},0}}\qquad
\equivpuzzle[\node at (-\size*0.35,\size*0.6) {$\ss -\frac{y_4}{y_1}$};]{{4,{0,0},0},{4,{0,0},1},{1,{0,1},0},{3,{0,1},1},{0,{0,2},0},{2,{0,2},1},{2,{0,3},0},{0,{1,0},0},{2,{1,0},1},{8,{1,1},0},{4,{1,1},1},{1,{1,2},0},{1,{2,0},0},{3,{2,0},1},{0,{2,1},0},{3,{3,0},0}}
\end{center}
which matches with
\begin{align*}
G^{\stableau{\\}}(\x;\y)G^{\stableau{\\}}(\x;\yy)
=
\left(1-\frac{y_4}{y_1}\right)
G^{\stableau{\\}}(\x;\yy)
+\frac{y_4}{y_1}
G^{\stableau{&\\}}(\x;\yy)
+\frac{y_4}{y_1}
G^{\stableau{\\\\}}(\x;\yy)
-\frac{y_4}{y_1}
G^{\stableau{&\\\\}}(\x;\yy).
\end{align*}
\end{exlink}

\begin{thmd}\label{thm:equivdualalt}
Expand
\begin{align*}
G_\lambda(\x;\y)
G_\mu(\x;\yy)
=
\sum_\nu
d_{\lambda,\mu}^\nu(\y)
G_\nu(\x;\yy)
\end{align*}
Then $d_{\lambda,\mu}^\nu(\y)$ is the sum of modified weights of equivariant puzzles
in a down-pointing triangle with weights given in terms of $w=y_i/y_j$ and boundary conditions:
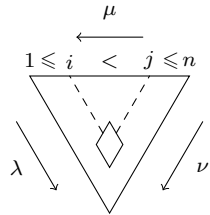
\begin{center}
\def\size{3}
\begin{tikzpicture}[puz,scale=-1]
\draw (0,0) -- (\size,0) -- (0,\size) -- cycle;
\draw (0.5,0.5) -- (0.5,1) -- (1,1) -- (1,0.5) -- cycle;
\node[above=-0.1mm] at (0,3) {$\ss n$};
\node[above=-0.5mm] at (0.35,2.65) {$\ss \leq$};
\node[above=-0.5mm] at (1.5,1.5) {$\ss <$};
\node[above=-0.5mm] at (2.65,0.35) {$\ss \leq$};
\node[above=-0.25mm] at (3,0) {$\ss 1$};
\draw[dashed] (0.75,1) -- (0.75,2.25) node[above=-0.5mm] {$\ss j$};
\draw[dashed] (1,0.75) -- (2.25,0.75) node[above=-0.5mm] {$\ss i$};
\draw[->] (2.75,-0.75) -- (1.25,-0.75);
\node at (2.25,-1.25) {$\ss\lambda$};
\draw[<-] (2.55,1.3) -- (1.3,2.55);
\node at (2.2,2.2) {$\ss\mu$};
\draw[<-] (-0.75,1.25) -- (-0.75,2.75);
\node at (-1.25,2.25) {$\ss\nu$};
\end{tikzpicture}
\end{center}
\end{thmd}

We also have an alternative expression for the same coefficients 
$c_{\lambda,\mu}^\nu(\z;\y)$ introduced in Theorem~\ref{thm:MS}:
\begin{thmdd}\label{thm:MSalt}
$c_{\lambda,\mu}^\nu(\z;\y)$ is also the sum of modified weights of equivariant puzzles in a lozenge,
with sides labelled by Young diagrams $\lambda,\mu,\varnothing,\nu$ as indicated on the diagram, and the parameter $w$
in the weights being given by $w=y_i/z_j$ for an elementary lozenge at location $(i,j)$:
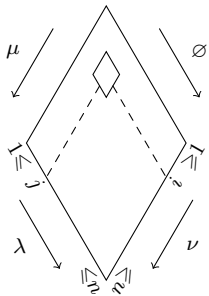
\begin{center}
\def\size{3}
\begin{tikzpicture}[puz]
\draw (0,0) -- (\size,0) -- (\size,\size) -- (0,\size) -- cycle;
\draw (0.5,0.5) -- (0.5,1) -- (1,1) -- (1,0.5) -- cycle;
\node[rotate=-60] at (0,3.3) {$\ss 1\leq$};
\node[rotate=-60] at (2.7,3.3) {$\ss \leq n$};
\draw[dashed] (0.75,1) -- (0.75,3) node[rotate=-60,below=-1mm] {$\ss j$};
\node[rotate=60] at (3.3,0) {$\ss \geq 1$};
\node[rotate=60] at (3.3,2.7) {$\ss n\geq$};
\draw[dashed] (1,0.75) -- (3,0.75) node[rotate=60,below=-1mm] {$\ss i$};
\draw[<-] (2.75,-0.75) -- (1.25,-0.75);
\node at (2.25,-1.25) {$\ss\varnothing$};
\draw[->] (3.75,0.5) -- (3.75,2);
\node at (4.25,1) {$\ss\nu$};
\draw[->] (0.5,3.75) -- (2,3.75);
\node at (1,4.25) {$\ss\lambda$};
\draw[->] (-0.75,1.25) -- (-0.75,2.75);
\node at (-1.25,2.25) {$\ss\mu$};
\end{tikzpicture}
\end{center}
\end{thmdd}

\subsection{On the order of proofs}
Theorems~\ref{thm:nonequiv} and \ref{thm:nonequivdual}
are clearly special cases of Theorems~\ref{thm:equiv} and \ref{thm:equivdual}, respectively (noting
that in the absence of equivariant tiles, the number of $K$-tiles is fixed to be $|\nu|-|\lambda|-|\mu|$, resp.\ 
$|\lambda|+|\mu|-|\nu|-k(n-k)$).
In what follows, we shall prove, in this order,
Theorems~\ref{thm:MSalt}, \ref{thm:equivdualalt}, \ref{thm:equivalt}, \ref{thm:MS}, \ref{thm:equivdual} and \ref{thm:equiv}.

\section{Definition of the integrable model}
\label{sec:model}

\subsection{Grothendieck polynomials from an integrable five-vertex model}
\label{ssec:5v}

Like many families of symmetric functions, Grothendieck polynomials can be expressed as partition functions of an integrable vertex model. In the case of ordinary Grothendieck polynomials (without the secondary alphabet $\y$), such a construction is given in \cite{ms1,ms2}. The usual method of deriving these expressions is to translate the branching formula for the symmetric function under consideration into the language of row-to-row transfer matrices. Here we give a construction for the double Grothendieck polynomials $G^{\lambda}(\x;\y)$ that only makes use of their inductive defintion \eqref{ind-0}--\eqref{ind-i} and the Yang--Baxter equation of the underlying vertex model.

Consider the following $R$-matrix: 
\begin{align}
\label{5vR}
R_{a,b}(z)
=
\left(
\begin{array}{cccc}
1 & 0 & 0 & 0
\\
0 & 0 & z & 0
\\
0 & 1 & 1-z & 0
\\
0 & 0 & 0 & 1
\end{array}
\right)_{a,b}
=
\left(
\begin{array}{cccc}
\begin{tikzpicture}[scale=0.5,baseline=0.125cm]
\filldraw[fill=c4!15!white,draw=black] (0,0) -- (1,0) -- (1,1) -- (0,1) -- (0,0);
\end{tikzpicture}
&
0
&
0
&
0
\\
0 
& 
0 
& 
\begin{tikzpicture}[scale=0.5,baseline=0.125cm]
\filldraw[fill=c4!15!white,draw=black] (0,0) -- (1,0) -- (1,1) -- (0,1) -- (0,0);
\draw[edge=green] (0.5,0) -- (1,0.5);
\node[vertex=green] at (0.5,0) {};
\node[vertex=green] at (1,0.5) {};
\end{tikzpicture}
&
0
\\
0 
&
\begin{tikzpicture}[scale=0.5,baseline=0.125cm]
\filldraw[fill=c4!15!white,draw=black] (0,0) -- (1,0) -- (1,1) -- (0,1) -- (0,0);
\draw[edge=green] (0,0.5) -- (0.5,1);
\node[vertex=green] at (0,0.5) {};
\node[vertex=green] at (0.5,1) {};
\end{tikzpicture}
&
\begin{tikzpicture}[scale=0.5,baseline=0.125cm]
\filldraw[fill=c4!15!white,draw=black] (0,0) -- (1,0) -- (1,1) -- (0,1) -- (0,0);
\draw[edge=green] (0,0.5) -- (1,0.5);
\node[vertex=green] at (0,0.5) {};
\node[vertex=green] at (1,0.5) {};
\end{tikzpicture}
&
0
\\
0 
&
0
&
0
&
\begin{tikzpicture}[scale=0.5,baseline=0.125cm]
\filldraw[fill=c4!15!white,draw=black] (0,0) -- (1,0) -- (1,1) -- (0,1) -- (0,0);
\draw[edge=green] (0.5,0) -- (1,0.5);
\draw[edge=green] (0,0.5) -- (0.5,1);
\node[vertex=green] at (0.5,0) {};
\node[vertex=green] at (1,0.5) {};
\node[vertex=green] at (0,0.5) {};
\node[vertex=green] at (0.5,1) {};
\end{tikzpicture}
\end{array}
\right)_{a,b}
\end{align}
which acts in a tensor product $V_a \otimes V_b$ of two two-dimensional vector scaces. It can be obtained from the $R$-matrix of the stochastic six-vertex model, simply by sending the quantum parameter to zero (see, for example, equation (9) of \cite{wz-j} and set $t=0$). It satisfies the Yang--Baxter equation
\begin{align}
\label{yb}
R_{a,b}(z_a/z_b) R_{a,c}(z_a/z_c) R_{b,c}(z_b/z_c)
=
R_{b,c}(z_b/z_c) R_{a,c}(z_a/z_c) R_{a,b}(z_a/z_b)
\end{align}
for any three parameters $z_a,z_b,z_c$. For our purposes, it is useful to write \eqref{yb} in two slightly different ways, which are simply relabellings of the original equation:
\begin{align}
\label{yb1}
R_{a,b}(x_a/x_b)
R_{a,c}(x_a/y_c)
R_{b,c}(x_b/y_c)
=
R_{b,c}(x_b/y_c)
R_{a,c}(x_a/y_c)
R_{a,b}(x_a/x_b),
\\
\label{yb2}
R_{b,a}(y_b/y_a)
R_{c,a}(x_c/y_a)
R_{c,b}(x_c/y_b)
=
R_{c,b}(x_c/y_b)
R_{c,a}(x_c/y_a)
R_{b,a}(y_b/y_a).
\end{align}
Construct a monodromy matrix by taking an $n$-fold product of $R$-matrices:
\begin{align*}
T_a(x,\y)
=
R_{a,1}(x/y_1)
\dots
R_{a,n}(x/y_n)
=
\begin{pmatrix}
A(x,\y) & B(x,\y)
\\
C(x,\y) & D(x,\y)
\end{pmatrix}_a
\end{align*}
where each entry of the final matrix is an operator acting in $\mathbb{V} = V_1 \otimes \cdots \otimes V_n$. Using the Yang--Baxter equation \eqref{yb1}, one can easily show that
\begin{align*}
[B(x_i,\y),B(x_j,\y)]
=
[C(x_i,\y),C(x_j,\y)]
=
0
\quad
\forall\ i,j,
\end{align*}
while \eqref{yb2} gives rise to the relations
\begin{align}
\label{ex-B}
R_{i+1,i}(y_{i+1}/y_i)
B(x,\y)
=
B(x,s_i \circ \y)
R_{i+1,i}(y_{i+1}/y_i),
\\
\label{ex-C}
R_{i+1,i}(y_{i+1}/y_i)
C(x,\y)
=
C(x,s_i \circ \y)
R_{i+1,i}(y_{i+1}/y_i),
\end{align}
where $s_i \circ \y$ denotes the set $\y$ with $y_i$ and $y_{i+1}$ exchanged. 

In order to describe the action of the monodromy matrix entries, it is useful to have a canonical basis for $\mathbb{V}$. First note that it can be decomposed into subspaces:
\begin{align*}
\mathbb{V} = \bigoplus_{k=0}^{n} \mathbb{V}[k],
\end{align*}
where $\mathbb{V}[k]$ is the span of all $k$-particle states in $V_1 \otimes \cdots \otimes V_n$. We shall use the Young diagram basis for $\mathbb{V}[k]$, whose elements are given by
\begin{align*}
\ket{\lambda}
=
\bigotimes_{i \in \lambda}
\begin{pmatrix} 0 \\ 1 \end{pmatrix}_i
\bigotimes_{j \not\in \lambda}
\begin{pmatrix} 1 \\ 0 \end{pmatrix}_j
\equiv
\bigotimes_{i \in \lambda}
\ket{\begin{tikzpicture} \node[vertex=green] at (0,0) {}; \end{tikzpicture}}_i
\bigotimes_{j \not\in \lambda}
\ket{\begin{tikzpicture} \node[vertex=lightblue] at (0,0) {}; \end{tikzpicture}}_j
\end{align*}
where $\lambda$ is a Young diagram in the $k \times (n-k)$ rectangle. 
We also define 
\[
\ket{0} = \bigotimes_{j=1}^{n}
\ket{\begin{tikzpicture} \node[vertex=lightblue] at (0,0) {}; \end{tikzpicture}}_j \in \mathbb{V}[0].
\]
It should not be confused with the basis element of $\mathbb{V}[k]$ associated to the empty Young diagram $\lambda=\varnothing$, which is
$\ket{\varnothing}
=
\bigotimes_{i=1}^{k}
\ket{\begin{tikzpicture} \node[vertex=green] at (0,0) {}; \end{tikzpicture}}_i
\bigotimes_{j=k+1}^{n}
\ket{\begin{tikzpicture} \node[vertex=lightblue] at (0,0) {}; \end{tikzpicture}}_j$.

It is easy to show that the $B$ and $C$ operators map between the subspaces:
\begin{align*}
B(x,\y) : \mathbb{V}[k] \rightarrow \mathbb{V}[k+1],
\quad\quad
C(x,\y) : \mathbb{V}[k] \rightarrow \mathbb{V}[k-1],
\end{align*}
which motivates the following proposition.
\begin{prop}\label{prop:groth}
Double Grothendieck polynomials and their duals can be written as expectation values in this five-vertex model:
\begin{align}
\label{G-C}
G^{\lambda}(\x;\y)
=
\bra{0} C(x_1,\y) \dots C(x_k,\y) \ket{\lambda},
\\
\label{G*-B}
G_{\lambda}(\x;\y)
=
\bra{\lambda} B(x_1,\y) \dots B(x_k,\y) \ket{0}.
\end{align}
\end{prop}

\begin{proof}
We begin with the proof of \eqref{G-C}. First, it is obvious from graphical considerations that
\begin{align*}
\bra{0} C(x_1,\y) \dots C(x_k,\y) \ket{\lambda}
=
\prod_{i=1}^{k}
\prod_{j=1}^{n-k}
\left(1-\frac{x_i}{y_j}\right)
\end{align*}
when $\lambda = \{n-k+1,\dots,n\}$, since in this case the partition function is completely frozen:
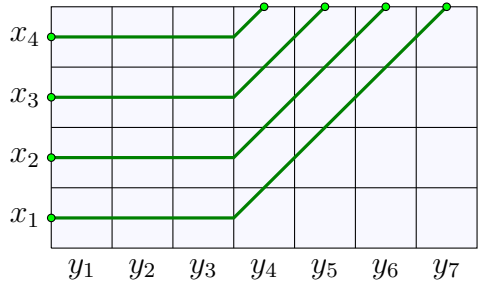
\begin{center}
\begin{tikzpicture}[scale=0.8]
\filldraw[fill=c4!15!white,draw=black] (0,0) -- (0,4) -- (7,4) -- (7,0) -- (0,0);
\foreach\y in {1,...,3}{
\draw[thin] (0,\y) -- (7,\y);
}
\foreach\x in {1,...,6}{
\draw[thin] (\x,0) -- (\x,4);
\node[below] at (\x-0.5,0) {$y_\x$};
}
\node[below] at (6.5,0) {$y_7$};
\foreach\y in {1,...,4}{
\draw[edge=green] (0,\y-0.5) -- (3,\y-0.5) -- (7.5-\y,4);
\node[left] at (0,\y-0.5) {$x_\y$};
\node[vertex=green] at (0,\y-0.5) {};
\node[vertex=green] at (7.5-\y,4) {};
} 
\end{tikzpicture}
\end{center}
This proves that $G^{\lambda}(\x;\y) = \bra{0} C(x_1,\y) \dots C(x_k,\y) \ket{\lambda}$ for $\lambda=k \times (n-k)$, the maximal Young diagram. Now suppose it holds for some generic $\lambda$, such that $i \not\in \lambda, i+1 \in \lambda$. Since $\bra{0} R_{i+1,i}(y_{i+1}/y_i) = \bra{0}$ for all $1 \leq i \leq n-1$, we have
\begin{multline}
\label{induct}
G^{\lambda}(\x;\y)
=
\bra{0} R_{i+1,i}(y_{i+1}/y_i) C(x_1,\y) \dots C(x_k,\y) \ket{\lambda}
\\
=
\bra{0} C(x_1,s_i \circ \y) \dots C(x_k,s_i \circ \y) R_{i+1,i}(y_{i+1}/y_i) \ket{\lambda}
\end{multline}
by $k$ repetitions of \eqref{ex-C}. Finally, $R_{i+1,i}(y_{i+1}/y_i) \ket{\lambda}$ is easily calculated using the explicit form of the $R$-matrix. One finds that
\begin{align*}
R_{i+1,i}(y_{i+1}/y_i) \ket{\lambda}
=
(1-y_{i+1}/y_i) \ket{s_i \circ \lambda}
+
y_{i+1}/y_i \ket{\lambda},
\end{align*}
where $s_i \circ \lambda$ denotes $\lambda$ under the substitution $i+1 \mapsto i$. Using this result in \eqref{induct}, we have shown that
\begin{align*}
G^{\lambda}(\x;\y)
=
(1-y_{i+1}/y_i)
\bra{0} C(x_1,s_i \circ \y) \dots C(x_k,s_i \circ \y) \ket{s_i \circ \lambda}
+
y_{i+1}/y_i
G^{\lambda}(\x;s_i \circ \y)
\end{align*}
Comparing with \eqref{ind-i}, it follows that
\begin{align*}
\bra{0} C(x_1,s_i \circ \y) \dots C(x_k,s_i \circ \y) \ket{s_i \circ \lambda}
=
G^{s_i \circ \lambda}(\x;s_i\circ\y).
\end{align*}
By induction on all such $i$, one proves that $G^{\lambda}(\x;\y) = \bra{0} C(x_1,\y) \dots C(x_k,\y) \ket{\lambda}$ for all $\lambda$ contained in the $k \times (n-k)$ rectangle.

The proof of \eqref{G*-B} proceeds along very similar lines. In this case we have initially
\begin{align*}
\bra{\lambda} B(x_1,\y) \dots B(x_k,\y) \ket{0}
=
\prod_{i=1}^{k}
\left(
\frac{x_i}{y_i}
\right)
\prod_{i=1}^{n}
\prod_{j=k+1}^{n}
\left(
1-\frac{x_i}{y_j}
\right)
\end{align*} 
when $\lambda=\{1,\dots,k\}$, again due to the freezing of the corresponding partition function:
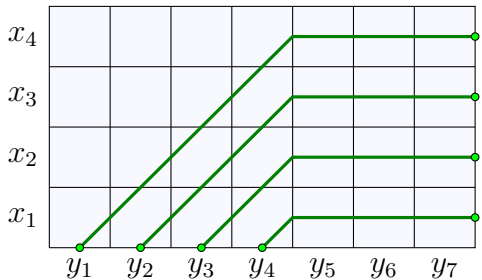
\begin{center}
\begin{tikzpicture}[scale=0.8]
\filldraw[fill=c4!15!white,draw=black] (0,0) -- (0,4) -- (7,4) -- (7,0) -- (0,0);
\foreach\y in {1,...,3}{
\draw[thin] (0,\y) -- (7,\y);
}
\foreach\x in {1,...,6}{
\draw[thin] (\x,0) -- (\x,4);
\node[below] at (\x-0.5,0) {$y_\x$};
}
\node[below] at (6.5,0) {$y_7$};
\foreach\y in {1,...,4}{
\draw[edge=green] (4.5-\y,0) -- (4,\y-0.5) -- (7,\y-0.5) ;
\node[left] at (0,\y-0.5) {$x_\y$};
\node[vertex=green] at (7,\y-0.5) {};
\node[vertex=green] at (4.5-\y,0) {};
} 
\end{tikzpicture}
\end{center}
This proves that $\bra{\lambda} B(x_1,\y) \dots B(x_k,\y) \ket{0} = \prod_{i=1}^{k} x_i \prod_{i \in \lambda} y_i^{-1} G^{\lambda^\ast} (\x;\yy) = G_{\lambda}(\x;\y)$ for $\lambda = \varnothing$. From there one makes use of \eqref{ex-B} and almost directly repeats the reasoning laid out above, to complete the proof of \eqref{G*-B} inductively.

\end{proof}

\subsection{Explicit determinant formulae for double Grothendieck polynomials}

\begin{prop}
Let $\lambda$ be a Young diagram with frame $\{\ell_1,\dots,\ell_k\}$. The double Grothendieck polynomial and its dual admit the following determinant expressions:
\begin{align}
\label{det-G}
G^{\lambda}(\x;\y)
&=
\frac{\det\left( x_j^{k-i} \prod_{m=1}^{\ell_i-1} (1-x_j/y_m) \right)_{1 \leq i,j \leq k}}
{\prod_{1 \leq i<j \leq k} (x_i-x_j)}
\\
\label{det-dual}
G_{\lambda}(\x;\y)
&=
\prod_{i=1}^{k}
(x_i / y_{\ell_i})
\times
\frac{\det\left(x_j^{i-1} \prod_{m = \ell_i+1}^{n} (1-x_j/y_m) \right)_{1\leq i,j \leq k}}
{\prod_{1 \leq i<j \leq k} (x_j-x_i)}
\end{align}
\end{prop}

\begin{proof}
We prove \eqref{det-G}, by showing that it satisfies the relations \eqref{ind-0}--\eqref{ind-i}. Let 
$\mathcal{G}_{\{\ell_1,\dots,\ell_k\}}$ denote the right hand side of \eqref{det-G}. Consider firstly the case 
$\{\ell_1,\dots,\ell_k\} = \{n-k+1,\dots,n\}$, when
\begin{align*}
\mathcal{G}_{\{n-k+1,\dots,n\}}
=
\frac{\det\left( x_j^{k-i} \prod_{m=1}^{n-k+i-1} (1-x_j/y_m) \right)_{1 \leq i,j \leq k}}
{\prod_{1 \leq i<j \leq k} (x_i-x_j)}.
\end{align*}
By subtracting the degree of the polynomial in the denominator from that of the polynomial in the numerator (considering that $\prod_{1 \leq i<j \leq k} (x_i-x_j)$ is a common factor of the numerator), we find that $\mathcal{G}_{\{n-k+1,\dots,n\}}$ is a polynomial of degree $n-k$ in every variable $x_i$. At the same time, $\mathcal{G}_{\{n-k+1,\dots,n\}}$ clearly has the common factor $\prod_{i=1}^{k} \prod_{j=1}^{n-k} (1-x_i/y_j)$, so by factor exhaustion this determines $\mathcal{G}_{\{n-k+1,\dots,n\}}$ up to a constant in $(x_1,\dots,x_k)$. An easy way to calculate this constant is by sending $x_k,\dots,x_1 \rightarrow 0$ successively, from which we find that
\begin{align*}
\left.
\mathcal{G}_{\{n-k+1,\dots,n\}}
\right|_{x_k,\dots,x_1 \rightarrow 0}
=
1.
\end{align*}
This suffices to show that
\begin{align*}
\mathcal{G}_{\{n-k+1,\dots,n\}}
=
\prod_{i=1}^{k}
\prod_{j=1}^{n-k}
\left(
1-\frac{x_i}{y_j}
\right),
\end{align*}
establishing the required initial condition \eqref{ind-0}. To verify the inductive relation \eqref{ind-i}, choose some $\ell_i$ such that $\ell_{i-1} < \ell_i-1$ and act with the Demazure operator $D_{\ell_i-1}$ on 
$\mathcal{G}_{\{\ell_1,\dots,\ell_k\}}$. Since the first $i-1$ rows of the determinant in $\mathcal{G}_{\{\ell_1,\dots,\ell_k\}}$ do not depend on either $y_{\ell_i-1}$ or $y_{\ell_i}$, while the final $k-i$ rows depend on them symmetrically, one finds that the action of $D_{\ell_i-1}$ can be localised to the $i$-th row of the determinant. It is then immediate that
\begin{align*}
D_{\ell_i-1}
\mathcal{G}_{\{\ell_1,\dots,\ell_i,\dots,\ell_k\}}
=
\mathcal{G}_{\{\ell_1,\dots,\ell_i-1,\dots,\ell_k\}},
\end{align*}
establishing \eqref{ind-i}.

Having proved \eqref{det-G}, it is straightforward to use the fact that $\prod_{i=1}^{k} x_i \prod_{i \in \lambda} y_i^{-1} G^{\lambda^\ast} (\x;\yy) = G_{\lambda}(\x;\y)$ to compute \eqref{det-dual} directly.

\end{proof}

\subsection{A rank-two vertex model}
\label{ssec:rank2}

To go beyond the mere construction of Grothendieck polynomials and study their product rules, it is necessary to construct a more general vertex model. Whereas the model \eqref{5vR} can be obtained as a limit of the six-vertex $R$-matrix (or $U_q(\widehat{\mathfrak{sl}(2)})$ universal $R$-matrix, under a tensor product of fundamental representations), in our subsequent calculations we consider a rank-two vertex model, which descends from the $U_q(\widehat{\mathfrak{sl}(3)})$ universal $R$-matrix. Unlike \eqref{yb}, the Yang--Baxter equation that we now consider contains three different $R$-matrices:
\begin{align}
\label{rank2yb}
\ra_{a,b}(y/x)
\rc_{a,c}(x/z)
\rb_{b,c}(z/y)
=
\rb_{b,c}(z/y)
\rc_{a,c}(x/z)
\ra_{a,b}(y/x).
\end{align}
The $R$-matrices each act in a tensor product $V_a \otimes V_b$ of two three-dimensional vector spaces, and are given explicitly by
\begin{Small}
\begin{align*}
\ra_{a,b}(z)
=
\left(
\begin{array}{ccc|ccc|ccc}
0 & 0 & 0 & 0 & -z & 0 & 0 & 0 & z
\\
0 & 0 & 0 & 0 &0 & 0 & 0 & 0 & 0
\\
0 & 0 & 1 & 0 & 0 & 0 & 0 & 0 & 0
\\
\hline
0 & 0 & 0 & 1 & 0 & 0 & 0 & 0 & 0
\\
0 & 0 & 0 & 0 &0 & 0 & 0 & 0 & 1
\\
0 & 0 & 0 & 0 & 0 & 1 & 0 & 0 & 0
\\
\hline
0 & 0 & 0 & 0 &0 & 0 & 1 & 0 & 0
\\
0 & 0 & 0 & 0 &0 & 0 & 0 & 1 & 0
\\
1 & 0 & 0 & 0 & z & 0 & 0 & 0 & 1-z
\end{array}
\right)_{a,b}
=
\left(
\begin{array}{ccc|ccc|ccc}
0 & 0 & 0 & 0 & \tikz[scale=0.4]{\dtile{{8,{0,0},0},{4,{0,0},1}}} & 0 & 0 & 0 & \tikz[scale=0.4]{\dtile{{2,{0,0},0},{0,{0,0},1}}}
\\
0 & 0 & 0 & 0 &0 & 0 & 0 & 0 & 0
\\
0 & 0 &  \tikz[scale=0.4]{\dtile{{2,{0,0},0},{2,{0,0},1}}} & 0 & 0 & 0 & 0 & 0 & 0
\\
\hline
0 & 0 & 0 & \tikz[scale=0.4]{\dtile{{4,{0,0},0},{4,{0,0},1}}} & 0 & 0 & 0 & 0 & 0
\\
0 & 0 & 0 & 0 &0 & 0 & 0 & 0 & \tikz[scale=0.4]{\dtile{{1,{0,0},0},{3,{0,0},1}}}
\\
0 & 0 & 0 & 0 & 0 & \tikz[scale=0.4]{\dtile{{1,{0,0},0},{1,{0,0},1}}} & 0 & 0 & 0
\\
\hline
0 & 0 & 0 & 0 &0 & 0 & \tikz[scale=0.4]{\dtile{{0,{0,0},0},{0,{0,0},1}}} & 0 & 0
\\
0 & 0 & 0 & 0 &0 & 0 & 0 & \tikz[scale=0.4]{\dtile{{3,{0,0},0},{3,{0,0},1}}} & 0
\\
\tikz[scale=0.4]{\dtile{{0,{0,0},0},{2,{0,0},1}}} & 0 & 0 & 0 & \tikz[scale=0.4]{\dtile{{3,{0,0},0},{1,{0,0},1}}} & 0 & 0 & 0 & \tikz[scale=0.4]{\dtile{{7,{0,0},0},{7,{0,0},1}}}
\end{array}
\right)_{a,b}
\end{align*}
\end{Small}
\begin{Small}
\begin{align*}
\rc_{a,b}(z)
=
\left(
\begin{array}{ccc|ccc|ccc}
0 & 0 & 0 & 0 &0 & 0 & 0 & 0 & 0
\\
0 & 1 & 0 & 1 &0 & 0 & 0 & 0 & 0
\\
0 & 0 & 0 & 0 &0 & 0 & -1 & 0 & 0
\\
\hline
0 & z & 0 & 1 &0 & 0 & 0 & 0 & 0
\\
0 & 0 & 0 & 0 & 1-z & 0 & 0 & 0 & 0
\\
0 & 0 & 0 & 0 &0 & 1 & 0 & z & 0
\\
\hline
0 & 0 & 0 & 0 &0 & 0 & 1 & 0 & 0
\\
0 & 0 & 0 & 0 &0 & 1 & 0 & 1 & 0
\\
0 & 0 & 0 & 0 &0 & 0 & 0 & 0 & 0
\end{array}
\right)_{a,b}
=
\left(
\begin{array}{ccc|ccc|ccc}
0 & 0 & 0 & 0 &0 & 0 & 0 & 0 & 0
\\
0 & \tikz[scale=0.4]{\dtilec{{2,{1,0},0},{2,{0,0},1}}} & 0 & \tikz[scale=0.4]{\dtilec{{2,{1,0},0},{1,{0,0},1}}} &0 & 0 & 0 & 0 & 0
\\
0 & 0 & 0 & 0 &0 & 0 & \tikz[scale=0.4]{\dtilec{{8,{1,0},0},{3,{0,0},1}}} & 0 & 0
\\
\hline
0 & \tikz[scale=0.4]{\dtilec{{1,{1,0},0},{2,{0,0},1}}} & 0 & \tikz[scale=0.4]{\dtilec{{1,{1,0},0},{1,{0,0},1}}} &0 & 0 & 0 & 0 & 0
\\
0 & 0 & 0 & 0 &\tikz[scale=0.4]{\dtilec{{5,{1,0},0},{5,{0,0},1}}} & 0 & 0 & 0 & 0
\\
0 & 0 & 0 & 0 &0 & \tikz[scale=0.4]{\dtilec{{4,{1,0},0},{4,{0,0},1}}} & 0 & \tikz[scale=0.4]{\dtilec{{4,{1,0},0},{0,{0,0},1}}} & 0
\\
\hline
0 & 0 & 0 & 0 &0 & 0 & \tikz[scale=0.4]{\dtilec{{3,{1,0},0},{3,{0,0},1}}} & 0 & 0
\\
0 & 0 & 0 & 0 &0 & \tikz[scale=0.4]{\dtilec{{0,{1,0},0},{4,{0,0},1}}} & 0 & \tikz[scale=0.4]{\dtilec{{0,{1,0},0},{0,{0,0},1}}} & 0
\\
0 & 0 & 0 & 0 &0 & 0 & 0 & 0 & 0
\end{array}
\right)_{a,b}
\end{align*}
\end{Small}
\begin{Small}
\begin{align*}
\rb_{a,b}(z)
=
\left(
\begin{array}{ccc|ccc|ccc}
1-z & 0 & 0 & 0 &1 & 0 & 0 & 0 & 1
\\
0 & 1 & 0 & 0 &0 & 0 & 0 & 0 & 0
\\
0 & 0 & 1 & 0 &0 & 0 & 0 & 0 & 0
\\
\hline
0 & 0 & 0 & 1 &0 & 0 & 0 & 0 & 0
\\
z & 0 & 0 & 0 &0 & 0 & 0 & 0 & 0
\\
0 & 0 & 0 & 0 &0 & 0 & 0 & 0 & 0
\\
\hline
0 & 0 & 0 & 0 &0 & 0 & 1 & 0 & 0
\\
0 & 0 & 0 & 0 &0 & 0 & 0 & 1 & 0
\\
z & 0 & 0 & 0 & -1 & 0 & 0 & 0 & 0
\end{array}
\right)_{a,b}
=
\left(
\begin{array}{ccc|ccc|ccc}
\tikz[scale=0.4]{\dtileb{{6,{0,1},0},{6,{0,0},1}}} & 0 & 0 & 0 &\tikz[scale=0.4]{\dtileb{{3,{0,1},0},{0,{0,0},1}}} & 0 & 0 & 0 & \tikz[scale=0.4]{\dtileb{{1,{0,1},0},{4,{0,0},1}}}
\\
0 & \tikz[scale=0.4]{\dtileb{{0,{0,1},0},{0,{0,0},1}}} & 0 & 0 &0 & 0 & 0 & 0 & 0
\\
0 & 0 & \tikz[scale=0.4]{\dtileb{{4,{0,1},0},{4,{0,0},1}}} & 0 &0 & 0 & 0 & 0 & 0
\\
\hline
0 & 0 & 0 & \tikz[scale=0.4]{\dtileb{{3,{0,1},0},{3,{0,0},1}}} &0 & 0 & 0 & 0 & 0
\\
\tikz[scale=0.4]{\dtileb{{0,{0,1},0},{3,{0,0},1}}} & 0 & 0 & 0 &0 & 0 & 0 & 0 & 0
\\
0 & 0 & 0 & 0 &0 & 0 & 0 & 0 & 0
\\
\hline
0 & 0 & 0 & 0 &0 & 0 & \tikz[scale=0.4]{\dtileb{{1,{0,1},0},{1,{0,0},1}}} & 0 & 0
\\
0 & 0 & 0 & 0 &0 & 0 & 0 & \tikz[scale=0.4]{\dtileb{{2,{0,1},0},{2,{0,0},1}}} & 0
\\
\tikz[scale=0.4]{\dtileb{{4,{0,1},0},{1,{0,0},1}}} & 0 & 0 & 0 & \tikz[scale=0.4]{\dtileb{{8,{0,1},0},{2,{0,0},1}}} & 0 & 0 & 0 & 0
\end{array}
\right)_{a,b}
\end{align*}
\end{Small}

Alongside each $R$-matrix we have written a corresponding matrix whose entries are rhombi. The meaning of these pictorial entries, and how they lead to a combinatorial version of the Yang--Baxter equation \eqref{rank2yb}, will be explained in the next subsection.

It is not necessary to give a detailed derivation of \eqref{rank2yb}, since it can be directly verified. Let us simply remark that it can be obtained as a certain limit of the $U_q(\widehat{\mathfrak{sl}(3)})$ Yang--Baxter equation, where one of the participating vector spaces ($V_b$) is now an {\it anti-fundamental} representation space. This ultimately explains the fact that the $R$-matrices in \eqref{rank2yb} have different non-zero entries, and the inverted spectral parameter dependence of two of the matrices.
  
\subsection{Graphical version of Yang--Baxter equation}

The Yang--Baxter equation \eqref{rank2yb} has a natural graphical representation, which will be our main tool in proving all theorems listed in Section \ref{sec:results}. We have already hinted at this graphical form by assigning each of the three $R$-matrices in \eqref{rank2yb} a rhombus superscript. By joining the three rhombi in the obvious way, we arrive at
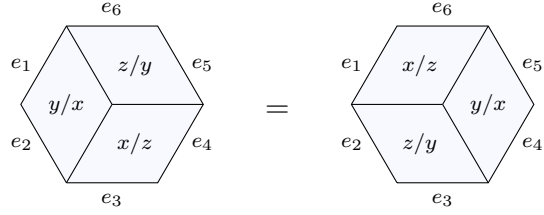
\begin{prop}
\label{prop:ybe}
\[
\begin{tikzpicture}[scale=1.2,baseline=0.5*1.73205cm]
\node[text centered,above,left] at (-0.25,0.75*1.73205) {\tiny $e_1$};
\node[text centered,below,left] at (-0.25,0.25*1.73205) {\tiny $e_2$};
\node[text centered,below] at (0.5,0) {\tiny $e_3$};
\node[text centered,above] at (0.5,1.73205) {\tiny $e_6$};
\node[text centered,above,right] at (1.25,0.75*1.73205) {\tiny $e_5$};
\node[text centered,below,right] at (1.25,0.25*1.73205) {\tiny $e_4$};
\filldraw[fill=c4!15!white] (0,0) -- (1,0) -- (1.5,0.5*1.73205) -- (1,1.73205) -- (0,1.73205) -- (-0.5,0.5*1.73205) -- (0,0); 
\draw (0,0) -- (0.5,0.5*1.73205); 
\draw (0,1.73205) -- (0.5,0.5*1.73205); 
\draw (0.5,0.5*1.73205) -- (1.5,0.5*1.73205); 
\node[text centered] at (0,0.5*1.73205) {\tiny $y/x$};
\node[text centered] at (0.75,0.25*1.73205) {\tiny $x/z$};
\node[text centered] at (0.75,0.75*1.73205) {\tiny $z/y$};
\end{tikzpicture}
\quad
=
\quad
\begin{tikzpicture}[scale=1.2,baseline=0.5*1.73205cm]
\node[text centered,above,left] at (-0.25,0.75*1.73205) {\tiny $e_1$};
\node[text centered,below,left] at (-0.25,0.25*1.73205) {\tiny $e_2$};
\node[text centered,below] at (0.5,0) {\tiny $e_3$};
\node[text centered,above] at (0.5,1.73205) {\tiny $e_6$};
\node[text centered,above,right] at (1.25,0.75*1.73205) {\tiny $e_5$};
\node[text centered,below,right] at (1.25,0.25*1.73205) {\tiny $e_4$};
\filldraw[fill=c4!15!white] (0,0) -- (1,0) -- (1.5,0.5*1.73205) -- (1,1.73205) -- (0,1.73205) -- (-0.5,0.5*1.73205) -- (0,0); 
\draw (1,0) -- (0.5,0.5*1.73205);
\draw (1,1.73205) -- (0.5,0.5*1.73205);
\draw (-0.5,0.5*1.73205) -- (0.5,0.5*1.73205);
\node[text centered] at (1,0.5*1.73205) {\tiny $y/x$};
\node[text centered] at (0.25,0.25*1.73205) {\tiny $z/y$};
\node[text centered] at (0.25,0.75*1.73205) {\tiny $x/z$};
\end{tikzpicture}
\]
\end{prop}
The six external edges of the resulting hexagon are set to some definite states $\{e_1,\dots,e_6\}$, where each $e_j$ can take values in $(1,2,3)$, which are held fixed on both sides of the equation. The three internal edges are summed over all possible values $(1,2,3)$. This relation therefore encodes all $3^6$ components of the matrix equation \eqref{rank2yb}. In addition, we have indicated the spectral parameter dependence of each rhombus.

Alternatively, one can consider each edge state (external or internal) to take values in some subset of (empty, red, green, red+green). We construct a dictionary between numeric and coloured labellings, which varies according to the rotation of an edge:
\begin{itemize}
\item On horizontal edges, $(1,2,3)=(\text{red},\text{green},\text{empty})$.
\item On 60/240 degree edges, $(1,2,3) = (\text{red+green},\text{empty},\text{green})$.
\item On 120/300 degree edges, $(1,2,3) = (\text{empty},\text{red+green},\text{red})$.
\end{itemize}
Using this dictionary, one is then able to pair each component $(i_a,j_a | i_b,j_b)$ of an $R$-matrix with a tile having edge states $(i_a,j_a|i_b,j_b)$ (note, however, that the internal decoration of the tiles is purely aesthetic). Doing this, one obtains the graphical version of the $R$-matrices in the previous subsection. For convenience, we tabulate the Boltzmann weights of all rhombi below:
\begin{align}\label{eq:table}
\begin{tabular}{|c|c|c|c|c|c|c|c|c|c|c|}
\hline &&&&&&&&&&
\\
\dtileb{{3,{0,1},0},{3,{0,0},1}} 
&
\dtileb{{0,{0,1},0},{3,{0,0},1}}
&
\dtileb{{3,{0,1},0},{0,{0,0},1}}
&
\dtileb{{0,{0,1},0},{0,{0,0},1}}
&
\dtileb{{6,{0,1},0},{6,{0,0},1}}
&
\dtileb{{1,{0,1},0},{1,{0,0},1}}
&
\dtileb{{2,{0,1},0},{2,{0,0},1}}
&
\dtileb{{4,{0,1},0},{1,{0,0},1}}
&
\dtileb{{1,{0,1},0},{4,{0,0},1}}
&
\dtileb{{4,{0,1},0},{4,{0,0},1}}
&
\dtileb{{8,{0,1},0},{2,{0,0},1}}
\\
$1$
&
$z$
&
$1$
&
$1$
&
$1-z$
&
$1$
&
$1$
&
$z$
&
$1$
&
$1$
&
$-1$
\\
\hline &&&&&&&&&&
\\
\dtile{{2,{0,0},0},{2,{0,0},1}}
&
\dtile{{0,{0,0},0},{2,{0,0},1}}
&
\dtile{{2,{0,0},0},{0,{0,0},1}}
&
\dtile{{0,{0,0},0},{0,{0,0},1}}
&
\dtile{{7,{0,0},0},{7,{0,0},1}}
&
\dtile{{1,{0,0},0},{1,{0,0},1}}
&
\dtile{{4,{0,0},0},{4,{0,0},1}}
&
\dtile{{3,{0,0},0},{1,{0,0},1}}
&
\dtile{{1,{0,0},0},{3,{0,0},1}}
&
\dtile{{3,{0,0},0},{3,{0,0},1}}
&
\dtile{{8,{0,0},0},{4,{0,0},1}}
\\
$1$
&
$1$
&
$z$
&
$1$
&
$1-z$
&
$1$
&
$1$
&
$z$
&
$1$
&
$1$
&
$-z$
\\
\hline &&&&&&&&&&
\\
\dtilec{{4,{1,0},0},{4,{0,0},1}}
&
\dtilec{{0,{1,0},0},{4,{0,0},1}}
&
\dtilec{{4,{1,0},0},{0,{0,0},1}}
&
\dtilec{{0,{1,0},0},{0,{0,0},1}}
&
\dtilec{{5,{1,0},0},{5,{0,0},1}}
&
\dtilec{{1,{1,0},0},{1,{0,0},1}}
&
\dtilec{{3,{1,0},0},{3,{0,0},1}}
&
\dtilec{{2,{1,0},0},{1,{0,0},1}}
&
\dtilec{{1,{1,0},0},{2,{0,0},1}}
&
\dtilec{{2,{1,0},0},{2,{0,0},1}}
&
\dtilec{{8,{1,0},0},{3,{0,0},1}}
\\
$1$
&
$1$
&
$z$
&
$1$
&
$1-z$
&
$1$
&
$1$
&
$1$
&
$z$
&
$1$
&
$-1$
\\
\hline
\end{tabular}
\end{align}

\begin{rmk}
Note that the weights are almost the same in the various rows of the table \eqref{eq:table}.
In fact, simple conjugations of the weights allow going from one row to another;
in particular, the modified weights \eqref{eq:modwei} are exactly the weights of the third row.
One could even, with an appropriate conjugation, make the three series of weights equal, thus rendering the model completely 120 degree rotationally invariant; 
but this would involve introducing cubic roots of $z$, which is rather cumbersome.
\end{rmk}

\subsection{Reduction to two colors}\label{sec:red2}
Finally, let us explain how the Grothendieck polynomials and their duals can be recovered from the three-state 
$R$-matrices in Section \ref{ssec:rank2}. We will only show how the reduction works in the case of $\rc(z)$, since this is the case which impacts on our subsequent proofs, but a similar reduction is also possible for $\ra(z)$ and $\rb(z)$.

The key to the reduction is the following, trivial observation. Let $\mathcal{T}$ be a tiling of the plane using the set of tiles $\{ \tikz[scale=0.5]{\dtilec{{4,{1,0},0},{4,{0,0},1}}} \}$. If no red particle appears on the boundary $\partial \mathcal{T}$ of the domain, the only tiles used in its interior are \tikz[scale=0.5]{\dtilec{{4,{1,0},0},{4,{0,0},1}}}, \tikz[scale=0.5]{\dtilec{{0,{1,0},0},{4,{0,0},1}}}, \tikz[scale=0.5]{\dtilec{{4,{1,0},0},{0,{0,0},1}}}, \tikz[scale=0.5]{\dtilec{{0,{1,0},0},{0,{0,0},1}}}, \tikz[scale=0.5]{\dtilec{{5,{1,0},0},{5,{0,0},1}}}. To see this, notice that the remaining six tiles in the set $\{ \tikz[scale=0.5]{\dtilec{{4,{1,0},0},{4,{0,0},1}}} \}$ all feature red lines, and since it is impossible to construct red loops, using of any of these tiles would necessarily produce red particles on $\partial \mathcal{T}$ by conservation arguments.

From this observation, we conclude that the Grothendieck polynomials and their duals are given by the partition functions
\begin{equation}\label{eq:red2}
G^{\lambda}(\x;\y)=\hskip-1cm
\begin{tikzpicture}[scale=0.7,rotate=90,baseline=3cm]
\begin{scope}[cm={sin(60),-cos(60),0,1,(0,0)}]
\filldraw[fill=c4!15!white,draw=black] (0,0) -- (0,4) -- (7,4) -- (7,0) -- (0,0);
\foreach\y in {1,...,3}{
\draw[thin] (0,\y) -- (7,\y);
}
\foreach\x in {1,...,6}{
\draw[thin] (\x,0) -- (\x,4);
}
\foreach\x in {0,...,6}{
\node[vertex=lightblue] at (\x+0.5,0) {};
}
\node[right] at (0.5,0) {$y_1$};
\node[right,rotate=60] at (3,-0.5) {$\ldots$};
\node[right] at (6.5,0) {$y_n$};
\draw[->] (0.5,4.5) --node[left] {$\lambda$} (6.5,4.5);
\node[vertex=lightblue] at (0.5,4) {};\node[vertex=lightblue] at (4.5,4) {};\node[vertex=lightblue] at (6.5,4) {};
\node[vertex=green] at (1.5,4) {};\node[vertex=green] at (2.5,4) {};\node[vertex=green] at (3.5,4) {};\node[vertex=green] at (5.5,4) {};
\foreach\y in {1,...,4}{
\node[vertex=green] at (0,\y-0.5) {};
\node[vertex=lightblue] at (7,\y-0.5) {};
}
\node[below] at (0,0.5) {$x_1$};
\node[below] at (-0.3,2) {$\ldots$};
\node[below] at (0,3.5) {$x_k$}; 
\end{scope}
\end{tikzpicture}
\quad
G_{\lambda}(\x;\y)=\hskip-1.2cm
\begin{tikzpicture}[scale=0.7,rotate=90,baseline=3cm]
\begin{scope}[cm={sin(60),-cos(60),0,1,(0,0)}]
\filldraw[fill=c4!15!white,draw=black] (0,0) -- (0,4) -- (7,4) -- (7,0) -- (0,0);
\foreach\y in {1,...,3}{
\draw[thin] (0,\y) -- (7,\y);
}
\foreach\x in {1,...,6}{
\draw[thin] (\x,0) -- (\x,4);
}
\foreach\x in {0,...,6}{
\node[vertex=lightblue] at (\x+0.5,4) {};
}
\node[right] at (0.5,0) {$y_1$};
\node[right,rotate=60] at (3,-0.5) {$\ldots$};
\node[right] at (6.5,0) {$y_n$};
\draw[->] (0.5,-1) --node[right] {$\lambda$} (6.5,-1);
\node[vertex=lightblue] at (0.5,0) {};\node[vertex=lightblue] at (2.5,0) {};\node[vertex=lightblue] at (5.5,0) {};
\node[vertex=green] at (1.5,0) {};\node[vertex=green] at (4.5,0) {};\node[vertex=green] at (3.5,0) {};\node[vertex=green] at (6.5,0) {};
\foreach\y in {1,...,4}{
\node[vertex=lightblue] at (0,\y-0.5) {};
\node[vertex=green] at (7,\y-0.5) {};
}
\node[below] at (0,0.5) {$x_1$};
\node[below] at (-0.3,2) {$\ldots$};
\node[below] at (0,3.5) {$x_k$}; 
\end{scope} 
\end{tikzpicture}
\end{equation}
where both regions are considered to be tilings by the entries of $\rc(z)$, and the spectral parameter of each tile is the ratio $x_i/y_j$ of the two variables which intersect it. Indeed, we can easily verify that these partition functions are equivalent to those studied in Section \ref{ssec:5v}, up to an affine transformation of the tiles (rotation by 90 degrees, followed by a 30 degree shearing).

\section{Proofs of the main theorems}
\label{sec:proofs}

In what follows, we use an abbreviated notation to describe the parametrization of the weights of puzzles:
we write the name of the alphabet, the arrow indicating as usual the ordering of the indices. For example, the figures of
Theorems~\ref{thm:equiv} and \ref{thm:MS} become respectively
\[
c^{\lambda,\mu}_{\,\,\nu}(\y)=\triup{\lambda/->,\mu/->,\nu/<-}{\y}
\qquad
c^\nu_{\lambda,\mu}(\z;\y)=
\loz{\nu/->,\varnothing/->,\mu/<-,\lambda/<-}{\z}{\y}
\]

We begin with a technical result.

\subsection{Key lemma}
\begin{lem}\label{lem:split}
Consider an $n \times n$ lozenge-shaped puzzle framed by four Young diagrams $\lambda,\mu,\nu,\rho$: 
\begin{align}
\label{carreau}
\loz{\lambda/->,\mu/->,\nu/->,\rho/->}{\z}{\z}
\end{align}
The Young diagrams $\lambda$ and $\nu$ are specified by binary strings consisting of green particles and empty sites, while $\mu$ and $\rho$ are specified by binary strings of empty sites and red particles. All Young diagrams live in a $k \times (n-k)$ rectangle.

Assuming that the two alphabets of the puzzle coincide, as shown, the following results hold:
\begin{enumerate}
\item The puzzle can be expressed as a sum over a product of two triangle-shaped lozenges:
\[
\loz{\lambda/->,\mu/->,\nu/->,\rho/->}{\z}{\z}
=
\sum_\sigma
\triup{\lambda/->,\mu/->,\sigma/<-}{\z}
\tridown{\sigma/->,\nu/->,\rho/->}{\z}
\]
where $\sigma$ is a Young diagram encoded by a binary string of green and red particles.

\item If one of the top Young diagrams is empty, \ie\ (a) $\lambda=\varnothing$ or (b) $\mu=\varnothing$,
then the top half of the puzzle is frozen with weight $1$, and $\sigma=\mu$ in case (a), $\sigma=\lambda$ in case (b).

\item If one of the bottom Young diagrams is empty, \ie\ (a) $\rho=\varnothing$ or (b) $\nu=\varnothing$, then the bottom half vanishes unless $\sigma^\ast \vartriangleright \nu$ in case (a), $\sigma^\ast \vartriangleright \rho$ in case (b). For each such $\sigma$ there is a unique configuration of the bottom half, with weight
\begin{align*}
(-1)^{|\sigma^\ast|-|\rho|-|\nu|}\prod_{i\in\sigma} z_i \prod_{i=n-k+1}^n z_i^{-1}.
\end{align*}
\end{enumerate}
\end{lem}

\begin{proof}
Draw a horizontal line bisecting \eqref{carreau}, and consider the $n$ tiles which lie on it. Because the two alphabets coincide, the ratio of spectral parameters of every tile along this line is equal to 1. This means that the tile 
$\tikz[scale=0.5,baseline=-0.45cm]{\dtile{{7,{0,0},0},{7,{0,0},1}}}$ can never occur along the central horizontal line. This, in turn, implies that each site along the horizontal line can be unoccupied, occupied by a green or red particle, but never occupied by both.

Now consider the Young diagrams $\nu$ and $\rho$. $\nu$ contains $k$ green particles and $n-k$ empty sites, while $\rho$ contains $n-k$ red particles and $k$ empty sites. Since particles must be conserved, we conclude that in every possible configuration of \eqref{carreau}, a total of $n$ particles cross the bisecting horizontal line. This completely saturates the sites along the line, since double-occupation is forbidden. We conclude that the central horizontal line always produces a binary string of green and red particles, and (1) follows immediately.

For the proof of (2), observe that the set of tiles $\{ \tikz[scale=0.5,baseline=-0.45cm] {\dtile{{4,{0,0},0},{4,{0,0},1}}} \}$ are (collectively) invariant under reflection about their central vertical axis and the interchange red $\leftrightarrow$ green. Using this symmetry, one can easily show that the cases (a) and (b) are equivalent, so we focus on (a). Because $\lambda = \varnothing$, the first $k$ sites on the top left boundary are occupied by green particles, and the remainder are empty. Drawing only the top half of \eqref{carreau} we have the boundary conditions shown in Figure \ref{fig:1}(i).
\begin{figure}
\begin{tabular}{cccc}
\def\size{5}
\tikz[scale=0.9]{\equivpuzzle{}\begin{scope}[puz] 
\node[vertex=lightblue] at (0,0.5) {}; \node[vertex=lightblue] at (0,1.5) {};
\node[vertex=lightblue] at (0.5,0) {}; \node[vertex=lightblue] at (2.5,0) {}; \node[vertex=lightblue] at (4.5,0) {};
\node[vertex=green] at (0,4.5) {}; \node[vertex=green] at (0,3.5) {}; \node[vertex=green] at (0,2.5) {}; \node[vertex=red] at (1.5,0) {}; \node[vertex=red] at (3.5,0) {}; \end{scope}}
&\quad
\def\size{5}
\tikz[scale=0.9]{\equivpuzzle{{1,{0,0},0},{1,{0,0},1},{1,{1,0},0},{3,{1,0},1},{3,{2,0},0},{3,{2,0},1},{3,{3,0},0},{3,{3,0},1},{3,{4,0},0}}}
&\quad
\def\size{5}
\tikz[scale=0.9]{\equivpuzzle{{4,{0,0},0},{4,{0,0},1},{4,{1,0},0},{4,{1,0},1},{0,{2,0},0},{0,{2,0},1},{0,{3,0},0},{0,{3,0},1},{0,{4,0},0}}}
\\ \\ 
(i) &\quad (ii) &\quad (iii) 
\end{tabular}
\caption{The top half of \eqref{carreau} for $\lambda = \varnothing$.}
\label{fig:1}
\end{figure}

Now consider the first site on the boundary corresponding with $\mu$. {\bf 1.} If it is occupied by a red particle, the path traced by that particle is completely constrained: it must propagate diagonally downward, ultimately terminating at the first site of the horizontal boundary. Since that site cannot then also be occupied by a green particle, the $k$ green particles propagate horizontally one step to the right. The result is a shrinking of the triangular lozenge from size $n$ to $n-1$, with boundary conditions now featuring $k$ green particles and $n-k-1$ red particles, as shown in Figure \ref{fig:1}(ii). {\bf 2.} If it is empty, we can conclude that the first site of the horizontal boundary cannot be occupied by a red particle. Indeed, were that site occupied by a red particle, it would be impossible to tile the triangular lozenge using the available tiles. We conclude that the first site of the horizontal boundary is occupied by a green particle, forcing all $k$ green particles to propagate one step diagonally downward. The result is a shrinking of the triangular lozenge from size $n$ to $n-1$, with boundary conditions now featuring $k-1$ green particles and $n-k$ red particles, as shown in Figure \ref{fig:1}(iii).

In either of the cases {\bf 1} or {\bf 2}, we arrive at a smaller triangular region with boundary conditions that allow us to repeat our reasoning. After $n$ iterations of this argument, the triangular lozenge will be completely frozen, with all red lines propagating diagonally downward, and green particles occupying all ``spare'' sites on the horizontal boundary: 
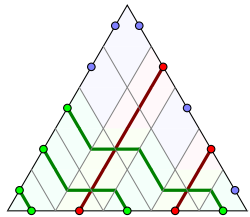
\begin{center}
\def\size{5}
\tikz[scale=0.9]{\equivpuzzle{{4,{0,0},0},{4,{0,0},1},{4,{1,0},0},{4,{1,0},1},{0,{2,0},0},{0,{2,0},1},{0,{3,0},0},{0,{3,0},1},{0,{4,0},0},{1,{0,1},0},{1,{0,1},1},{1,{1,1},0},{3,{1,1},1},{3,{2,1},0},{3,{2,1},1},{3,{3,1},0},
{4,{0,2},0},{4,{0,2},1},{0,{1,2},0},{0,{1,2},1},{0,{2,2},0},{1,{0,3},0},{3,{0,3},1},{3,{1,3},0},{0,{0,4},0}}}
\end{center}
The total Boltzmann weight is just 1 (since the tiles \tikz[scale=0.5,baseline=-0.45cm]{\dtile{{2,{0,0},0},{0,{0,0},1}}} and \tikz[scale=0.5,baseline=-0.45cm]{\dtile{{8,{0,0},0},{4,{0,0},1}}} can never occur, and \tikz[scale=0.5,baseline=-0.45cm]{\dtile{{3,{0,0},0},{1,{0,0},1}}} can only occur along the horizontal line, where the ratio of spectral parameters is always 1), while the resulting Young diagram along the horizontal boundary is 
$\sigma=\mu$. This completes the proof of (2)(a). 

The proof of (3) is similar, but extra care must be taken because of the tile \tikz[scale=0.5,baseline=-0.45cm]{\dtile{{8,{0,0},0},{4,{0,0},1}}}. Indeed, the presence of this tile prevents the set $\{ \tikz[scale=0.5,baseline=-0.45cm] {\dtile{{4,{0,0},0},{4,{0,0},1}}} \}$  from being 180 degree rotationally invariant, so we cannot simply treat the bottom half of the lozenge \eqref{carreau} in the same way as we treated the top half. As before, we focus on (3)(a), since (3)(b) can be deduced by reflective symmetry arguments. Because $\rho = \varnothing$, the final $n-k$ sites along the bottom left boundary are occupied by red particles, and the remainder are empty. Drawing only the bottom half of \eqref{carreau}, we have boundary conditions of the type shown in Figure \ref{fig:2}(i).
\begin{figure}
\begin{tabular}{cccc}
\def\size{5}
\tikz[scale=-0.9]{\equivpuzzle{}\begin{scope}[puz] 
\node[vertex=lightblue] at (0,0.5) {}; \node[vertex=lightblue] at (0,3.5) {};
\node[vertex=lightblue] at (0.5,0) {}; \node[vertex=lightblue] at (1.5,0) {}; \node[vertex=lightblue] at (2.5,0) {};
\node[vertex=green] at (0,4.5) {}; \node[vertex=green] at (0,2.5) {}; \node[vertex=green] at (0,1.5) {}; \node[vertex=red] at (4.5,0) {}; \node[vertex=red] at (3.5,0) {}; \end{scope}}
&\quad
\def\size{5}
\tikz[scale=-0.9]{\equivpuzzle{
{0,{0,0},0},{0,{0,0},1},{0,{0,1},0},{0,{0,1},1},{0,{0,2},0},{2,{0,2},1},{2,{0,3},0},{2,{0,3},1},{2,{0,4},0}}}
&\quad
\def\size{5}
\tikz[scale=-0.9]{\equivpuzzle{
{4,{0,0},0},{4,{0,0},1},{4,{0,1},0},{4,{0,1},1},{4,{0,2},0},{4,{0,2},1},{1,{0,3},0},{1,{0,3},1},{1,{0,4},0}}}
&\quad
\def\size{5}
\tikz[scale=-0.9]{\equivpuzzle{
{4,{0,0},0},{4,{0,0},1},{4,{0,1},0},{4,{0,1},1},{4,{0,2},0},{8,{0,2},1},{2,{0,3},0},{2,{0,3},1},{2,{0,4},0},{0,{1,0},0},{0,{1,0},1},{0,{1,1},0},{0,{1,1},1},{3,{1,2},0},{1,{1,2},1},{1,{1,3},0}}}
\\ \\
(i) &\quad (ii) &\quad (iii) &\quad (iv)
\end{tabular}
\caption{The bottom half of \eqref{carreau} for $\rho=\varnothing$.}
\label{fig:2}
\end{figure}

Now consider the last site of the boundary corresponding with $\nu$. {\bf 1.} If it is occupied by a green particle, that particle is forced to evolve to the first site of the horizontal boundary, without deviation. This constrains the red particles to each take one step in the horizontal direction, and we obtain a triangular region of size $n-1$ bordered by $n-k$ red and $k-1$ green particles, as shown in Figure \ref{fig:2}(ii).
The Boltzmann weight of the frozen strip is $z_1/z_{n-k+1}$, due to the appearance of the tile  \tikz[scale=0.5,baseline=-0.45cm]{\dtile{{2,{0,0},0},{0,{0,0},1}}}. {\bf 2.} If it is unoccupied, two situations may arise. The first is that all red particles propagate one step diagonally upwards, as in Figure \ref{fig:2}(iii). This shrinks the size of the triangular region by one, and the resulting region is bordered by $n-k-1$ red and $k$ green particles. The Boltzmann weight of the frozen strip is 1. The second possibility is the configuration of Figure \ref{fig:2}(iv), which shrinks the size of the triangular region by {\it two,} the resulting region being bordered by $n-k-1$ red and $k-1$ green particles. This configuration is only possible if the second-last site of the $\nu$ boundary is occupied by a green particle. In this case the Boltzmann weight of the frozen strip is $-z_1/z_{n-k+1}$, due to the tile \tikz[scale=0.5,baseline=-0.45cm]{\dtile{{8,{0,0},0},{4,{0,0},1}}}.

One is able to iterate this process, successively reducing the size of the triangular region. In view of the two possibilities listed in {\bf 2} it will not, in general, be completely frozen (in contrast to part (2) of the lemma). In fact there will be a total of $2^{g(\nu)}$ configurations; $g(\nu)$ being the number of green particles in $\nu$ which are followed by an empty site. For each of these $g(\nu)$ particles there is the possibility of hopping one step to the left, or not, in the transition to the state along the horizontal line:
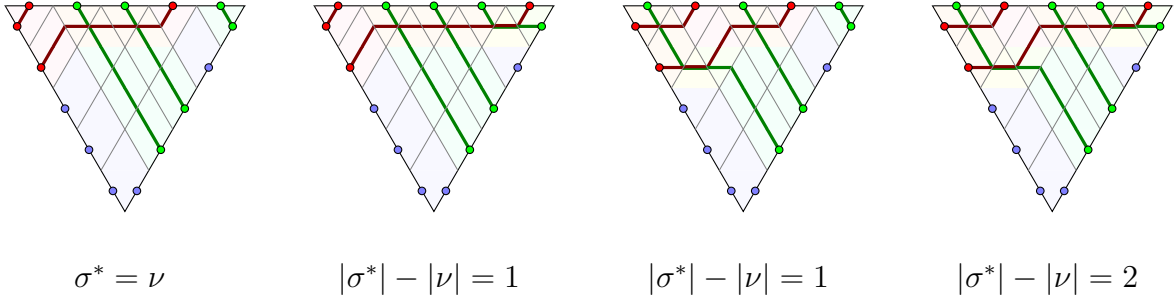
\begin{center}
\begin{tabular}{cccc}
\def\size{5}
\tikz[scale=-0.9]{\equivpuzzle{
{4,{0,0},0},{4,{0,0},1},{4,{0,1},0},{4,{0,1},1},{4,{0,2},0},{4,{0,2},1},{1,{0,3},0},{1,{0,3},1},{1,{0,4},0},{0,{1,0},0},{0,{1,0},1},{0,{1,1},0},{0,{1,1},1},{0,{1,2},0},{2,{1,2},1},{2,{1,3},0},{0,{2,0},0},{0,{2,0},1},{0,{2,1},0},{2,{2,1},1},{2,{2,2},0},{4,{3,0},0},{4,{3,0},1},{1,{3,1},0},{0,{4,0},0}}}
&\quad
\def\size{5}
\tikz[scale=-0.9]{\equivpuzzle{
{4,{0,0},0},{4,{0,0},1},{4,{0,1},0},{4,{0,1},1},{4,{0,2},0},{4,{0,2},1},{1,{0,3},0},{1,{0,3},1},{1,{0,4},0},{0,{1,0},0},{0,{1,0},1},{0,{1,1},0},{0,{1,1},1},{0,{1,2},0},{2,{1,2},1},{2,{1,3},0},{0,{2,0},0},{0,{2,0},1},{0,{2,1},0},{2,{2,1},1},{2,{2,2},0},{4,{3,0},0},{8,{3,0},1},{2,{3,1},0},{3,{4,0},0}}}
&\quad
\def\size{5}
\tikz[scale=-0.9]{\equivpuzzle{
{4,{0,0},0},{4,{0,0},1},{4,{0,1},0},{4,{0,1},1},{4,{0,2},0},{8,{0,2},1},{2,{0,3},0},{2,{0,3},1},{2,{0,4},0},{0,{1,0},0},{0,{1,0},1},{0,{1,1},0},{0,{1,1},1},{3,{1,2},0},{1,{1,2},1},{1,{1,3},0},{0,{2,0},0},{0,{2,0},1},{0,{2,1},0},{2,{2,1},1},{2,{2,2},0},{4,{3,0},0},{4,{3,0},1},{1,{3,1},0},{0,{4,0},0}}}
&\quad
\def\size{5}
\tikz[scale=-0.9]{\equivpuzzle{
{4,{0,0},0},{4,{0,0},1},{4,{0,1},0},{4,{0,1},1},{4,{0,2},0},{8,{0,2},1},{2,{0,3},0},{2,{0,3},1},{2,{0,4},0},{0,{1,0},0},{0,{1,0},1},{0,{1,1},0},{0,{1,1},1},{3,{1,2},0},{1,{1,2},1},{1,{1,3},0},{0,{2,0},0},{0,{2,0},1},{0,{2,1},0},{2,{2,1},1},{2,{2,2},0},{4,{3,0},0},{8,{3,0},1},{2,{3,1},0},{3,{4,0},0}}}
\\ \\
$\sigma^\ast=\nu$ &\quad $|\sigma^\ast|-|\nu|=1$ &\quad $|\sigma^\ast|-|\nu|=1$ 
&\quad $|\sigma^\ast|-|\nu|=2$
\end{tabular}
\end{center}
It is now easy to conclude that 
\begin{align*}
\tridown{\sigma/->,\nu/->,\varnothing/->}{\z}
=
(-1)^{|\sigma^\ast|-|\nu|}\prod_{i\in\sigma} z_i \prod_{i=n-k+1}^n z_i^{-1}
\end{align*}
if $\sigma^{*} \vartriangleright \nu$, and vanishes otherwise. This proves (3)(a).

\end{proof}

\subsection{Proof of main theorems for $n$ large enough}
\newcommand\bl[1]{\textrm{\color{blue}#1}}
\def\k{3}\def\n{5}\def\spc{0.5}\pgfmathtruncatemacro\kk{\n-\k}
\def\puzzlescale{1}
We now proceed with the proof of the various theorems, modulo a hypothesis on $n$ which will be made clear below.
We consider the following identity, obtained by repeated application of the Yang--Baxter equation (Prop.~\ref{prop:ybe}):
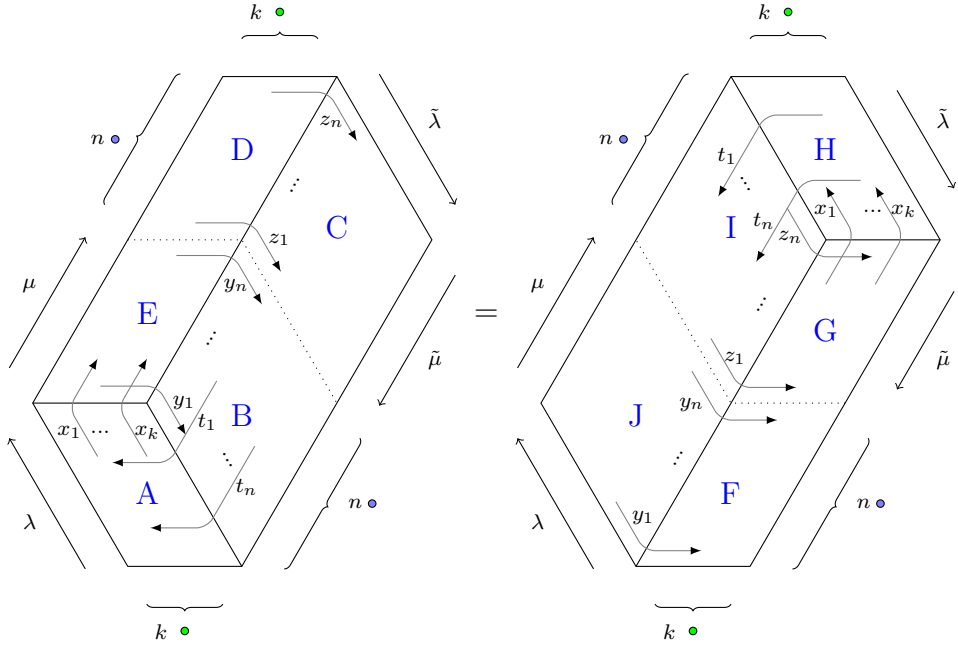
\begin{equation*}
\begin{tikzpicture}[scale=0.5,baseline=0]
\draw (-\n*0.25-\k*0.5,-\n*1.299) coordinate (A) -- ++(120:\n) coordinate (B) -- ++(60:2*\n) coordinate (C) -- ++(0:\k) coordinate (D) -- ++(-60:\n) coordinate (E) -- ++(-120:2*\n) coordinate (F) -- cycle;
\draw (B) -- ++(0:\k) coordinate (G) -- (D); \draw (F) -- (G);
\draw[->] (A) ++(210:1) ++(120:\spc) -- node[below left] {$\ss\lambda$} ++(120:\n-2*\spc);
\draw[->] (B) ++(150:1) ++(60:\spc) -- node[above left] {$\ss\mu$} ++ (60:\n-2*\spc);
\draw[decorate,decoration={brace,mirror}] (C) ++(150:1) ++(240:\spc) -- node[left=3mm,label={[label distance=3.5mm]180:$\ss n$},vertex=lightblue] {} ++(240:\n-2*\spc);
\draw[decorate,decoration=brace] (C) ++(90:1) ++(0:\spc) -- node[above=3mm,label={[label distance=1.5mm]125:$\ss k$},vertex=green] {} ++(0:\k-2*\spc);
\draw[->] (D) ++(30:1) ++(-60:\spc) -- node[above right] {$\ss\tilde\lambda$} ++(-60:\n-2*\spc); 
\draw[->] (E) ++(-30:1) ++(-120:\spc) -- node[below right] {$\ss\tilde\mu$} ++(-120:\n-2*\spc);
\draw[decorate,decoration={brace,mirror}] (F) ++(-30:1) ++(60:\spc) -- node[right=6mm,label={[label distance=2mm]0:$\ss n$},vertex=lightblue] {} ++(60:\n-2*\spc);
\draw[decorate,decoration=brace] (F) ++(-90:1) ++(180:\spc) -- node[below=3mm,label={[label distance=1.5mm]235:$\ss k$},vertex=green] {} ++(180:\k-2*\spc);
\draw[dotted] (B) ++(60:\n) -- ++(0:\k) -- ++(-60:\n);
\begin{scope}[rounded corners]
\draw[-latex,draw=gray] (A) ++(80:1.2) ++(120:2) -- node[left=-1.2mm] {$\ss x_1$} node[right=-0.5mm] {$\ss\ldots$} ++(120:1.5) -- ++(60:1.5);
\draw[-latex,draw=gray] (A) ++(80:1.2) ++(0:1.3) ++(120:2) -- node[right=-1.2mm] {$\ss x_k$} ++(120:1.5) -- ++(60:1.5);
\draw[-latex,draw=gray] (B) ++(30:0.9) ++(0:1) -- ++(0:1.5) -- node[above right=-1.3mm] {$\ss y_1$} ++(-60:1.5);
\path (B) ++(20:5) node[rotate=60] {$\ss\ldots$};
\draw[-latex,draw=gray] (B) ++(30:0.9) ++(60:4) ++(0:1) -- ++(0:1.5) -- node[below left=-1.6mm] {$\ss y_n$} ++(-60:1.5);
\draw[-latex,draw=gray] (B) ++(30:0.9) ++(60:5) ++(0:1) -- ++(0:1.5) -- node[above right=-1.2mm] {$\ss z_1$} ++(-60:1.5);
\path (B) ++(40:9) node[rotate=60] {$\ss\ldots$};
\draw[-latex,draw=gray] (B) ++(30:0.9) ++(60:9) ++(0:1) -- ++(0:1.5) -- node[below left=-1.6mm] {$\ss z_n$} ++(-60:1.5);
\draw[-latex,draw=gray] (D) ++(260:3.8) ++(-120:5) -- node[right=-0.8mm] {$\ss t_1$} ++(240:2.5) -- ++(180:1.5);
\draw[-latex,draw=gray] (D) ++(260:3.8) ++(-120:5) ++(-60:2) -- node[right=-0.8mm] {$\ss t_n$} node[above left=0.5mm,rotate=-60] {$\ss\ldots$} ++(240:2.5) -- ++(180:1.5);
\end{scope}
\path (A) ++(75:2) node {\bl A}; 
\path (F) ++(90:4) node {\bl B}; 
\path (D) ++(-90:4) node {\bl C}; 
\path (C) ++(-75:2) node {\bl D}; 
\path (C) ++(240:\n)  ++(-75:2) node {\bl E}; 
\end{tikzpicture}
=
\begin{tikzpicture}[scale=-0.5,baseline=0]
\draw (-\n*0.25-\k*0.5,-\n*1.299) coordinate (A) -- ++(120:\n) coordinate (B) -- ++(60:2*\n) coordinate (C) -- ++(0:\k) coordinate (D) -- ++(-60:\n) coordinate (E) -- ++(-120:2*\n) coordinate (F) -- cycle;
\draw (B) -- ++(0:\k) coordinate (G) -- (D); \draw (F) -- (G);
\draw[->] (A) ++(210:1) ++(120:1) -- node[above right] {$\ss\tilde\lambda$} ++(120:\n-2);
\draw[->] (B) ++(150:1) ++(60:1) -- node[below right] {$\ss\tilde\mu$} ++ (60:\n-2);
\draw[decorate,decoration={brace,mirror}] (C) ++(150:1) ++(240:\spc) -- node[right=6mm,label={[label distance=2mm]0:$\ss n$},vertex=lightblue] {} ++(240:\n-2*\spc);
\draw[decorate,decoration=brace] (C) ++(90:1) ++(0:\spc) -- node[below=3mm,label={[label distance=1.5mm]235:$\ss k$},vertex=green] {} ++(0:\k-2*\spc);
\draw[->] (D) ++(30:1) ++(-60:\spc) -- node[below left] {$\ss\lambda$} ++(-60:\n-2*\spc); 
\draw[->] (E) ++(-30:1) ++(-120:\spc) -- node[above left] {$\ss\mu$} ++(-120:\n-2*\spc);
\draw[decorate,decoration={brace,mirror}] (F) ++(-30:1) ++(60:\spc) -- node[left=3mm,label={[label distance=3.5mm]180:$\ss n$},vertex=lightblue] {} ++(60:\n-2*\spc);
\draw[decorate,decoration=brace] (F) ++(-90:1) ++(180:\spc) -- node[above=3mm,label={[label distance=1.5mm]125:$\ss k$},vertex=green] {} ++(180:\k-2*\spc);
\draw[dotted] (B) ++(60:\n) -- ++(0:\k) -- ++(-60:\n);
\begin{scope}[rounded corners]
\draw[latex-,draw=gray] (A) ++(80:1.2) ++(120:2) -- node[right=-1.2mm] {$\ss x_k$} node[left=-0.5mm] {$\ss\ldots$} ++(120:1.5) -- ++(60:1.5);
\draw[latex-,draw=gray] (A) ++(80:1.2) ++(0:1.3) ++(120:2) -- node[left=-1.2mm] {$\ss x_1$} ++(120:1.5) -- ++(60:1.5);
\draw[latex-,draw=gray] (B) ++(30:0.9) ++(0:1) -- ++(0:1.5) -- node[below left=-1.6mm] {$\ss z_n$} ++(-60:1.5);
\path (B) ++(20:5) node[rotate=60] {$\ss\ldots$};
\draw[latex-,draw=gray] (B) ++(30:0.9) ++(60:4) ++(0:1) -- ++(0:1.5) -- node[above right=-1.6mm] {$\ss z_1$} ++(-60:1.5);
\draw[latex-,draw=gray] (B) ++(30:0.9) ++(60:5) ++(0:1) -- ++(0:1.5) -- node[below left=-1.2mm] {$\ss y_n$} ++(-60:1.5);
\path (B) ++(40:9) node[rotate=60] {$\ss\ldots$};
\draw[latex-,draw=gray] (B) ++(30:0.9) ++(60:9) ++(0:1) -- ++(0:1.5) -- node[above right=-1.2mm] {$\ss y_1$} ++(-60:1.5);
\draw[latex-,draw=gray] (D) ++(260:3.8) ++(-120:5) -- node[left=-0.8mm] {$\ss t_n$} ++(240:2.5) -- ++(180:1.5);
\draw[latex-,draw=gray] (D) ++(260:3.8) ++(-120:5) ++(-60:2) -- node[left=-0.8mm] {$\ss t_1$} node[below right=0.5mm,rotate=-60] {$\ss\ldots$} ++(240:2.5) -- ++(180:1.5);
\end{scope}
\path (A) ++(75:2) node {\bl H}; 
\path (F) ++(90:4) node {\bl I}; 
\path (D) ++(-90:4) node {\bl J}; 
\path (C) ++(-75:2) node {\bl F}; 
\path (C) ++(240:\n)  ++(-75:2) node {\bl G}; 
\end{tikzpicture}
\end{equation*}
Here $\lambda,\tilde\lambda,\mu,\tilde\mu$ are four Young diagrams inside the $k\times (n-k)$ rectangle.

We now discuss the behavior of red and green lines inside these hexagons by subdividing them into regions, labelled
from \bl{A} to \bl{J}. Note that examples of configurations are given in Sect.~\ref{sec:firstcase}--\ref{sec:lastcase}.

There are four ``frozen'' regions, namely, \bl{A}, \bl{B}, \bl{H}, \bl{I}, in which by inspection, lines can only go
straight: green lines go north-west from the boundary in \bl{A} while red lines go east in \bl{A} and then north-east in \bl{B}; and the same up to a 180 degree rotation in \bl{H} and \bl{I}.

Next, there are four ``two-colour'' regions, namely \bl{D}, \bl{E}, \bl{F}, \bl{G}, in which only green lines occur.
As in \cite{z-j}, one needs to carefully analyse regions \bl{D} and \bl{F}. Let us consider for example \bl{F}
(\bl{D} can be treated identically by rotating the argument 180 degrees). We want to make sure that the green lines
starting at the bottom exit through the left side of \bl{F}. The rightmost green line must end at location $k+w(\mu)$,
if we number the north-west boundary from bottom to top.
This means that the location (counted from bottom to top) of the rightmost green line when it crosses
the diagonal line starting at the bottom of the junction of \bl{J} and \bl{F} cannot exceed $k+w(\mu)$ 
plus the number of steps to the left it can make in region \bl{J}. Now there are two types of such steps: 
either (i) without using
the $K$-tile, which means the red lines go straight through that green line; but only the $w(\lambda)+h(\tilde\lambda)$ bottommost red lines are allowed to make
such straight crossings, because of their starting/endpoints; or (ii) using the $K$-tile, which is innocuous for
red particles; however, each $K$-tile is always surrounded by the same three neighboring triangles,
namely $\vcenter{\hbox{\def\size{2}\tikz[scale=-0.5]{\puzzle{{4,{0,0},0},{8,{0,0},1},{2,{0,1},0},{3,{1,0},0}}}}}$,
so that a given green line can cross at most $n/2$ such tiles in region \bl{J}.
The result is that the location in question is at most $k+w(\mu)+w(\lambda)+h(\tilde\lambda)+n/2$. Let us therefore
assume in what follows that $n\geq 2(k+w(\mu)+w(\lambda)+h(\tilde\lambda))$ (we do not claim this bound to be optimal, but any
bound is sufficient for our purposes). Then this location is less or equal to $n$,
and the rightmost green line (and therefore, all others) exit through the left side of \bl{F}.

We conclude that regions
\bl{D}, \bl{E}, \bl{F} and \bl{G} are exactly of the type considered in Sect.~\ref{sec:red2}, \ie,
without any red lines and with boundary conditions of the form of \eqref{eq:red2}.
Therefore, the summation over configurations of these regions produce four Grothendieck polynomials.

Finally, putting everything together, and using the graphical notation that was defined at the beginning
of this section for regions \bl{C} and \bl{J}, we find the ``master identity''
\begin{equation}\label{eq:master}
\sum_\nu \loz{\nu/->,\tilde\lambda/->,\tilde\mu/->,\lambda/->}{\t}{\z} G_\nu(\x;\z) G^\mu(\x;\y)
=
\sum_{\tilde\nu} \loz{\mu/->,\tilde\lambda/->,\tilde\nu/->,\lambda/->}{\t}{\y} G_{\tilde\mu{}^\ast}(\x;\z) G^{\tilde\nu{}^\ast}(\x;\y)
\end{equation}
which is valid for $n$ large enough at fixed $k,\lambda,\mu,\tilde\lambda,\tilde\mu$.

The proofs now consist of specializing this identity to four cases:

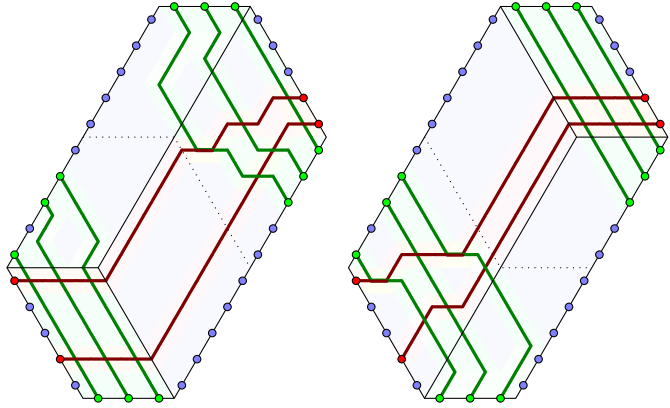
\begin{samepage}
\subsubsection{$\tilde\lambda=\tilde\mu=\varnothing$, $\t=\z$.}\label{sec:firstcase}
\begin{center}
\begin{tikzpicture}[scale=0.4,baseline=-3.2cm]
\begin{scope}[puz]
\rawpuzzle{{0, {0, 2}, 1}, {0, {0, 3}, 0}, {0, {0, 3}, 1}, {0, {0, 4}, 
  0}, {0, {0, 4}, 1}, {0, {0, 5}, 0}, {2, {0, 5}, 1}, {2, {0, 6}, 
  0}, {2, {0, 6}, 1}, {2, {0, 7}, 0}, {0, {0, 7}, 1}, {0, {1, 1}, 
  1}, {0, {1, 2}, 0}, {5, {1, 2}, 1}, {4, {1, 3}, 0}, {4, {1, 3}, 
  1}, {4, {1, 4}, 0}, {4, {1, 4}, 1}, {1, {1, 5}, 0}, {1, {1, 5}, 
  1}, {1, {1, 6}, 0}, {3, {1, 6}, 1}, {0, {1, 7}, 0}, {0, {1, 7}, 
  1}, {0, {2, 0}, 1}, {0, {2, 1}, 0}, {5, {2, 1}, 1}, {5, {2, 2}, 
  0}, {0, {2, 2}, 1}, {0, {2, 3}, 0}, {0, {2, 3}, 1}, {0, {2, 4}, 
  0}, {2, {2, 4}, 1}, {2, {2, 5}, 0}, {0, {2, 5}, 1}, {3, {2, 6}, 
  0}, {3, {2, 6}, 1}, {0, {2, 7}, 0}, {0, {2, 7}, 1}, {4, {3, 0}, 
  0}, {4, {3, 0}, 1}, {5, {3, 1}, 0}, {5, {3, 1}, 1}, {4, {3, 2}, 
  0}, {4, {3, 2}, 1}, {4, {3, 3}, 0}, {4, {3, 3}, 1}, {1, {3, 4}, 
  0}, {3, {3, 4}, 1}, {0, {3, 5}, 0}, {0, {3, 5}, 1}, {3, {3, 6}, 
  0}, {1, {3, 6}, 1}, {4, {3, 7}, 0}, {4, {3, 7}, 1}, {4, {4, 0}, 
  0}, {4, {4, 0}, 1}, {5, {4, 1}, 0}, {0, {4, 1}, 1}, {0, {4, 2}, 0},
  {0, {4, 2}, 1}, {0, {4, 3}, 0}, {2, {4, 3}, 1}, {8, {4, 4}, 
  0}, {4, {4, 4}, 1}, {4, {4, 5}, 0}, {4, {4, 5}, 1}, {1, {4, 6}, 
  0}, {1, {4, 6}, 1}, {4, {4, 7}, 0}, {4, {4, 7}, 1}, {4, {5, 0}, 
  0}, {4, {5, 0}, 1}, {4, {5, 1}, 0}, {4, {5, 1}, 1}, {4, {5, 2}, 
  0}, {4, {5, 2}, 1}, {1, {5, 3}, 0}, {1, {5, 3}, 1}, {4, {5, 4}, 
  0}, {4, {5, 4}, 1}, {4, {5, 5}, 0}, {4, {5, 5}, 1}, {1, {5, 6}, 
  0}, {1, {5, 6}, 1}, {4, {5, 7}, 0}, {4, {5, 7}, 1}, {4, {6, 0}, 
  0}, {4, {6, 0}, 1}, {4, {6, 1}, 0}, {4, {6, 1}, 1}, {4, {6, 2}, 
  0}, {4, {6, 2}, 1}, {1, {6, 3}, 0}, {1, {6, 3}, 1}, {4, {6, 4}, 
  0}, {4, {6, 4}, 1}, {4, {6, 5}, 0}, {4, {6, 5}, 1}, {1, {6, 6}, 
  0}, {1, {6, 6}, 1}, {4, {6, 7}, 0}, {4, {6, 7}, 1}, {4, {7, 0}, 
  0}, {4, {7, 0}, 1}, {4, {7, 1}, 0}, {4, {7, 1}, 1}, {4, {7, 2}, 
  0}, {4, {7, 2}, 1}, {1, {7, 3}, 0}, {1, {7, 3}, 1}, {4, {7, 4}, 
  0}, {4, {7, 4}, 1}, {4, {7, 5}, 0}, {4, {7, 5}, 1}, {1, {7, 6}, 
  0}, {1, {7, 6}, 1}, {4, {7, 7}, 0}, {4, {7, 7}, 1}, {4, {8, 0}, 0}, 
 {4, {8, 0}, 1}, {4, {8, 1}, 0}, {4, {8, 1}, 1}, {4, {8, 2}, 
  0}, {4, {8, 2}, 1}, {1, {8, 3}, 0}, {1, {8, 3}, 1}, {4, {8, 4}, 
  0}, {4, {8, 4}, 1}, {4, {8, 5}, 0}, {4, {8, 5}, 1}, {1, {8, 6}, 
  0}, {1, {8, 6}, 1}, {4, {8, 7}, 0}, {4, {8, 7}, 1}, {0, {9, 0}, 
  0}, {0, {9, 0}, 1}, {0, {9, 1}, 0}, {0, {9, 1}, 1}, {0, {9, 2}, 
  0}, {5, {9, 2}, 1}, {1, {9, 3}, 0}, {1, {9, 3}, 1}, {4, {9, 4}, 
  0}, {4, {9, 4}, 1}, {4, {9, 5}, 0}, {4, {9, 5}, 1}, {1, {9, 6}, 
  0}, {1, {9, 6}, 1}, {4, {9, 7}, 0}, {4, {9, 7}, 1}, {0, {10, 0}, 
  0}, {5, {10, 0}, 1}, {4, {10, 1}, 0}, {4, {10, 1}, 1}, {5, {10, 2}, 
  0}, {2, {10, 2}, 1}, {2, {10, 3}, 0}, {0, {10, 3}, 1}, {0, {10, 4}, 
  0}, {0, {10, 4}, 1}, {0, {10, 5}, 0}, {2, {10, 5}, 1}, {2, {10, 6}, 
  0}, {0, {10, 6}, 1}, {0, {10, 7}, 0}, {5, {11, 0}, 0}, {0, {11, 0}, 
  1}, {0, {11, 1}, 0}, {2, {11, 1}, 1}, {2, {11, 2}, 0}, {0, {11, 2}, 
  1}, {0, {11, 3}, 0}, {0, {11, 3}, 1}, {0, {11, 4}, 0}, {2, {11, 4}, 
  1}, {2, {11, 5}, 0}, {0, {11, 5}, 1}, {0, {11, 6}, 0}, {0, {12, 0}, 
  0}, {2, {12, 0}, 1}, {2, {12, 1}, 0}, {0, {12, 1}, 1}, {0, {12, 2}, 
  0}, {0, {12, 2}, 1}, {0, {12, 3}, 0}, {2, {12, 3}, 1}, {2, {12, 4}, 
  0}, {0, {12, 4}, 1}, {0, {12, 5}, 0}}
\end{scope}
\draw (240:\k) coordinate (C) -- ++(0:\k) coordinate (D) -- ++(-60:\n) coordinate (E) -- ++(-120:2*\n) coordinate (F) -- ++(180:\k) coordinate (A) -- ++(120:\n) coordinate (B) -- cycle;
\draw (B) -- ++(0:\k) coordinate (G) -- (D); \draw (F) -- (G);
\foreach\x in {1,...,\k} \path (C) ++(0:\x-0.5) node[vertex=green] {};
\foreach\x in {1,...,\k} \path (A) ++(0:\x-0.5) node[vertex=green] {};
\foreach\x in {1,...,\n} \path (C) ++(240:\x-0.5) node[vertex=lightblue] {};
\foreach\x in {1,...,\n} \path (F) ++(60:\x-0.5) node[vertex=lightblue] {};
\foreach\x in {1,...,\k} \path (D) ++(-60:\x-0.5) node[vertex=lightblue] {};
\foreach\x in {1,...,\kk} \path (E) ++(120:\x-0.5) node[vertex=red] {};
\foreach\x in {1,...,\k} \path (E) ++(-120:\x-0.5) node[vertex=green] {};
\foreach\x in {1,...,\kk} \path (E) ++(-120:\x+\k-0.5) node[vertex=lightblue] {};
\foreach\x/\c in {1/lightblue,2/red,3/lightblue,4/lightblue,5/red} \path (A) ++(120:\x-0.5) node[vertex=\c] {}; 
\foreach\x/\c in {1/lightblue,2/green,3/green,4/lightblue,5/green} \path (C) ++(240:\x+\n-0.5) node[vertex=\c] {}; 
\draw[dotted] (B) ++(60:\n) -- ++(0:\k) -- ++(-60:\n);
\end{tikzpicture}
\ 
\begin{tikzpicture}[scale=0.4,baseline=-3.2cm]
\begin{scope}[puz]
\rawpuzzle{{0, {0, 2}, 1}, {0, {0, 3}, 0}, {0, {0, 3}, 1}, {0, {0, 4}, 
  0}, {0, {0, 4}, 1}, {0, {0, 5}, 0}, {2, {0, 5}, 1}, {2, {0, 6}, 
  0}, {2, {0, 6}, 1}, {2, {0, 7}, 0}, {0, {0, 7}, 1}, {0, {1, 1}, 
  1}, {0, {1, 2}, 0}, {0, {1, 2}, 1}, {0, {1, 3}, 0}, {0, {1, 3}, 
  1}, {0, {1, 4}, 0}, {2, {1, 4}, 1}, {2, {1, 5}, 0}, {2, {1, 5}, 
  1}, {2, {1, 6}, 0}, {0, {1, 6}, 1}, {0, {1, 7}, 0}, {0, {1, 7}, 
  1}, {0, {2, 0}, 1}, {0, {2, 1}, 0}, {0, {2, 1}, 1}, {0, {2, 2}, 
  0}, {0, {2, 2}, 1}, {0, {2, 3}, 0}, {2, {2, 3}, 1}, {2, {2, 4}, 
  0}, {2, {2, 4}, 1}, {2, {2, 5}, 0}, {0, {2, 5}, 1}, {0, {2, 6}, 
  0}, {0, {2, 6}, 1}, {0, {2, 7}, 0}, {0, {2, 7}, 1}, {4, {3, 0}, 
  0}, {4, {3, 0}, 1}, {4, {3, 1}, 0}, {4, {3, 1}, 1}, {4, {3, 2}, 
  0}, {4, {3, 2}, 1}, {1, {3, 3}, 0}, {1, {3, 3}, 1}, {1, {3, 4}, 
  0}, {1, {3, 4}, 1}, {4, {3, 5}, 0}, {4, {3, 5}, 1}, {4, {3, 6}, 
  0}, {4, {3, 6}, 1}, {4, {3, 7}, 0}, {4, {3, 7}, 1}, {4, {4, 0}, 
  0}, {4, {4, 0}, 1}, {4, {4, 1}, 0}, {4, {4, 1}, 1}, {4, {4, 2}, 0},
  {4, {4, 2}, 1}, {1, {4, 3}, 0}, {1, {4, 3}, 1}, {1, {4, 4}, 
  0}, {1, {4, 4}, 1}, {4, {4, 5}, 0}, {4, {4, 5}, 1}, {4, {4, 6}, 
  0}, {4, {4, 6}, 1}, {4, {4, 7}, 0}, {4, {4, 7}, 1}, {4, {5, 0}, 
  0}, {4, {5, 0}, 1}, {4, {5, 1}, 0}, {4, {5, 1}, 1}, {4, {5, 2}, 
  0}, {4, {5, 2}, 1}, {1, {5, 3}, 0}, {1, {5, 3}, 1}, {1, {5, 4}, 
  0}, {1, {5, 4}, 1}, {4, {5, 5}, 0}, {4, {5, 5}, 1}, {4, {5, 6}, 
  0}, {4, {5, 6}, 1}, {4, {5, 7}, 0}, {4, {5, 7}, 1}, {4, {6, 0}, 
  0}, {4, {6, 0}, 1}, {4, {6, 1}, 0}, {4, {6, 1}, 1}, {4, {6, 2}, 
  0}, {4, {6, 2}, 1}, {1, {6, 3}, 0}, {1, {6, 3}, 1}, {1, {6, 4}, 
  0}, {1, {6, 4}, 1}, {4, {6, 5}, 0}, {4, {6, 5}, 1}, {4, {6, 6}, 
  0}, {4, {6, 6}, 1}, {4, {6, 7}, 0}, {4, {6, 7}, 1}, {4, {7, 0}, 
  0}, {4, {7, 0}, 1}, {4, {7, 1}, 0}, {4, {7, 1}, 1}, {4, {7, 2}, 
  0}, {4, {7, 2}, 1}, {1, {7, 3}, 0}, {1, {7, 3}, 1}, {1, {7, 4}, 
  0}, {1, {7, 4}, 1}, {4, {7, 5}, 0}, {4, {7, 5}, 1}, {4, {7, 6}, 
  0}, {4, {7, 6}, 1}, {4, {7, 7}, 0}, {4, {7, 7}, 1}, {4, {8, 0}, 0}, 
 {4, {8, 0}, 1}, {4, {8, 1}, 0}, {4, {8, 1}, 1}, {4, {8, 2}, 
  0}, {4, {8, 2}, 1}, {1, {8, 3}, 0}, {3, {8, 3}, 1}, {7, {8, 4}, 
  0}, {7, {8, 4}, 1}, {0, {8, 5}, 0}, {0, {8, 5}, 1}, {0, {8, 6}, 
  0}, {0, {8, 6}, 1}, {0, {8, 7}, 0}, {5, {8, 7}, 1}, {0, {9, 0}, 
  0}, {0, {9, 0}, 1}, {0, {9, 1}, 0}, {0, {9, 1}, 1}, {0, {9, 2}, 
  0}, {2, {9, 2}, 1}, {8, {9, 3}, 0}, {4, {9, 3}, 1}, {1, {9, 4}, 
  0}, {1, {9, 4}, 1}, {4, {9, 5}, 0}, {4, {9, 5}, 1}, {4, {9, 6}, 
  0}, {4, {9, 6}, 1}, {5, {9, 7}, 0}, {5, {9, 7}, 1}, {0, {10, 0}, 
  0}, {0, {10, 0}, 1}, {0, {10, 1}, 0}, {2, {10, 1}, 1}, {2, {10, 2}, 
  0}, {0, {10, 2}, 1}, {0, {10, 3}, 0}, {2, {10, 3}, 1}, {2, {10, 4}, 
  0}, {0, {10, 4}, 1}, {0, {10, 5}, 0}, {0, {10, 5}, 1}, {0, {10, 6}, 
  0}, {5, {10, 6}, 1}, {5, {10, 7}, 0}, {4, {11, 0}, 0}, {4, {11, 0}, 
  1}, {1, {11, 1}, 0}, {3, {11, 1}, 1}, {0, {11, 2}, 0}, {0, {11, 2}, 
  1}, {7, {11, 3}, 0}, {7, {11, 3}, 1}, {0, {11, 4}, 0}, {0, {11, 4}, 
  1}, {0, {11, 5}, 0}, {5, {11, 5}, 1}, {5, {11, 6}, 0}, {0, {12, 0}, 
  0}, {2, {12, 0}, 1}, {8, {12, 1}, 0}, {4, {12, 1}, 1}, {4, {12, 2}, 
  0}, {4, {12, 2}, 1}, {1, {12, 3}, 0}, {1, {12, 3}, 1}, {4, {12, 4}, 
  0}, {4, {12, 4}, 1}, {5, {12, 5}, 0}}
\end{scope}
\draw (240:\k) coordinate (C) -- ++(0:\k) coordinate (D) -- ++(-60:\n) coordinate (E) -- ++(-120:2*\n) coordinate (F) -- ++(180:\k) coordinate (A) -- ++(120:\n) coordinate (B) -- cycle;
\draw (A) -- ++(60:2*\n) coordinate (G) -- (C); \draw (E) -- (G);
\foreach\x in {1,...,\k} \path (C) ++(0:\x-0.5) node[vertex=green] {};
\foreach\x in {1,...,\k} \path (A) ++(0:\x-0.5) node[vertex=green] {};
\foreach\x in {1,...,\n} \path (C) ++(240:\x-0.5) node[vertex=lightblue] {};
\foreach\x in {1,...,\n} \path (F) ++(60:\x-0.5) node[vertex=lightblue] {};
\foreach\x in {1,...,\k} \path (D) ++(-60:\x-0.5) node[vertex=lightblue] {};
\foreach\x in {1,...,\kk} \path (E) ++(120:\x-0.5) node[vertex=red] {};
\foreach\x in {1,...,\k} \path (E) ++(-120:\x-0.5) node[vertex=green] {};
\foreach\x in {1,...,\kk} \path (E) ++(-120:\x+\k-0.5) node[vertex=lightblue] {};
\foreach\x/\c in {1/lightblue,2/red,3/lightblue,4/lightblue,5/red} \path (A) ++(120:\x-0.5) node[vertex=\c] {}; 
\foreach\x/\c in {1/lightblue,2/green,3/green,4/lightblue,5/green} \path (C) ++(240:\x+\n-0.5) node[vertex=\c] {}; 
\draw[dotted] (B) ++(60:\n) -- ++(-60:\n) -- ++(0:\k);
\end{tikzpicture}
\end{center}
\end{samepage}

We rewrite the identity \eqref{eq:master}:
\begin{equation}\label{eq:1}
\sum_\nu \loz{\nu/->,\varnothing/->,\varnothing/->,\lambda/->}{\z}{\z} G_\nu(\x;\z) G^\mu(\x;\y)
=
\sum_{\tilde\nu} \loz{\mu/->,\varnothing/->,\tilde\nu/->,\lambda/->}{\z}{\y} G_{\varnothing^\ast}(\x;\z) G^{\tilde\nu{}^\ast}(\x;\y).
\end{equation}
In the l.h.s., one can apply Lemma~\ref{lem:split}:
\begin{equation}\label{eq:1b}
\sum_{\nu:\, \nu^\ast \vartriangleright \lambda} (-1)^{|\nu^\ast|-|\lambda|} \prod_{i\in\nu} z_i
\prod_{i=n-k+1}^n z_i^{-1}\,
G_\nu(\x;\z) G^\mu(\x;\y)
=
\sum_{\tilde\nu} \loz{\mu/->,\varnothing/->,\tilde\nu/->,\lambda/->}{\z}{\y} G_{\varnothing^\ast}(\x;\z) G^{\tilde\nu{}^\ast}(\x;\y).
\end{equation}
The sum in the l.h.s.\ is actually known (see \cite{Lenart}, for the non-equivariant case); we can evaluate it in the present framework
by first considering \eqref{eq:1} at the further specialization 
$\lambda=\varnothing$, $\y=\z$. In the l.h.s., using once more Lemma~\ref{lem:split}, it is easy to see
that only two configurations contribute to the puzzle, namely
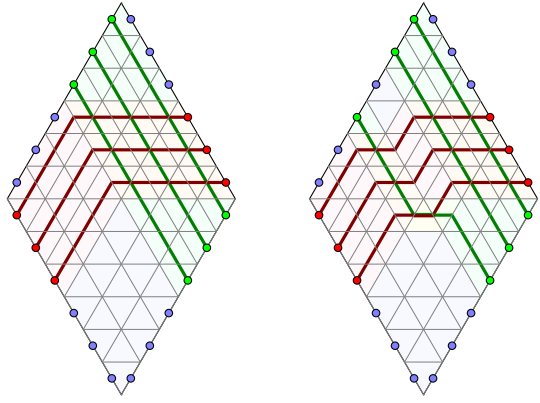
\begin{center}
\def\size{6}
\begin{tikzpicture}[scale=0.5]
\doublepuzzle{{0, {0, 0}, 0}, {0, {0, 0}, 1}, {0, {0, 1}, 0}, {0, {0, 1}, 
  1}, {0, {0, 2}, 0}, {2, {0, 2}, 1}, {2, {0, 3}, 0}, {2, {0, 3}, 
  1}, {2, {0, 4}, 0}, {2, {0, 4}, 1}, {2, {0, 5}, 0}, {0, {0, 5}, 
  1}, {0, {1, 0}, 0}, {0, {1, 0}, 1}, {0, {1, 1}, 0}, {2, {1, 1}, 
  1}, {2, {1, 2}, 0}, {2, {1, 2}, 1}, {2, {1, 3}, 0}, {2, {1, 3}, 
  1}, {2, {1, 4}, 0}, {0, {1, 4}, 1}, {0, {1, 5}, 0}, {0, {1, 5}, 
  1}, {0, {2, 0}, 0}, {2, {2, 0}, 1}, {2, {2, 1}, 0}, {2, {2, 1}, 
  1}, {2, {2, 2}, 0}, {2, {2, 2}, 1}, {2, {2, 3}, 0}, {0, {2, 3}, 
  1}, {0, {2, 4}, 0}, {0, {2, 4}, 1}, {0, {2, 5}, 0}, {0, {2, 5}, 
  1}, {1, {3, 0}, 0}, {1, {3, 0}, 1}, {1, {3, 1}, 0}, {1, {3, 1}, 
  1}, {1, {3, 2}, 0}, {1, {3, 2}, 1}, {4, {3, 3}, 0}, {4, {3, 3}, 
  1}, {4, {3, 4}, 0}, {4, {3, 4}, 1}, {4, {3, 5}, 0}, {4, {3, 5}, 
  1}, {1, {4, 0}, 0}, {1, {4, 0}, 1}, {1, {4, 1}, 0}, {1, {4, 1}, 
  1}, {1, {4, 2}, 0}, {1, {4, 2}, 1}, {4, {4, 3}, 0}, {4, {4, 3}, 
  1}, {4, {4, 4}, 0}, {4, {4, 4}, 1}, {4, {4, 5}, 0}, {4, {4, 5}, 1},
  {1, {5, 0}, 0}, {1, {5, 0}, 1}, {1, {5, 1}, 0}, {1, {5, 1}, 
  1}, {1, {5, 2}, 0}, {1, {5, 2}, 1}, {4, {5, 3}, 0}, {4, {5, 3}, 
  1}, {4, {5, 4}, 0}, {4, {5, 4}, 1}, {4, {5, 5}, 0}, {4, {5, 5}, 1}}
\end{tikzpicture}
\qquad
\begin{tikzpicture}[scale=0.5]
\doublepuzzle{{0, {0, 0}, 0}, {0, {0, 0}, 1}, {0, {0, 1}, 0}, {0, {0, 1}, 
  1}, {0, {0, 2}, 0}, {2, {0, 2}, 1}, {2, {0, 3}, 0}, {2, {0, 3}, 
  1}, {2, {0, 4}, 0}, {2, {0, 4}, 1}, {2, {0, 5}, 0}, {0, {0, 5}, 
  1}, {0, {1, 0}, 0}, {0, {1, 0}, 1}, {0, {1, 1}, 0}, {2, {1, 1}, 
  1}, {2, {1, 2}, 0}, {2, {1, 2}, 1}, {2, {1, 3}, 0}, {2, {1, 3}, 
  1}, {2, {1, 4}, 0}, {0, {1, 4}, 1}, {0, {1, 5}, 0}, {0, {1, 5}, 
  1}, {4, {2, 0}, 0}, {4, {2, 0}, 1}, {1, {2, 1}, 0}, {1, {2, 1}, 
  1}, {1, {2, 2}, 0}, {1, {2, 2}, 1}, {1, {2, 3}, 0}, {3, {2, 3}, 
  1}, {0, {2, 4}, 0}, {0, {2, 4}, 1}, {0, {2, 5}, 0}, {0, {2, 5}, 
  1}, {0, {3, 0}, 0}, {2, {3, 0}, 1}, {2, {3, 1}, 0}, {2, {3, 1}, 
  1}, {2, {3, 2}, 0}, {2, {3, 2}, 1}, {8, {3, 3}, 0}, {4, {3, 3}, 
  1}, {4, {3, 4}, 0}, {4, {3, 4}, 1}, {4, {3, 5}, 0}, {4, {3, 5}, 
  1}, {1, {4, 0}, 0}, {1, {4, 0}, 1}, {1, {4, 1}, 0}, {1, {4, 1}, 
  1}, {1, {4, 2}, 0}, {1, {4, 2}, 1}, {4, {4, 3}, 0}, {4, {4, 3}, 
  1}, {4, {4, 4}, 0}, {4, {4, 4}, 1}, {4, {4, 5}, 0}, {4, {4, 5}, 1},
  {1, {5, 0}, 0}, {1, {5, 0}, 1}, {1, {5, 1}, 0}, {1, {5, 1}, 
  1}, {1, {5, 2}, 0}, {1, {5, 2}, 1}, {4, {5, 3}, 0}, {4, {5, 3}, 
  1}, {4, {5, 4}, 0}, {4, {5, 4}, 1}, {4, {5, 5}, 0}, {4, {5, 5}, 1}}
\end{tikzpicture}
\end{center}
corresponding respectively to $\nu=\varnothing^\ast=k\times(n-k)$ and $\nu=\fund^\ast$.
Using $G^\varnothing(\x;\z)=1$ and $G^{\stableau{\\}}(\x;\z)=1-\prod_{i=1}^k x_i/z_i$,
the sum
$\sum_{\nu:\, \nu^\ast \vartriangleright \varnothing} (-1)^{|\nu^\ast|} \prod_{i\in\nu} z_i\, G_\nu(\x;\z)$
then recombines to $\prod_{i=1}^k x_i^2 \prod_{i=n-k+1}^n z_i^{-1}$.

In the r.h.s., we can also apply Lemma~\ref{lem:split}; we then obtain
\[
\prod_{i=1}^k x_i^2
\prod_{i=n-k+1}^n z_i^{-2}\, G^\mu(\x;\z)=
\sum_{\tilde\nu:\, \tilde\nu\vartriangleleft\mu^\ast} 
(-1)^{|\mu^\ast|-|\tilde\nu|} \prod_{i=1}^k x_i \prod_{i=n-k+1}^n z_i^{-2} \prod_{i\in\mu} z_i 
\,G^{\tilde\nu{}^\ast}(\x;\z).
\]
Simplifying and renaming $\tilde\nu{}^\ast$ to $\nu$, we find the desired identity
\begin{equation}\label{eq:id}
\sum_{\nu:\, \nu\vartriangleright\mu} 
(-1)^{|\nu|-|\mu|} G^{\nu}(\x;\z)
=
\frac{\prod_{i=1}^k x_i}
{\prod_{i\in\mu} z_i}
G^\mu(\x;\z)=G_{\mu^\ast}(\x;\zz)
\end{equation}
which interprets nicely part (3) of Lemma~\ref{lem:split}.

Now we go back to the general case of \eqref{eq:1b};
using the identity \eqref{eq:id}, we obtain
\begin{equation}\label{eq:preMSalt}
G_{\lambda^\ast}(\x;\z) G^\mu(\x;\y)
=
\sum_{\tilde\nu} \loz{\mu/->,\varnothing/->,\tilde\nu/->,\lambda/->}{\z}{\y} 
G^{\tilde\nu{}^\ast}(\x;\y).
\end{equation}
This is our first nontrivial identity, showing how to expand on the basis
of Grothendieck polynomials
the product of a Grothendieck polynomial and a dual Grothendieck polynomial, the latter
having a different secondary alphabet.

We can further rewrite this identity by lowering all the indices; we easily find,
after subtitutions $\lambda\to\lambda^\ast$, $\mu\to\mu^\ast$, $\y\to\yy$:
\[
G_{\lambda}(\x;\z)
G_{\mu}(\x;\y)
=
\sum_{\tilde\nu} 
\prod_{i\in\mu} y_i^{-1}
\loz{\mu/<-,\varnothing/->,\tilde\nu/->,\lambda/<-}{\z}{\yy} 
\prod_{i\in\tilde\nu} y_i
\,
G_{\tilde\nu}(\x;\y).
\]
This is very similar to Theorem~\ref{thm:MSalt}, the extra factors being manifestly a conjugation 
that should be absorbed in the weights of the puzzle.
Indeed, we perform the following transformation: for each elementary lozenge whose weight depends on
$w=y_i/z_j$, as is the case above, we give an extra
weight $y_i$ (resp.\ $y_i^{-1}$) to NW (resp.\ SE) empty edges;
and an extra weight $z_j$ (resp.\ $z_j^{-1}$) to NE (resp.\ SW) edges with both red and green lines.
This turns the weights \eqref{eq:wei} into the modified weights \eqref{eq:modwei},
and the equation above becomes exactly the statement of Theorem~\ref{thm:MSalt}.

Next we impose $\y=\zz$; one more application of Lemma~\ref{lem:split} directly leads to Theorem~\ref{thm:equivdualalt}.

\begin{samepage}
\subsubsection{$\lambda=\mu=\varnothing$, $\t=\y$.}
\begin{center}
\begin{tikzpicture}[scale=0.4,baseline=-3.2cm]
\begin{scope}[puz]
\rawpuzzle{{0, {0, 2}, 1}, {0, {0, 3}, 0}, {0, {0, 3}, 1}, {0, {0, 4}, 
  0}, {0, {0, 4}, 1}, {7, {0, 5}, 0}, {7, {0, 5}, 1}, {0, {0, 6}, 
  0}, {2, {0, 6}, 1}, {2, {0, 7}, 0}, {0, {0, 7}, 1}, {0, {1, 1}, 
  1}, {0, {1, 2}, 0}, {5, {1, 2}, 1}, {4, {1, 3}, 0}, {4, {1, 3}, 
  1}, {4, {1, 4}, 0}, {4, {1, 4}, 1}, {1, {1, 5}, 0}, {1, {1, 5}, 
  1}, {1, {1, 6}, 0}, {1, {1, 6}, 1}, {4, {1, 7}, 0}, {4, {1, 7}, 
  1}, {5, {2, 0}, 1}, {4, {2, 1}, 0}, {4, {2, 1}, 1}, {5, {2, 2}, 
  0}, {0, {2, 2}, 1}, {0, {2, 3}, 0}, {0, {2, 3}, 1}, {0, {2, 4}, 
  0}, {2, {2, 4}, 1}, {2, {2, 5}, 0}, {2, {2, 5}, 1}, {2, {2, 6}, 
  0}, {0, {2, 6}, 1}, {0, {2, 7}, 0}, {0, {2, 7}, 1}, {5, {3, 0}, 
  0}, {5, {3, 0}, 1}, {4, {3, 1}, 0}, {4, {3, 1}, 1}, {4, {3, 2}, 
  0}, {4, {3, 2}, 1}, {4, {3, 3}, 0}, {4, {3, 3}, 1}, {1, {3, 4}, 
  0}, {1, {3, 4}, 1}, {1, {3, 5}, 0}, {3, {3, 5}, 1}, {0, {3, 6}, 
  0}, {0, {3, 6}, 1}, {0, {3, 7}, 0}, {0, {3, 7}, 1}, {5, {4, 0}, 
  0}, {0, {4, 0}, 1}, {0, {4, 1}, 0}, {0, {4, 1}, 1}, {0, {4, 2}, 0},
  {0, {4, 2}, 1}, {0, {4, 3}, 0}, {2, {4, 3}, 1}, {2, {4, 4}, 
  0}, {2, {4, 4}, 1}, {8, {4, 5}, 0}, {4, {4, 5}, 1}, {4, {4, 6}, 
  0}, {4, {4, 6}, 1}, {4, {4, 7}, 0}, {4, {4, 7}, 1}, {4, {5, 0}, 
  0}, {4, {5, 0}, 1}, {4, {5, 1}, 0}, {4, {5, 1}, 1}, {4, {5, 2}, 
  0}, {4, {5, 2}, 1}, {1, {5, 3}, 0}, {1, {5, 3}, 1}, {1, {5, 4}, 
  0}, {1, {5, 4}, 1}, {4, {5, 5}, 0}, {4, {5, 5}, 1}, {4, {5, 6}, 
  0}, {4, {5, 6}, 1}, {4, {5, 7}, 0}, {4, {5, 7}, 1}, {4, {6, 0}, 
  0}, {4, {6, 0}, 1}, {4, {6, 1}, 0}, {4, {6, 1}, 1}, {4, {6, 2}, 
  0}, {4, {6, 2}, 1}, {1, {6, 3}, 0}, {1, {6, 3}, 1}, {1, {6, 4}, 
  0}, {1, {6, 4}, 1}, {4, {6, 5}, 0}, {4, {6, 5}, 1}, {4, {6, 6}, 
  0}, {4, {6, 6}, 1}, {4, {6, 7}, 0}, {4, {6, 7}, 1}, {4, {7, 0}, 
  0}, {4, {7, 0}, 1}, {4, {7, 1}, 0}, {4, {7, 1}, 1}, {4, {7, 2}, 
  0}, {4, {7, 2}, 1}, {1, {7, 3}, 0}, {1, {7, 3}, 1}, {1, {7, 4}, 
  0}, {1, {7, 4}, 1}, {4, {7, 5}, 0}, {4, {7, 5}, 1}, {4, {7, 6}, 
  0}, {4, {7, 6}, 1}, {4, {7, 7}, 0}, {4, {7, 7}, 1}, {4, {8, 0}, 0}, 
 {4, {8, 0}, 1}, {4, {8, 1}, 0}, {4, {8, 1}, 1}, {4, {8, 2}, 
  0}, {4, {8, 2}, 1}, {1, {8, 3}, 0}, {1, {8, 3}, 1}, {1, {8, 4}, 
  0}, {1, {8, 4}, 1}, {4, {8, 5}, 0}, {4, {8, 5}, 1}, {4, {8, 6}, 
  0}, {4, {8, 6}, 1}, {4, {8, 7}, 0}, {4, {8, 7}, 1}, {4, {9, 0}, 
  0}, {4, {9, 0}, 1}, {4, {9, 1}, 0}, {4, {9, 1}, 1}, {4, {9, 2}, 
  0}, {4, {9, 2}, 1}, {1, {9, 3}, 0}, {1, {9, 3}, 1}, {1, {9, 4}, 
  0}, {1, {9, 4}, 1}, {4, {9, 5}, 0}, {4, {9, 5}, 1}, {4, {9, 6}, 
  0}, {4, {9, 6}, 1}, {4, {9, 7}, 0}, {4, {9, 7}, 1}, {0, {10, 0}, 
  0}, {0, {10, 0}, 1}, {0, {10, 1}, 0}, {0, {10, 1}, 1}, {0, {10, 2}, 
  0}, {2, {10, 2}, 1}, {2, {10, 3}, 0}, {2, {10, 3}, 1}, {2, {10, 4}, 
  0}, {0, {10, 4}, 1}, {0, {10, 5}, 0}, {0, {10, 5}, 1}, {0, {10, 6}, 
  0}, {0, {10, 6}, 1}, {0, {10, 7}, 0}, {0, {11, 0}, 0}, {0, {11, 0}, 
  1}, {0, {11, 1}, 0}, {2, {11, 1}, 1}, {2, {11, 2}, 0}, {2, {11, 2}, 
  1}, {2, {11, 3}, 0}, {0, {11, 3}, 1}, {0, {11, 4}, 0}, {0, {11, 4}, 
  1}, {0, {11, 5}, 0}, {0, {11, 5}, 1}, {0, {11, 6}, 0}, {0, {12, 0}, 
  0}, {2, {12, 0}, 1}, {2, {12, 1}, 0}, {2, {12, 1}, 1}, {2, {12, 2}, 
  0}, {0, {12, 2}, 1}, {0, {12, 3}, 0}, {0, {12, 3}, 1}, {0, {12, 4}, 
  0}, {0, {12, 4}, 1}, {0, {12, 5}, 0}}
\end{scope}
\draw (240:\k) coordinate (C) -- ++(0:\k) coordinate (D) -- ++(-60:\n) coordinate (E) -- ++(-120:2*\n) coordinate (F) -- ++(180:\k) coordinate (A) -- ++(120:\n) coordinate (B) -- cycle;
\draw (B) -- ++(0:\k) coordinate (G) -- (D); \draw (F) -- (G);
\foreach\x in {1,...,\k} \path (C) ++(0:\x-0.5) node[vertex=green] {};
\foreach\x in {1,...,\k} \path (A) ++(0:\x-0.5) node[vertex=green] {};
\foreach\x in {1,...,\n} \path (C) ++(240:\x-0.5) node[vertex=lightblue] {};
\foreach\x in {1,...,\n} \path (F) ++(60:\x-0.5) node[vertex=lightblue] {};
\foreach\x/\c in {1/lightblue,2/lightblue,3/lightblue,4/red,5/red} \path (A) ++(120:\x-0.5) node[vertex=\c] {};
\foreach\x/\c in {1/lightblue,2/lightblue,3/green,4/green,5/green} \path (C) ++(240:\x+\n-0.5) node[vertex=\c] {};
\foreach\x/\c in {1/lightblue,2/lightblue,3/red,4/lightblue,5/red} \path (D) ++(-60:\x-0.5) node[vertex=\c] {}; 
\foreach\x/\c in {1/green,2/lightblue,3/green,4/green,5/lightblue} \path (E) ++(-120:\x-0.5) node[vertex=\c] {}; 
\draw[dotted] (B) ++(60:\n) -- ++(0:\k) -- ++(-60:\n);
\end{tikzpicture}
\ 
\begin{tikzpicture}[scale=0.4,baseline=-3.2cm]
\begin{scope}[puz]
\rawpuzzle{{0, {0, 2}, 1}, {0, {0, 3}, 0}, {0, {0, 3}, 1}, {0, {0, 4}, 
  0}, {2, {0, 4}, 1}, {2, {0, 5}, 0}, {0, {0, 5}, 1}, {0, {0, 6}, 
  0}, {2, {0, 6}, 1}, {2, {0, 7}, 0}, {0, {0, 7}, 1}, {0, {1, 1}, 
  1}, {0, {1, 2}, 0}, {0, {1, 2}, 1}, {0, {1, 3}, 0}, {2, {1, 3}, 
  1}, {2, {1, 4}, 0}, {0, {1, 4}, 1}, {0, {1, 5}, 0}, {2, {1, 5}, 
  1}, {2, {1, 6}, 0}, {0, {1, 6}, 1}, {0, {1, 7}, 0}, {5, {1, 7}, 
  1}, {0, {2, 0}, 1}, {0, {2, 1}, 0}, {0, {2, 1}, 1}, {0, {2, 2}, 
  0}, {2, {2, 2}, 1}, {2, {2, 3}, 0}, {0, {2, 3}, 1}, {0, {2, 4}, 
  0}, {2, {2, 4}, 1}, {2, {2, 5}, 0}, {5, {2, 5}, 1}, {4, {2, 6}, 
  0}, {4, {2, 6}, 1}, {5, {2, 7}, 0}, {0, {2, 7}, 1}, {4, {3, 0}, 
  0}, {4, {3, 0}, 1}, {4, {3, 1}, 0}, {4, {3, 1}, 1}, {1, {3, 2}, 
  0}, {1, {3, 2}, 1}, {4, {3, 3}, 0}, {4, {3, 3}, 1}, {1, {3, 4}, 
  0}, {1, {3, 4}, 1}, {5, {3, 5}, 0}, {0, {3, 5}, 1}, {0, {3, 6}, 
  0}, {0, {3, 6}, 1}, {0, {3, 7}, 0}, {0, {3, 7}, 1}, {4, {4, 0}, 
  0}, {4, {4, 0}, 1}, {4, {4, 1}, 0}, {4, {4, 1}, 1}, {1, {4, 2}, 0},
  {1, {4, 2}, 1}, {4, {4, 3}, 0}, {4, {4, 3}, 1}, {1, {4, 4}, 
  0}, {1, {4, 4}, 1}, {4, {4, 5}, 0}, {4, {4, 5}, 1}, {4, {4, 6}, 
  0}, {4, {4, 6}, 1}, {4, {4, 7}, 0}, {4, {4, 7}, 1}, {4, {5, 0}, 
  0}, {4, {5, 0}, 1}, {4, {5, 1}, 0}, {4, {5, 1}, 1}, {1, {5, 2}, 
  0}, {1, {5, 2}, 1}, {4, {5, 3}, 0}, {4, {5, 3}, 1}, {1, {5, 4}, 
  0}, {1, {5, 4}, 1}, {4, {5, 5}, 0}, {4, {5, 5}, 1}, {4, {5, 6}, 
  0}, {4, {5, 6}, 1}, {4, {5, 7}, 0}, {4, {5, 7}, 1}, {4, {6, 0}, 
  0}, {4, {6, 0}, 1}, {4, {6, 1}, 0}, {4, {6, 1}, 1}, {1, {6, 2}, 
  0}, {1, {6, 2}, 1}, {4, {6, 3}, 0}, {4, {6, 3}, 1}, {1, {6, 4}, 
  0}, {1, {6, 4}, 1}, {4, {6, 5}, 0}, {4, {6, 5}, 1}, {4, {6, 6}, 
  0}, {4, {6, 6}, 1}, {4, {6, 7}, 0}, {4, {6, 7}, 1}, {4, {7, 0}, 
  0}, {4, {7, 0}, 1}, {4, {7, 1}, 0}, {4, {7, 1}, 1}, {1, {7, 2}, 
  0}, {1, {7, 2}, 1}, {4, {7, 3}, 0}, {4, {7, 3}, 1}, {1, {7, 4}, 
  0}, {1, {7, 4}, 1}, {4, {7, 5}, 0}, {4, {7, 5}, 1}, {4, {7, 6}, 
  0}, {4, {7, 6}, 1}, {4, {7, 7}, 0}, {4, {7, 7}, 1}, {4, {8, 0}, 0}, 
 {4, {8, 0}, 1}, {4, {8, 1}, 0}, {4, {8, 1}, 1}, {1, {8, 2}, 
  0}, {1, {8, 2}, 1}, {4, {8, 3}, 0}, {4, {8, 3}, 1}, {1, {8, 4}, 
  0}, {1, {8, 4}, 1}, {4, {8, 5}, 0}, {4, {8, 5}, 1}, {4, {8, 6}, 
  0}, {4, {8, 6}, 1}, {4, {8, 7}, 0}, {4, {8, 7}, 1}, {4, {9, 0}, 
  0}, {4, {9, 0}, 1}, {4, {9, 1}, 0}, {4, {9, 1}, 1}, {1, {9, 2}, 
  0}, {3, {9, 2}, 1}, {0, {9, 3}, 0}, {2, {9, 3}, 1}, {2, {9, 4}, 
  0}, {0, {9, 4}, 1}, {0, {9, 5}, 0}, {0, {9, 5}, 1}, {0, {9, 6}, 
  0}, {5, {9, 6}, 1}, {4, {9, 7}, 0}, {4, {9, 7}, 1}, {0, {10, 0}, 
  0}, {0, {10, 0}, 1}, {0, {10, 1}, 0}, {0, {10, 1}, 1}, {3, {10, 2}, 
  0}, {1, {10, 2}, 1}, {1, {10, 3}, 0}, {3, {10, 3}, 1}, {0, {10, 4}, 
  0}, {0, {10, 4}, 1}, {0, {10, 5}, 0}, {5, {10, 5}, 1}, {5, {10, 6}, 
  0}, {0, {10, 6}, 1}, {0, {10, 7}, 0}, {0, {11, 0}, 0}, {0, {11, 0}, 
  1}, {0, {11, 1}, 0}, {2, {11, 1}, 1}, {2, {11, 2}, 0}, {2, {11, 2}, 
  1}, {8, {11, 3}, 0}, {4, {11, 3}, 1}, {4, {11, 4}, 0}, {4, {11, 4}, 
  1}, {5, {11, 5}, 0}, {0, {11, 5}, 1}, {0, {11, 6}, 0}, {0, {12, 0}, 
  0}, {2, {12, 0}, 1}, {2, {12, 1}, 0}, {2, {12, 1}, 1}, {2, {12, 2}, 
  0}, {0, {12, 2}, 1}, {0, {12, 3}, 0}, {0, {12, 3}, 1}, {0, {12, 4}, 
  0}, {0, {12, 4}, 1}, {0, {12, 5}, 0}}
\end{scope}
\draw (240:\k) coordinate (C) -- ++(0:\k) coordinate (D) -- ++(-60:\n) coordinate (E) -- ++(-120:2*\n) coordinate (F) -- ++(180:\k) coordinate (A) -- ++(120:\n) coordinate (B) -- cycle;
\draw (A) -- ++(60:2*\n) coordinate (G) -- (C); \draw (E) -- (G);
\foreach\x in {1,...,\k} \path (C) ++(0:\x-0.5) node[vertex=green] {};
\foreach\x in {1,...,\k} \path (A) ++(0:\x-0.5) node[vertex=green] {};
\foreach\x in {1,...,\n} \path (C) ++(240:\x-0.5) node[vertex=lightblue] {};
\foreach\x in {1,...,\n} \path (F) ++(60:\x-0.5) node[vertex=lightblue] {};
\foreach\x/\c in {1/lightblue,2/lightblue,3/lightblue,4/red,5/red} \path (A) ++(120:\x-0.5) node[vertex=\c] {};
\foreach\x/\c in {1/lightblue,2/lightblue,3/green,4/green,5/green} \path (C) ++(240:\x+\n-0.5) node[vertex=\c] {};
\foreach\x/\c in {1/lightblue,2/lightblue,3/red,4/lightblue,5/red} \path (D) ++(-60:\x-0.5) node[vertex=\c] {}; 
\foreach\x/\c in {1/green,2/lightblue,3/green,4/green,5/lightblue} \path (E) ++(-120:\x-0.5) node[vertex=\c] {}; 
\draw[dotted] (B) ++(60:\n) -- ++(-60:\n) -- ++(0:\k);
\end{tikzpicture}
\end{center}
\end{samepage}

The reasoning is the same: we start from \eqref{eq:master} and use Lemma~\ref{lem:split}
supplemented by identity \eqref{eq:id}:
\begin{align*}
\sum_\nu \loz{\nu/->,\tilde\lambda/->,\tilde\mu/->,\varnothing/->}{\y}{\z} G_\nu(\x;\z) 
&=
\sum_{\tilde\nu} \loz{\varnothing/->,\tilde\lambda/->,\tilde\nu/->,\varnothing/->}{\y}{\y} 
G_{\tilde\mu{}^\ast}(\x;\z) G^{\tilde\nu{}^\ast}(\x;\y)
\\
&=
G_{\tilde\mu{}^\ast}(\x;\z)
\sum_{\tilde\nu:\,\tilde\nu{}^\ast\vartriangleright\tilde\lambda}
(-1)^{|\tilde\nu{}^\ast|-|\tilde\lambda|}\prod_{i\in\tilde\lambda} y_i 
\prod_{i=n-k+1}^n y_i^{-1}
\,
G^{\tilde\nu{}^\ast}(\x;\y)
\\
&=
\prod_{i=1}^k x_i \prod_{i=n-k+1}^n y_i^{-1}
\,
G_{\tilde\mu{}^\ast}(\x;\z) G^{\tilde\lambda}(\x;\y).
\end{align*}

After some obvious notational simplifications, we obtain another remarkable identity:
\begin{equation}\label{eq:preequivalt}
G^\lambda(\x;\y)G_\mu(\x;\z)
=
\sum_\nu
\prod_{i=n-k+1}^n y_i
\prod_{i\in\nu^\ast} z_i^{-1}
\loz{\nu/<-,\lambda/->,\mu/<-,\varnothing/->}{\y}{\z}
G^\nu(\x;\zz)
\end{equation}
(compare with \eqref{eq:preMSalt}).

Finally, set $\z=\y$ and compute (using $\left<\cdot\right>$ in $K_T(Gr(k,n))$):
\begin{align*}
\left<
G^\lambda(\x;\y)
G{}^\kappa(\x;\yy)
G_\nu(\x;\yy)
\right>
&=
\sum_\mu
\left<
G^\lambda(\x;\y)
G_\nu(\x;\yy)
G_\mu(\x;\y)
\right>
\left<
G^\mu(\x;\y)
G{}^\kappa(\x;\yy)
\right>
\\
&=
\sum_{\mu}
\prod_{i=n-k+1}^n y_i
\prod_{i\in\nu^\ast} y_i^{-1}
\loz{\nu/<-,\lambda/->,\mu/<-,\varnothing/->}{\y}{\y}
\left<
G^\mu(\x;\y)
G{}^\kappa(\x;\yy)
\right>
\\
&=
\sum_{\mu,\rho}
\prod_{i=n-k+1}^n y_i
\prod_{i\in\nu^\ast} y_i^{-1}
\triup{\nu/<-,\lambda/->,\rho/->}{\y}\tridown{\rho/<-,\mu/<-,\varnothing/->}{\y}
\left<
G^\mu(\x;\y)
G{}^\kappa(\x;\yy)
\right>
\\
&=
\sum_{\rho}
\prod_{i\in\rho^\ast} y_i
\prod_{i\in\nu^\ast} y_i^{-1}
\triup{\nu/<-,\lambda/->,\rho/->}{\y}
\left<
G_\rho(\x;\yy)
G{}^\kappa(\x;\yy)
\right>
\\
&=
\prod_{i\in\kappa^\ast} y_i
\prod_{i\in\nu^\ast} y_i^{-1}
\triup{\nu/<-,\lambda/->,\kappa/->}{\y}
\end{align*}
We find that the prefactors can once again be absorbed by using modified weights \eqref{eq:modwei}. We then recognize
Theorem~\ref{thm:equivalt}.

\begin{samepage}
\subsubsection{$\tilde\lambda=\mu=\varnothing$, $\t=\yy$.}
\begin{center}
\begin{tikzpicture}[scale=0.4,baseline=-3.2cm]
\begin{scope}[puz]
\rawpuzzle{{0, {0, 2}, 1}, {0, {0, 3}, 0}, {0, {0, 3}, 1}, {0, {0, 4}, 
  0}, {0, {0, 4}, 1}, {0, {0, 5}, 0}, {0, {0, 5}, 1}, {7, {0, 6}, 
  0}, {7, {0, 6}, 1}, {7, {0, 7}, 0}, {7, {0, 7}, 1}, {0, {1, 1}, 
  1}, {0, {1, 2}, 0}, {5, {1, 2}, 1}, {4, {1, 3}, 0}, {4, {1, 3}, 
  1}, {4, {1, 4}, 0}, {4, {1, 4}, 1}, {4, {1, 5}, 0}, {4, {1, 5}, 
  1}, {1, {1, 6}, 0}, {1, {1, 6}, 1}, {1, {1, 7}, 0}, {1, {1, 7}, 
  1}, {5, {2, 0}, 1}, {4, {2, 1}, 0}, {4, {2, 1}, 1}, {5, {2, 2}, 
  0}, {0, {2, 2}, 1}, {0, {2, 3}, 0}, {0, {2, 3}, 1}, {0, {2, 4}, 
  0}, {0, {2, 4}, 1}, {0, {2, 5}, 0}, {2, {2, 5}, 1}, {2, {2, 6}, 
  0}, {2, {2, 6}, 1}, {2, {2, 7}, 0}, {0, {2, 7}, 1}, {5, {3, 0}, 
  0}, {5, {3, 0}, 1}, {4, {3, 1}, 0}, {4, {3, 1}, 1}, {4, {3, 2}, 
  0}, {4, {3, 2}, 1}, {4, {3, 3}, 0}, {4, {3, 3}, 1}, {4, {3, 4}, 
  0}, {4, {3, 4}, 1}, {1, {3, 5}, 0}, {3, {3, 5}, 1}, {7, {3, 6}, 
  0}, {7, {3, 6}, 1}, {0, {3, 7}, 0}, {0, {3, 7}, 1}, {5, {4, 0}, 
  0}, {0, {4, 0}, 1}, {0, {4, 1}, 0}, {0, {4, 1}, 1}, {0, {4, 2}, 0},
  {0, {4, 2}, 1}, {0, {4, 3}, 0}, {0, {4, 3}, 1}, {0, {4, 4}, 
  0}, {2, {4, 4}, 1}, {8, {4, 5}, 0}, {4, {4, 5}, 1}, {1, {4, 6}, 
  0}, {1, {4, 6}, 1}, {4, {4, 7}, 0}, {4, {4, 7}, 1}, {4, {5, 0}, 
  0}, {4, {5, 0}, 1}, {4, {5, 1}, 0}, {4, {5, 1}, 1}, {4, {5, 2}, 
  0}, {4, {5, 2}, 1}, {4, {5, 3}, 0}, {4, {5, 3}, 1}, {1, {5, 4}, 
  0}, {1, {5, 4}, 1}, {4, {5, 5}, 0}, {4, {5, 5}, 1}, {1, {5, 6}, 
  0}, {1, {5, 6}, 1}, {4, {5, 7}, 0}, {4, {5, 7}, 1}, {4, {6, 0}, 
  0}, {4, {6, 0}, 1}, {4, {6, 1}, 0}, {4, {6, 1}, 1}, {4, {6, 2}, 
  0}, {4, {6, 2}, 1}, {4, {6, 3}, 0}, {4, {6, 3}, 1}, {1, {6, 4}, 
  0}, {1, {6, 4}, 1}, {4, {6, 5}, 0}, {4, {6, 5}, 1}, {1, {6, 6}, 
  0}, {1, {6, 6}, 1}, {4, {6, 7}, 0}, {4, {6, 7}, 1}, {4, {7, 0}, 
  0}, {4, {7, 0}, 1}, {4, {7, 1}, 0}, {4, {7, 1}, 1}, {4, {7, 2}, 
  0}, {4, {7, 2}, 1}, {4, {7, 3}, 0}, {4, {7, 3}, 1}, {1, {7, 4}, 
  0}, {1, {7, 4}, 1}, {4, {7, 5}, 0}, {4, {7, 5}, 1}, {1, {7, 6}, 
  0}, {1, {7, 6}, 1}, {4, {7, 7}, 0}, {4, {7, 7}, 1}, {4, {8, 0}, 0}, 
 {4, {8, 0}, 1}, {4, {8, 1}, 0}, {4, {8, 1}, 1}, {4, {8, 2}, 
  0}, {4, {8, 2}, 1}, {4, {8, 3}, 0}, {4, {8, 3}, 1}, {1, {8, 4}, 
  0}, {1, {8, 4}, 1}, {4, {8, 5}, 0}, {4, {8, 5}, 1}, {1, {8, 6}, 
  0}, {1, {8, 6}, 1}, {4, {8, 7}, 0}, {4, {8, 7}, 1}, {4, {9, 0}, 
  0}, {4, {9, 0}, 1}, {4, {9, 1}, 0}, {4, {9, 1}, 1}, {4, {9, 2}, 
  0}, {4, {9, 2}, 1}, {4, {9, 3}, 0}, {4, {9, 3}, 1}, {1, {9, 4}, 
  0}, {1, {9, 4}, 1}, {4, {9, 5}, 0}, {4, {9, 5}, 1}, {1, {9, 6}, 
  0}, {1, {9, 6}, 1}, {4, {9, 7}, 0}, {4, {9, 7}, 1}, {0, {10, 0}, 
  0}, {0, {10, 0}, 1}, {0, {10, 1}, 0}, {0, {10, 1}, 1}, {0, {10, 2}, 
  0}, {0, {10, 2}, 1}, {0, {10, 3}, 0}, {2, {10, 3}, 1}, {2, {10, 4}, 
  0}, {0, {10, 4}, 1}, {0, {10, 5}, 0}, {2, {10, 5}, 1}, {2, {10, 6}, 
  0}, {0, {10, 6}, 1}, {0, {10, 7}, 0}, {0, {11, 0}, 0}, {0, {11, 0}, 
  1}, {0, {11, 1}, 0}, {0, {11, 1}, 1}, {0, {11, 2}, 0}, {2, {11, 2}, 
  1}, {2, {11, 3}, 0}, {0, {11, 3}, 1}, {0, {11, 4}, 0}, {2, {11, 4}, 
  1}, {2, {11, 5}, 0}, {0, {11, 5}, 1}, {0, {11, 6}, 0}, {0, {12, 0}, 
  0}, {0, {12, 0}, 1}, {0, {12, 1}, 0}, {2, {12, 1}, 1}, {2, {12, 2}, 
  0}, {0, {12, 2}, 1}, {0, {12, 3}, 0}, {2, {12, 3}, 1}, {2, {12, 4}, 
  0}, {0, {12, 4}, 1}, {0, {12, 5}, 0}}
\end{scope}
\draw (240:\k) coordinate (C) -- ++(0:\k) coordinate (D) -- ++(-60:\n) coordinate (E) -- ++(-120:2*\n) coordinate (F) -- ++(180:\k) coordinate (A) -- ++(120:\n) coordinate (B) -- cycle;
\draw (B) -- ++(0:\k) coordinate (G) -- (D); \draw (F) -- (G);
\foreach\x in {1,...,\k} \path (C) ++(0:\x-0.5) node[vertex=green] {};
\foreach\x in {1,...,\k} \path (A) ++(0:\x-0.5) node[vertex=green] {};
\foreach\x in {1,...,\n} \path (C) ++(240:\x-0.5) node[vertex=lightblue] {};
\foreach\x in {1,...,\n} \path (F) ++(60:\x-0.5) node[vertex=lightblue] {};
\foreach\x/\c in {1/lightblue,2/red,3/lightblue,4/red,5/lightblue} \path (A) ++(120:\x-0.5) node[vertex=\c] {};
\foreach\x/\c in {1/lightblue,2/lightblue,3/green,4/green,5/green} \path (C) ++(240:\x+\n-0.5) node[vertex=\c] {};
\foreach\x/\c in {1/lightblue,2/lightblue,3/lightblue,4/red,5/red} \path (D) ++(-60:\x-0.5) node[vertex=\c] {}; 
\foreach\x/\c in {1/green,2/lightblue,3/green,4/green,5/lightblue} \path (E) ++(-120:\x-0.5) node[vertex=\c] {}; 
\draw[dotted] (B) ++(60:\n) -- ++(0:\k) -- ++(-60:\n);
\end{tikzpicture}
\ 
\begin{tikzpicture}[scale=0.4,baseline=-3.2cm]
\begin{scope}[puz]
\rawpuzzle{{0, {0, 2}, 1}, {0, {0, 3}, 0}, {0, {0, 3}, 1}, {0, {0, 4}, 
  0}, {0, {0, 4}, 1}, {0, {0, 5}, 0}, {2, {0, 5}, 1}, {2, {0, 6}, 
  0}, {2, {0, 6}, 1}, {2, {0, 7}, 0}, {0, {0, 7}, 1}, {0, {1, 1}, 
  1}, {0, {1, 2}, 0}, {0, {1, 2}, 1}, {0, {1, 3}, 0}, {0, {1, 3}, 
  1}, {0, {1, 4}, 0}, {2, {1, 4}, 1}, {2, {1, 5}, 0}, {2, {1, 5}, 
  1}, {2, {1, 6}, 0}, {0, {1, 6}, 1}, {0, {1, 7}, 0}, {5, {1, 7}, 
  1}, {0, {2, 0}, 1}, {0, {2, 1}, 0}, {0, {2, 1}, 1}, {0, {2, 2}, 
  0}, {0, {2, 2}, 1}, {0, {2, 3}, 0}, {2, {2, 3}, 1}, {2, {2, 4}, 
  0}, {2, {2, 4}, 1}, {2, {2, 5}, 0}, {5, {2, 5}, 1}, {4, {2, 6}, 
  0}, {4, {2, 6}, 1}, {5, {2, 7}, 0}, {0, {2, 7}, 1}, {4, {3, 0}, 
  0}, {4, {3, 0}, 1}, {4, {3, 1}, 0}, {4, {3, 1}, 1}, {4, {3, 2}, 
  0}, {4, {3, 2}, 1}, {1, {3, 3}, 0}, {1, {3, 3}, 1}, {1, {3, 4}, 
  0}, {1, {3, 4}, 1}, {5, {3, 5}, 0}, {0, {3, 5}, 1}, {0, {3, 6}, 
  0}, {0, {3, 6}, 1}, {0, {3, 7}, 0}, {0, {3, 7}, 1}, {4, {4, 0}, 
  0}, {4, {4, 0}, 1}, {4, {4, 1}, 0}, {4, {4, 1}, 1}, {4, {4, 2}, 0},
  {4, {4, 2}, 1}, {1, {4, 3}, 0}, {1, {4, 3}, 1}, {1, {4, 4}, 
  0}, {1, {4, 4}, 1}, {4, {4, 5}, 0}, {4, {4, 5}, 1}, {4, {4, 6}, 
  0}, {4, {4, 6}, 1}, {4, {4, 7}, 0}, {4, {4, 7}, 1}, {4, {5, 0}, 
  0}, {4, {5, 0}, 1}, {4, {5, 1}, 0}, {4, {5, 1}, 1}, {4, {5, 2}, 
  0}, {4, {5, 2}, 1}, {1, {5, 3}, 0}, {1, {5, 3}, 1}, {1, {5, 4}, 
  0}, {1, {5, 4}, 1}, {4, {5, 5}, 0}, {4, {5, 5}, 1}, {4, {5, 6}, 
  0}, {4, {5, 6}, 1}, {4, {5, 7}, 0}, {4, {5, 7}, 1}, {4, {6, 0}, 
  0}, {4, {6, 0}, 1}, {4, {6, 1}, 0}, {4, {6, 1}, 1}, {4, {6, 2}, 
  0}, {4, {6, 2}, 1}, {1, {6, 3}, 0}, {1, {6, 3}, 1}, {1, {6, 4}, 
  0}, {1, {6, 4}, 1}, {4, {6, 5}, 0}, {4, {6, 5}, 1}, {4, {6, 6}, 
  0}, {4, {6, 6}, 1}, {4, {6, 7}, 0}, {4, {6, 7}, 1}, {4, {7, 0}, 
  0}, {4, {7, 0}, 1}, {4, {7, 1}, 0}, {4, {7, 1}, 1}, {4, {7, 2}, 
  0}, {4, {7, 2}, 1}, {1, {7, 3}, 0}, {1, {7, 3}, 1}, {1, {7, 4}, 
  0}, {1, {7, 4}, 1}, {4, {7, 5}, 0}, {4, {7, 5}, 1}, {4, {7, 6}, 
  0}, {4, {7, 6}, 1}, {4, {7, 7}, 0}, {4, {7, 7}, 1}, {4, {8, 0}, 0}, 
 {4, {8, 0}, 1}, {4, {8, 1}, 0}, {4, {8, 1}, 1}, {4, {8, 2}, 
  0}, {4, {8, 2}, 1}, {1, {8, 3}, 0}, {1, {8, 3}, 1}, {1, {8, 4}, 
  0}, {3, {8, 4}, 1}, {0, {8, 5}, 0}, {5, {8, 5}, 1}, {4, {8, 6}, 
  0}, {4, {8, 6}, 1}, {4, {8, 7}, 0}, {4, {8, 7}, 1}, {4, {9, 0}, 
  0}, {4, {9, 0}, 1}, {4, {9, 1}, 0}, {4, {9, 1}, 1}, {4, {9, 2}, 
  0}, {4, {9, 2}, 1}, {1, {9, 3}, 0}, {3, {9, 3}, 1}, {3, {9, 4}, 
  0}, {1, {9, 4}, 1}, {5, {9, 5}, 0}, {0, {9, 5}, 1}, {0, {9, 6}, 
  0}, {5, {9, 6}, 1}, {4, {9, 7}, 0}, {4, {9, 7}, 1}, {0, {10, 0}, 
  0}, {0, {10, 0}, 1}, {0, {10, 1}, 0}, {0, {10, 1}, 1}, {0, {10, 2}, 
  0}, {2, {10, 2}, 1}, {8, {10, 3}, 0}, {4, {10, 3}, 1}, {1, {10, 4}, 
  0}, {3, {10, 4}, 1}, {0, {10, 5}, 0}, {5, {10, 5}, 1}, {5, {10, 6}, 
  0}, {0, {10, 6}, 1}, {0, {10, 7}, 0}, {0, {11, 0}, 0}, {0, {11, 0}, 
  1}, {0, {11, 1}, 0}, {2, {11, 1}, 1}, {2, {11, 2}, 0}, {0, {11, 2}, 
  1}, {0, {11, 3}, 0}, {2, {11, 3}, 1}, {8, {11, 4}, 0}, {4, {11, 4}, 
  1}, {5, {11, 5}, 0}, {0, {11, 5}, 1}, {0, {11, 6}, 0}, {0, {12, 0}, 
  0}, {0, {12, 0}, 1}, {7, {12, 1}, 0}, {7, {12, 1}, 1}, {0, {12, 2}, 
  0}, {0, {12, 2}, 1}, {7, {12, 3}, 0}, {7, {12, 3}, 1}, {0, {12, 4}, 
  0}, {0, {12, 4}, 1}, {0, {12, 5}, 0}}
\end{scope}
\draw (240:\k) coordinate (C) -- ++(0:\k) coordinate (D) -- ++(-60:\n) coordinate (E) -- ++(-120:2*\n) coordinate (F) -- ++(180:\k) coordinate (A) -- ++(120:\n) coordinate (B) -- cycle;
\draw (A) -- ++(60:2*\n) coordinate (G) -- (C); \draw (E) -- (G);
\foreach\x in {1,...,\k} \path (C) ++(0:\x-0.5) node[vertex=green] {};
\foreach\x in {1,...,\k} \path (A) ++(0:\x-0.5) node[vertex=green] {};
\foreach\x in {1,...,\n} \path (C) ++(240:\x-0.5) node[vertex=lightblue] {};
\foreach\x in {1,...,\n} \path (F) ++(60:\x-0.5) node[vertex=lightblue] {};
\foreach\x/\c in {1/lightblue,2/red,3/lightblue,4/red,5/lightblue} \path (A) ++(120:\x-0.5) node[vertex=\c] {};
\foreach\x/\c in {1/lightblue,2/lightblue,3/green,4/green,5/green} \path (C) ++(240:\x+\n-0.5) node[vertex=\c] {};
\foreach\x/\c in {1/lightblue,2/lightblue,3/lightblue,4/red,5/red} \path (D) ++(-60:\x-0.5) node[vertex=\c] {}; 
\foreach\x/\c in {1/green,2/lightblue,3/green,4/green,5/lightblue} \path (E) ++(-120:\x-0.5) node[vertex=\c] {}; 
\draw[dotted] (B) ++(60:\n) -- ++(-60:\n) -- ++(0:\k);
\end{tikzpicture}
\end{center}
\end{samepage}

Rewriting the master identity \eqref{eq:master}:
\[
\sum_\nu \loz{\nu/->,\varnothing/->,\tilde\mu/->,\lambda/->}{\yy}{\z} G_\nu(\x;\z)
=
\sum_{\tilde\nu} \loz{\varnothing/->,\varnothing/->,\tilde\nu/->,\lambda/->}{\yy}{\y} G_{\tilde\mu{}^\ast}(\x;\z) G^{\tilde\nu{}^\ast}(\x;\y)
\]
we note that we can apply the relation \eqref{eq:preMSalt} (a form of Theorem~\ref{thm:MSalt}) to its r.h.s.,
resulting in
\[
\sum_\nu \loz{\nu/->,\varnothing/->,\tilde\mu/->,\lambda/->}{\yy}{\z} G_\nu(\x;\z)
=
G_{\lambda^\ast}(\x;\yy) G_{\tilde\mu{}^\ast}(\x;\z)
\]
With obvious changes of notation, we obtain
\[
\sum_\nu \loz{\nu/->,\varnothing/->,\mu/<-,\lambda/<-}{\z}{\y} G_\nu(\x;\y)
=
G_{\lambda}(\x;\z) G_{\mu}(\x;\y)
\]
which is exactly the content of Theorem~\ref{thm:MS}.

We now further specialize to $\z=\y$; the top half of the puzzle becomes frozen according to Lemma~\ref{lem:split},
and we immediately reproduce Theorem~\ref{thm:equivdual}.

\begin{samepage}
\subsubsection{$\lambda=\tilde\mu=\varnothing$, $\t=\zz$.}\label{sec:lastcase}
\begin{center}
\begin{tikzpicture}[scale=0.4,baseline=-3.2cm]
\begin{scope}[puz]
\rawpuzzle{{0, {0, 2}, 1}, {0, {0, 3}, 0}, {0, {0, 3}, 1}, {7, {0, 4}, 
  0}, {7, {0, 4}, 1}, {0, {0, 5}, 0}, {0, {0, 5}, 1}, {7, {0, 6}, 
  0}, {7, {0, 6}, 1}, {0, {0, 7}, 0}, {0, {0, 7}, 1}, {0, {1, 1}, 
  1}, {0, {1, 2}, 0}, {0, {1, 2}, 1}, {0, {1, 3}, 0}, {2, {1, 3}, 
  1}, {2, {1, 4}, 0}, {0, {1, 4}, 1}, {0, {1, 5}, 0}, {2, {1, 5}, 
  1}, {2, {1, 6}, 0}, {0, {1, 6}, 1}, {0, {1, 7}, 0}, {0, {1, 7}, 
  1}, {0, {2, 0}, 1}, {0, {2, 1}, 0}, {5, {2, 1}, 1}, {4, {2, 2}, 
  0}, {4, {2, 2}, 1}, {1, {2, 3}, 0}, {1, {2, 3}, 1}, {4, {2, 4}, 
  0}, {4, {2, 4}, 1}, {1, {2, 5}, 0}, {3, {2, 5}, 1}, {0, {2, 6}, 
  0}, {0, {2, 6}, 1}, {0, {2, 7}, 0}, {0, {2, 7}, 1}, {4, {3, 0}, 
  0}, {4, {3, 0}, 1}, {5, {3, 1}, 0}, {5, {3, 1}, 1}, {4, {3, 2}, 
  0}, {4, {3, 2}, 1}, {1, {3, 3}, 0}, {3, {3, 3}, 1}, {0, {3, 4}, 
  0}, {2, {3, 4}, 1}, {8, {3, 5}, 0}, {4, {3, 5}, 1}, {4, {3, 6}, 
  0}, {4, {3, 6}, 1}, {4, {3, 7}, 0}, {4, {3, 7}, 1}, {4, {4, 0}, 
  0}, {4, {4, 0}, 1}, {5, {4, 1}, 0}, {0, {4, 1}, 1}, {0, {4, 2}, 0},
  {0, {4, 2}, 1}, {3, {4, 3}, 0}, {1, {4, 3}, 1}, {1, {4, 4}, 
  0}, {1, {4, 4}, 1}, {4, {4, 5}, 0}, {4, {4, 5}, 1}, {4, {4, 6}, 
  0}, {4, {4, 6}, 1}, {4, {4, 7}, 0}, {4, {4, 7}, 1}, {4, {5, 0}, 
  0}, {4, {5, 0}, 1}, {4, {5, 1}, 0}, {4, {5, 1}, 1}, {4, {5, 2}, 
  0}, {4, {5, 2}, 1}, {1, {5, 3}, 0}, {1, {5, 3}, 1}, {1, {5, 4}, 
  0}, {1, {5, 4}, 1}, {4, {5, 5}, 0}, {4, {5, 5}, 1}, {4, {5, 6}, 
  0}, {4, {5, 6}, 1}, {4, {5, 7}, 0}, {4, {5, 7}, 1}, {4, {6, 0}, 
  0}, {4, {6, 0}, 1}, {4, {6, 1}, 0}, {4, {6, 1}, 1}, {4, {6, 2}, 
  0}, {4, {6, 2}, 1}, {1, {6, 3}, 0}, {1, {6, 3}, 1}, {1, {6, 4}, 
  0}, {1, {6, 4}, 1}, {4, {6, 5}, 0}, {4, {6, 5}, 1}, {4, {6, 6}, 
  0}, {4, {6, 6}, 1}, {4, {6, 7}, 0}, {4, {6, 7}, 1}, {4, {7, 0}, 
  0}, {4, {7, 0}, 1}, {4, {7, 1}, 0}, {4, {7, 1}, 1}, {4, {7, 2}, 
  0}, {4, {7, 2}, 1}, {1, {7, 3}, 0}, {1, {7, 3}, 1}, {1, {7, 4}, 
  0}, {1, {7, 4}, 1}, {4, {7, 5}, 0}, {4, {7, 5}, 1}, {4, {7, 6}, 
  0}, {4, {7, 6}, 1}, {4, {7, 7}, 0}, {4, {7, 7}, 1}, {4, {8, 0}, 0}, 
 {4, {8, 0}, 1}, {4, {8, 1}, 0}, {4, {8, 1}, 1}, {4, {8, 2}, 
  0}, {4, {8, 2}, 1}, {1, {8, 3}, 0}, {1, {8, 3}, 1}, {1, {8, 4}, 
  0}, {1, {8, 4}, 1}, {4, {8, 5}, 0}, {4, {8, 5}, 1}, {4, {8, 6}, 
  0}, {4, {8, 6}, 1}, {4, {8, 7}, 0}, {4, {8, 7}, 1}, {0, {9, 0}, 
  0}, {0, {9, 0}, 1}, {0, {9, 1}, 0}, {0, {9, 1}, 1}, {0, {9, 2}, 
  0}, {5, {9, 2}, 1}, {1, {9, 3}, 0}, {1, {9, 3}, 1}, {1, {9, 4}, 
  0}, {1, {9, 4}, 1}, {4, {9, 5}, 0}, {4, {9, 5}, 1}, {4, {9, 6}, 
  0}, {4, {9, 6}, 1}, {4, {9, 7}, 0}, {4, {9, 7}, 1}, {0, {10, 0}, 
  0}, {5, {10, 0}, 1}, {4, {10, 1}, 0}, {4, {10, 1}, 1}, {5, {10, 2}, 
  0}, {2, {10, 2}, 1}, {2, {10, 3}, 0}, {2, {10, 3}, 1}, {2, {10, 4}, 
  0}, {0, {10, 4}, 1}, {0, {10, 5}, 0}, {0, {10, 5}, 1}, {0, {10, 6}, 
  0}, {0, {10, 6}, 1}, {0, {10, 7}, 0}, {5, {11, 0}, 0}, {0, {11, 0}, 
  1}, {0, {11, 1}, 0}, {2, {11, 1}, 1}, {2, {11, 2}, 0}, {2, {11, 2}, 
  1}, {2, {11, 3}, 0}, {0, {11, 3}, 1}, {0, {11, 4}, 0}, {0, {11, 4}, 
  1}, {0, {11, 5}, 0}, {0, {11, 5}, 1}, {0, {11, 6}, 0}, {0, {12, 0}, 
  0}, {2, {12, 0}, 1}, {2, {12, 1}, 0}, {2, {12, 1}, 1}, {2, {12, 2}, 
  0}, {0, {12, 2}, 1}, {0, {12, 3}, 0}, {0, {12, 3}, 1}, {0, {12, 4}, 
  0}, {0, {12, 4}, 1}, {0, {12, 5}, 0}}
\end{scope}
\draw (240:\k) coordinate (C) -- ++(0:\k) coordinate (D) -- ++(-60:\n) coordinate (E) -- ++(-120:2*\n) coordinate (F) -- ++(180:\k) coordinate (A) -- ++(120:\n) coordinate (B) -- cycle;
\draw (B) -- ++(0:\k) coordinate (G) -- (D); \draw (F) -- (G);
\foreach\x in {1,...,\k} \path (C) ++(0:\x-0.5) node[vertex=green] {};
\foreach\x in {1,...,\k} \path (A) ++(0:\x-0.5) node[vertex=green] {};
\foreach\x in {1,...,\n} \path (C) ++(240:\x-0.5) node[vertex=lightblue] {};
\foreach\x in {1,...,\n} \path (F) ++(60:\x-0.5) node[vertex=lightblue] {};
\foreach\x/\c in {1/lightblue,2/lightblue,3/lightblue,4/red,5/red} \path (A) ++(120:\x-0.5) node[vertex=\c] {};
\foreach\x/\c in {1/lightblue,2/green,3/green,4/lightblue,5/green} \path (C) ++(240:\x+\n-0.5) node[vertex=\c] {};
\foreach\x/\c in {1/lightblue,2/red,3/lightblue,4/red,5/lightblue} \path (D) ++(-60:\x-0.5) node[vertex=\c] {}; 
\foreach\x/\c in {1/green,2/green,3/green,4/lightblue,5/lightblue} \path (E) ++(-120:\x-0.5) node[vertex=\c] {}; 
\draw[dotted] (B) ++(60:\n) -- ++(0:\k) -- ++(-60:\n);
\end{tikzpicture}
\ 
\begin{tikzpicture}[scale=0.4,baseline=-3.2cm]
\begin{scope}[puz]
\rawpuzzle{{0, {0, 2}, 1}, {0, {0, 3}, 0}, {2, {0, 3}, 1}, {2, {0, 4}, 
  0}, {0, {0, 4}, 1}, {0, {0, 5}, 0}, {2, {0, 5}, 1}, {2, {0, 6}, 
  0}, {0, {0, 6}, 1}, {0, {0, 7}, 0}, {0, {0, 7}, 1}, {0, {1, 1}, 
  1}, {0, {1, 2}, 0}, {2, {1, 2}, 1}, {2, {1, 3}, 0}, {0, {1, 3}, 
  1}, {0, {1, 4}, 0}, {2, {1, 4}, 1}, {2, {1, 5}, 0}, {0, {1, 5}, 
  1}, {0, {1, 6}, 0}, {0, {1, 6}, 1}, {0, {1, 7}, 0}, {0, {1, 7}, 
  1}, {0, {2, 0}, 1}, {0, {2, 1}, 0}, {2, {2, 1}, 1}, {2, {2, 2}, 
  0}, {0, {2, 2}, 1}, {0, {2, 3}, 0}, {2, {2, 3}, 1}, {2, {2, 4}, 
  0}, {0, {2, 4}, 1}, {0, {2, 5}, 0}, {0, {2, 5}, 1}, {0, {2, 6}, 
  0}, {0, {2, 6}, 1}, {0, {2, 7}, 0}, {0, {2, 7}, 1}, {4, {3, 0}, 
  0}, {4, {3, 0}, 1}, {1, {3, 1}, 0}, {1, {3, 1}, 1}, {4, {3, 2}, 
  0}, {4, {3, 2}, 1}, {1, {3, 3}, 0}, {1, {3, 3}, 1}, {4, {3, 4}, 
  0}, {4, {3, 4}, 1}, {4, {3, 5}, 0}, {4, {3, 5}, 1}, {4, {3, 6}, 
  0}, {4, {3, 6}, 1}, {4, {3, 7}, 0}, {4, {3, 7}, 1}, {4, {4, 0}, 
  0}, {4, {4, 0}, 1}, {1, {4, 1}, 0}, {1, {4, 1}, 1}, {4, {4, 2}, 0},
  {4, {4, 2}, 1}, {1, {4, 3}, 0}, {1, {4, 3}, 1}, {4, {4, 4}, 
  0}, {4, {4, 4}, 1}, {4, {4, 5}, 0}, {4, {4, 5}, 1}, {4, {4, 6}, 
  0}, {4, {4, 6}, 1}, {4, {4, 7}, 0}, {4, {4, 7}, 1}, {4, {5, 0}, 
  0}, {4, {5, 0}, 1}, {1, {5, 1}, 0}, {1, {5, 1}, 1}, {4, {5, 2}, 
  0}, {4, {5, 2}, 1}, {1, {5, 3}, 0}, {1, {5, 3}, 1}, {4, {5, 4}, 
  0}, {4, {5, 4}, 1}, {4, {5, 5}, 0}, {4, {5, 5}, 1}, {4, {5, 6}, 
  0}, {4, {5, 6}, 1}, {4, {5, 7}, 0}, {4, {5, 7}, 1}, {4, {6, 0}, 
  0}, {4, {6, 0}, 1}, {1, {6, 1}, 0}, {1, {6, 1}, 1}, {4, {6, 2}, 
  0}, {4, {6, 2}, 1}, {1, {6, 3}, 0}, {1, {6, 3}, 1}, {4, {6, 4}, 
  0}, {4, {6, 4}, 1}, {4, {6, 5}, 0}, {4, {6, 5}, 1}, {4, {6, 6}, 
  0}, {4, {6, 6}, 1}, {4, {6, 7}, 0}, {4, {6, 7}, 1}, {4, {7, 0}, 
  0}, {4, {7, 0}, 1}, {1, {7, 1}, 0}, {1, {7, 1}, 1}, {4, {7, 2}, 
  0}, {4, {7, 2}, 1}, {1, {7, 3}, 0}, {1, {7, 3}, 1}, {4, {7, 4}, 
  0}, {4, {7, 4}, 1}, {4, {7, 5}, 0}, {4, {7, 5}, 1}, {4, {7, 6}, 
  0}, {4, {7, 6}, 1}, {4, {7, 7}, 0}, {4, {7, 7}, 1}, {4, {8, 0}, 0}, 
 {4, {8, 0}, 1}, {1, {8, 1}, 0}, {3, {8, 1}, 1}, {0, {8, 2}, 
  0}, {0, {8, 2}, 1}, {7, {8, 3}, 0}, {7, {8, 3}, 1}, {0, {8, 4}, 
  0}, {0, {8, 4}, 1}, {0, {8, 5}, 0}, {0, {8, 5}, 1}, {0, {8, 6}, 
  0}, {0, {8, 6}, 1}, {0, {8, 7}, 0}, {5, {8, 7}, 1}, {0, {9, 0}, 
  0}, {2, {9, 0}, 1}, {8, {9, 1}, 0}, {4, {9, 1}, 1}, {4, {9, 2}, 
  0}, {4, {9, 2}, 1}, {1, {9, 3}, 0}, {3, {9, 3}, 1}, {0, {9, 4}, 
  0}, {0, {9, 4}, 1}, {0, {9, 5}, 0}, {5, {9, 5}, 1}, {4, {9, 6}, 
  0}, {4, {9, 6}, 1}, {5, {9, 7}, 0}, {5, {9, 7}, 1}, {7, {10, 0}, 
  0}, {7, {10, 0}, 1}, {0, {10, 1}, 0}, {0, {10, 1}, 1}, {0, {10, 2}, 
  0}, {2, {10, 2}, 1}, {8, {10, 3}, 0}, {4, {10, 3}, 1}, {4, {10, 4}, 
  0}, {4, {10, 4}, 1}, {5, {10, 5}, 0}, {5, {10, 5}, 1}, {4, {10, 6}, 
  0}, {4, {10, 6}, 1}, {5, {10, 7}, 0}, {1, {11, 0}, 0}, {1, {11, 0}, 
  1}, {4, {11, 1}, 0}, {4, {11, 1}, 1}, {1, {11, 2}, 0}, {1, {11, 2}, 
  1}, {4, {11, 3}, 0}, {4, {11, 3}, 1}, {4, {11, 4}, 0}, {4, {11, 4}, 
  1}, {5, {11, 5}, 0}, {0, {11, 5}, 1}, {0, {11, 6}, 0}, {7, {12, 0}, 
  0}, {7, {12, 0}, 1}, {0, {12, 1}, 0}, {2, {12, 1}, 1}, {2, {12, 2}, 
  0}, {0, {12, 2}, 1}, {0, {12, 3}, 0}, {0, {12, 3}, 1}, {0, {12, 4}, 
  0}, {0, {12, 4}, 1}, {0, {12, 5}, 0}}
\end{scope}
\draw (240:\k) coordinate (C) -- ++(0:\k) coordinate (D) -- ++(-60:\n) coordinate (E) -- ++(-120:2*\n) coordinate (F) -- ++(180:\k) coordinate (A) -- ++(120:\n) coordinate (B) -- cycle;
\draw (A) -- ++(60:2*\n) coordinate (G) -- (C); \draw (E) -- (G);
\foreach\x in {1,...,\k} \path (C) ++(0:\x-0.5) node[vertex=green] {};
\foreach\x in {1,...,\k} \path (A) ++(0:\x-0.5) node[vertex=green] {};
\foreach\x in {1,...,\n} \path (C) ++(240:\x-0.5) node[vertex=lightblue] {};
\foreach\x in {1,...,\n} \path (F) ++(60:\x-0.5) node[vertex=lightblue] {};
\foreach\x/\c in {1/lightblue,2/lightblue,3/lightblue,4/red,5/red} \path (A) ++(120:\x-0.5) node[vertex=\c] {};
\foreach\x/\c in {1/lightblue,2/green,3/green,4/lightblue,5/green} \path (C) ++(240:\x+\n-0.5) node[vertex=\c] {};
\foreach\x/\c in {1/lightblue,2/red,3/lightblue,4/red,5/lightblue} \path (D) ++(-60:\x-0.5) node[vertex=\c] {}; 
\foreach\x/\c in {1/green,2/green,3/green,4/lightblue,5/lightblue} \path (E) ++(-120:\x-0.5) node[vertex=\c] {}; 
\draw[dotted] (B) ++(60:\n) -- ++(-60:\n) -- ++(0:\k);
\end{tikzpicture}
\end{center}
\end{samepage}

The final case of \eqref{eq:master} we consider is
\[
\sum_\nu \loz{\nu/->,\tilde\lambda/->,\varnothing/->,\varnothing/->}{\zz}{\z} G_\nu(\x;\z) G^\mu(\x;\y)
=
\sum_{\tilde\nu} \loz{\mu/->,\tilde\lambda/->,\tilde\nu/->,\varnothing/->}{\zz}{\y} G_{\varnothing^\ast}(\x;\z) G^{\tilde\nu{}^\ast}(\x;\y).
\]
By rewriting the l.h.s.\ as
\[
\sum_\nu \loz{\nu/->,\tilde\lambda/->,\varnothing/->,\varnothing/->}{\zz}{\z} \prod_{i=1}^k x_i
\prod_{i\in\nu}z_i^{-1}
\,G^{\nu^\ast}(\x;\zz) G^\mu(\x;\y)
\]
we see that we can apply to it the identity~\eqref{eq:preequivalt}, resulting after simplifying
by $G_{\varnothing^\ast}(\x;\z)$ on both sides in
\begin{equation}\label{eq:prequiv}
G^{\tilde\lambda}(\x;\zz) G^\mu(\x;\y)
\prod_{i=1}^k x_i \prod_{i=1}^k z_i^{-1}
=
\sum_{\tilde\nu} \loz{\mu/->,\tilde\lambda/->,\tilde\nu/->,\varnothing/->}{\zz}{\y} G^{\tilde\nu{}^\ast}(\x;\y).
\end{equation}

We now set $\z=\yy$, and apply one last time Lemma~\ref{lem:split} to the r.h.s.:
\[
G^{\tilde\lambda}(\x;\zz) G^\mu(\x;\y)
\prod_{i=1}^k x_i \prod_{i\in\sigma} y_i^{-1}
=
\sum_{\tilde\nu,\sigma:\, \sigma\vartriangleleft\tilde\nu{}^{\ast}}
\triup{\mu/->,\tilde\lambda/->,\sigma/<-}{\y}
(-1)^{|\tilde\nu{}^\ast|-|\sigma|}
G^{\tilde\nu{}^\ast}(\x;\y).
\]
Thanks to the identity~\eqref{eq:id}, we can perform the summation over $\tilde\nu$, and obtain
\[
G^{\tilde\lambda}(\x;\y)G^\mu(\x;\y)=
\sum_{\sigma} 
\triup{\mu/->,\tilde\lambda/->,\sigma/<-}{\y}
G^\sigma(\x;\y)
\]
which is nothing but Theorem~\ref{thm:equiv}.


\subsection{Stability}\label{sec:stab}
The results of the previous section were based on the identity \eqref{eq:master}, which is only valid
for $n$ large enough. We now show that our theorems remain true for all values of $n$ for which they make
sense, \ie such that all Young diagrams have width at most $n-k$; this is related to ``stability'' of the various quantities as one varies $n$.

The first observation is that given an (abstractly defined, \ie without bounding box) Young diagram
$\lambda$ with at most $k$ rows, the polynomial
$G^\lambda(\x;\y)$ as defined in Sect.~\ref{ssec:groth} is independent
of the choice of $n$ (with the condition mentioned above that $n-k\geq w(\lambda)$). This can be readily seen from Prop.~\ref{prop:groth}, since the location of
the rightmost green dot on the top row is precisely $w(\lambda)+k\leq n$. Anything to its right is empty, which
shows the stability property. We also conclude that the dependence of
$G^\lambda$ on its secondary alphabet is only on $y_1,\ldots,y_{w(\lambda)+k-1}$.

Next, we consider the equality of Theorem~\ref{thm:equiv}, which was proven in the previous section for
large enough $n$. Since all Grothendieck polynomials satisfy the stability above, the $c^{\lambda,\mu}_{\,\,\nu}(\y)$
are independent of $n$. Now consider an equivariant puzzle, \href{http://www.lpthe.jussieu.fr/~pzinn/puzzles/?height=5&width=5&y1=1&y2=2,1&y3=2,1,1&y3comp&K&equiv&mask=35&process}{\eg}, for $n=10$, $k=5$,
\[
\def\size{10}
\begin{tikzpicture}[scale=0.5]
\equivpuzzle%
{{4, {0, 0}, 0}, {4, {0, 0}, 1}, {4, {0, 1}, 0}, {4, {0, 1}, 
  1}, {4, {0, 2}, 0}, {4, {0, 2}, 1}, {1, {0, 3}, 0}, {1, {0, 3}, 
  1}, {4, {0, 4}, 0}, {4, {0, 4}, 1}, {1, {0, 5}, 0}, {1, {0, 5}, 
  1}, {4, {0, 6}, 0}, {4, {0, 6}, 1}, {1, {0, 7}, 0}, {1, {0, 7}, 
  1}, {1, {0, 8}, 0}, {1, {0, 8}, 1}, {1, {0, 9}, 0}, {4, {1, 0}, 
  0}, {4, {1, 0}, 1}, {4, {1, 1}, 0}, {4, {1, 1}, 1}, {4, {1, 2}, 
  0}, {4, {1, 2}, 1}, {1, {1, 3}, 0}, {1, {1, 3}, 1}, {4, {1, 4}, 
  0}, {4, {1, 4}, 1}, {1, {1, 5}, 0}, {1, {1, 5}, 1}, {4, {1, 6}, 
  0}, {4, {1, 6}, 1}, {1, {1, 7}, 0}, {1, {1, 7}, 1}, {1, {1, 8}, 
  0}, {4, {2, 0}, 0}, {4, {2, 0}, 1}, {4, {2, 1}, 0}, {4, {2, 1}, 
  1}, {4, {2, 2}, 0}, {4, {2, 2}, 1}, {1, {2, 3}, 0}, {1, {2, 3}, 
  1}, {4, {2, 4}, 0}, {4, {2, 4}, 1}, {1, {2, 5}, 0}, {1, {2, 5}, 
  1}, {4, {2, 6}, 0}, {4, {2, 6}, 1}, {1, {2, 7}, 0}, {4, {3, 0}, 
  0}, {4, {3, 0}, 1}, {4, {3, 1}, 0}, {4, {3, 1}, 1}, {4, {3, 2}, 
  0}, {4, {3, 2}, 1}, {1, {3, 3}, 0}, {3, {3, 3}, 1}, {0, {3, 4}, 0},
  {0, {3, 4}, 1}, {7, {3, 5}, 0}, {7, {3, 5}, 1}, {0, {3, 6}, 
  0}, {0, {4, 0}, 0}, {0, {4, 0}, 1}, {0, {4, 1}, 0}, {0, {4, 1}, 
  1}, {0, {4, 2}, 0}, {2, {4, 2}, 1}, {8, {4, 3}, 0}, {4, {4, 3}, 
  1}, {4, {4, 4}, 0}, {4, {4, 4}, 1}, {1, {4, 5}, 0}, {4, {5, 0}, 
  0}, {4, {5, 0}, 1}, {4, {5, 1}, 0}, {4, {5, 1}, 1}, {1, {5, 2}, 
  0}, {3, {5, 2}, 1}, {0, {5, 3}, 0}, {0, {5, 3}, 1}, {0, {5, 4}, 
  0}, {0, {6, 0}, 0}, {0, {6, 0}, 1}, {0, {6, 1}, 0}, {0, {6, 1}, 
  1}, {3, {6, 2}, 0}, {3, {6, 2}, 1}, {0, {6, 3}, 0}, {0, {7, 0}, 
  0}, {0, {7, 0}, 1}, {0, {7, 1}, 0}, {0, {7, 1}, 1}, {3, {7, 2}, 
  0}, {0, {8, 0}, 0}, {0, {8, 0}, 1}, {0, {8, 1}, 0}, {0, {9, 0}, 0}}
\begin{scope}[puz]
\draw[<-] (6.5,-0.75) -- node[above right] {$\tableau{&\\\\}$} (4.5,-0.75);
\draw[<-] (6.25,4.5) -- node[below] {$\tableau{&\\\\\\}$} (4.5,6.25);
\draw[<-] (-0.75,4.5) -- node[above left] {$\tableau{\\}$} (-0.75,6.5);
\end{scope}
\end{tikzpicture}
\]
It is clear that the width of the bottom Young diagram $\nu$ determines the frozen region, namely,
anything right of the line coming out of the rightmost bottom green dot is frozen.
Therefore, we can truncate the puzzle around its left corner to a sub-puzzle of size
$w(\nu)+k$ (in the example, $7$), with same Young diagrams $\lambda,\mu,\nu$; 
the transformation preserves weights.

We conclude that puzzles and their weights are also stable, and therefore, the sum over puzzles is equal
to $c^{\lambda,\mu}_{\,\,\nu}(\y)$.

The other theorems can be proven similarly, and we shall not repeat the reasoning. For Theorems~\ref{thm:equivdual} and \ref{thm:MS}, one should note that dual Grothendieck polynomials are stable only after substitution
$\y\mapsto\yy$ and $\lambda\mapsto\lambda^\ast$, due to their very definition (this can also be obtained from
Prop.~\ref{prop:groth}). With such substitutions, the proper stability conditions
can be proven. Theorems~\ref{thm:equivalt} and \ref{thm:equivdualalt} follow (without any restriction on $n$) from identity~\eqref{eq:preequivalt} and Theorem~\ref{thm:MSalt}, respectively; the stability of the latter is again proven straightforwardly.

\section{Discussion}
\label{sec:discuss}

\subsection{Higher expectation values}
The main results of this paper concern the calculation of certain ``three-point functions'' in the ring $K_T(Gr(k,n))$, \ie $\left<G^\lambda G^\mu G_\nu\right>$
and $\left<G_\lambda G_\mu G^\nu\right>$.
Of course, higher expectation values follow immediately by repeated insertion of the decomposition of the identity
$1=\sum_\lambda \ket{G^\lambda}\bra{G_\lambda}$, recalling that the $\{G^\lambda\}$ and $\{G_\lambda\}$ form dual bases.

\subsubsection{The $R$-matrix}
In order to do so, it is convenient to switch to the dual graphical notation that is more usual in integrable models.

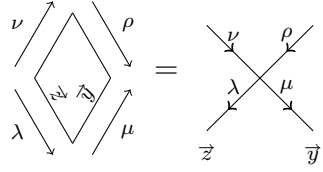
\[
\loz{\nu/->,\rho/->,\mu/<-,\lambda/<-}{\z}{\y}
=
\begin{tikzpicture}[baseline=0,scale=0.7]
\draw[arrow=0.25,arrow=0.8] (-1,1) -- node[pos=0.25,above] {$\ss\nu$} node[pos=0.75,above] {$\ss\mu$} (1,-1) node[below] {$\ss\y$};
\draw[arrow=0.25,arrow=0.8] (1,1) -- node[pos=0.25,above] {$\ss\rho$} node[pos=0.75,above] {$\ss\lambda$} (-1,-1) node[below] {$\ss\z$};
\end{tikzpicture}
\]
The r.h.s.\ lines require an orientation, because weights are not invariant by 180 degree rotation.
This orientation of the partitions in the l.h.s.\ is conventionally recovered from that of the r.h.s.\ by a 90 degree counterclockwise rotation. Each line should be thought of
as the space $(\mathbb C^3)^{\otimes n}$, with spectral/equivariant parameters $\z$ or $\y$ attached to each copy
of $\mathbb C^3$.


\subsubsection{Equivariant setting}
\tikzset{mydot/.style={circle,draw=black,fill=#1,inner sep=2pt}}
In these notations, the content of Lemma~\ref{lem:split} amounts essentially to the picture
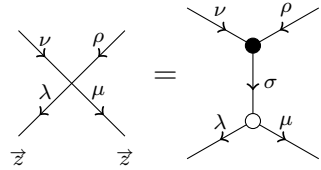
\[
\begin{tikzpicture}[baseline=0,scale=0.7]
\draw[arrow=0.25,arrow=0.8] (-1,1) -- node[pos=0.25,above] {$\ss\nu$} node[pos=0.75,above] {$\ss\mu$} (1,-1) node[below] {$\ss\z$};
\draw[arrow=0.25,arrow=0.8] (1,1) -- node[pos=0.25,above] {$\ss\rho$} node[pos=0.75,above] {$\ss\lambda$} (-1,-1) node[below] {$\ss\z$};
\end{tikzpicture}
=
\begin{tikzpicture}[baseline=0.5cm]
\draw[arrow] (0,0) -- node[above] {$\ss\lambda$} (210:1);
\draw[arrow] (0,0) -- node[above] {$\ss\mu$} (330:1);
\draw[invarrow] (0,0) -- node[right] {$\ss\sigma$} (90:1);
\draw[invarrow] (90:1) -- node[above] {$\ss\rho$} ++(30:1);
\draw[invarrow] (90:1) -- node[above] {$\ss\nu$} ++(150:1);
\node[mydot=white] (a) at (0,0) {};
\node[mydot=black] (a) at (0,1) {};
\end{tikzpicture}
\]
with the obvious notations
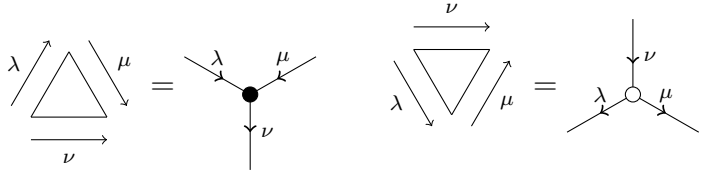
\[
\triup{\lambda/->,\mu/->,\nu/<-}{}
=
\begin{tikzpicture}[baseline=0]
\draw[invarrow] (0,0) -- node[above] {$\ss\mu$} (30:1);
\draw[invarrow] (0,0) -- node[above] {$\ss\lambda$} (150:1);
\draw[arrow] (0,0) -- node[right] {$\ss\nu$} (270:1);
\node[mydot=black] (a) at (0,0) {};
\end{tikzpicture}
\qquad
\tridown{\nu/->,\mu/<-,\lambda/<-}{}
=
\begin{tikzpicture}[baseline=0]
\draw[arrow] (0,0) -- node[above] {$\ss\lambda$} (210:1);
\draw[arrow] (0,0) -- node[above] {$\ss\mu$} (330:1);
\draw[invarrow] (0,0) -- node[right] {$\ss\nu$} (90:1);
\node[mydot=white] (a) at (0,0) {};
\end{tikzpicture}
\qquad
\]
and where we have removed spectral parameters since they are everywhere given by $\z$.
The circles at vertices are in principle redundant at this stage.
Note that this factorization of the $R$-matrix occurs in the general $U_q(\widehat{\mathfrak{sl}(3)})$ 
$R$-matrix in~\cite{Resh-On}; our $R$-matrix is a degenerate version of it,
and at ratio of spectral parameters equal to $1$, our model is a degenerate version of the loop model on the triangular
lattice considered in~\cite{Resh-On}.

Gluing together these vertices (and respecting the orientation of each edge, since the two dual
bases $\{ G^\lambda \}$ and $\{ G_\lambda \}$ are distinct),
we produce {\em planar trees}\/ where each vertex is colored white or black.
An important point is that bipartiteness, which was automatic when spectral parameters are generic
(which in particular explains why we do not have a Molev--Sagan type rule for [non-dual] double Grothendieck
polynomials), is no longer required.

For example, computing the coefficient of $G^\rho$ in the triple product $G^\lambda G^\mu G^\nu$ amounts to one of the two diagrams
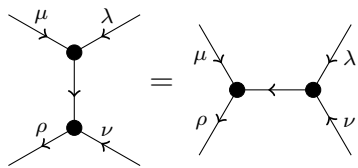
\[
\begin{tikzpicture}[baseline=0.5cm]
\draw[arrow] (0,0) -- node[above] {$\ss\rho$} (210:1);
\draw[invarrow] (0,0) -- node[above] {$\ss\nu$} (330:1);
\draw[invarrow] (0,0) -- (90:1);
\draw[invarrow] (90:1) -- node[above] {$\ss\lambda$} ++(30:1);
\draw[invarrow] (90:1) -- node[above] {$\ss\mu$} ++(150:1);
\node[mydot=black] (a) at (0,0) {};
\node[mydot=black] (a) at (0,1) {};
\end{tikzpicture}
=
\begin{tikzpicture}[rotate=-90,baseline=0cm]
\draw[invarrow] (0,0) -- node[left] {$\ss\mu$} (210:1);
\draw[arrow] (0,0) -- node[left] {$\ss\rho$} (330:1);
\draw[invarrow] (0,0) -- (90:1);
\draw[invarrow] (90:1) -- node[right] {$\ss\nu$} ++(30:1);
\draw[invarrow] (90:1) -- node[right] {$\ss\lambda$} ++(150:1);
\node[mydot=black] (a) at (0,0) {};
\node[mydot=black] (a) at (0,1) {};
\end{tikzpicture}
\]
If one insists on returning to the triangular lattice, then one should be careful that some puzzles are rotated $\pm 120$ degrees, requiring to use the tiles of table \eqref{eq:table}; \eg, for the first of the two diagrams above with $\lambda=\mu=\tableau{\\}$, $\nu=\rho=\tableau{&\\}$, we find
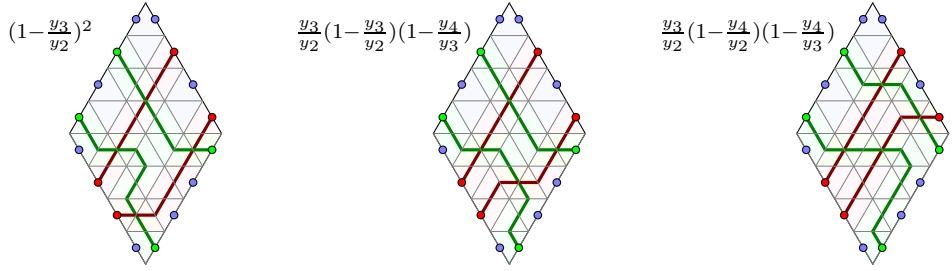
\begin{center}%
\def\size{4}\def\puzzlescale{0.5}%
\doublepuzzle[\node at (-\size*0.5,\size*0.75) {$\ss (1-\frac{y_3}{y_2})^2$};]{{4, {0, 0}, 0}, {4, {0, 0}, 1}, {1, {0, 1}, 0}, {1, {0, 1}, 
   1}, {4, {0, 2}, 0}, {4, {0, 2}, 1}, {1, {0, 3}, 0}, {3, {0, 3}, 
   1}, {0, {1, 0}, 0}, {0, {1, 0}, 1}, {7, {1, 1}, 0}, {7, {1, 1}, 
   1}, {0, {1, 2}, 0}, {0, {1, 2}, 1}, {3, {1, 3}, 0}, {1, {1, 3}, 
   1}, {4, {2, 0}, 0}, {4, {2, 0}, 1}, {1, {2, 1}, 0}, {3, {2, 1}, 
   1}, {0, {2, 2}, 0}, {5, {2, 2}, 1}, {1, {2, 3}, 0}, {1, {2, 3}, 
   1}, {0, {3, 0}, 0}, {0, {3, 0}, 1}, {3, {3, 1}, 0}, {1, {3, 1}, 
   1}, {5, {3, 2}, 0}, {2, {3, 2}, 1}, {2, {3, 3}, 0}, {0, {3, 3}, 
   1}}\qquad
\doublepuzzle[\node at (-\size*0.7,\size*0.95) {$\ss \frac{y_3}{y_2}(1-\frac{y_3}{y_2})(1-\frac{y_4}{y_3})$};]{{4, {0, 0}, 0}, {4, {0, 0}, 1}, {1, {0, 1}, 0}, {1, {0, 1}, 
   1}, {4, {0, 2}, 0}, {4, {0, 2}, 1}, {1, {0, 3}, 0}, {3, {0, 3}, 
   1}, {0, {1, 0}, 0}, {0, {1, 0}, 1}, {7, {1, 1}, 0}, {7, {1, 1}, 
   1}, {0, {1, 2}, 0}, {0, {1, 2}, 1}, {3, {1, 3}, 0}, {1, {1, 3}, 
   1}, {4, {2, 0}, 0}, {4, {2, 0}, 1}, {1, {2, 1}, 0}, {3, {2, 1}, 
   1}, {0, {2, 2}, 0}, {2, {2, 2}, 1}, {2, {2, 3}, 0}, {5, {2, 3}, 
   1}, {0, {3, 0}, 0}, {0, {3, 0}, 1}, {3, {3, 1}, 0}, {1, {3, 1}, 
   1}, {1, {3, 2}, 0}, {1, {3, 2}, 1}, {5, {3, 3}, 0}, {0, {3, 3}, 
   1}}\qquad
\doublepuzzle[\node at (-\size*0.7,\size*0.95) {$\ss \frac{y_3}{y_2}(1-\frac{y_4}{y_2})(1-\frac{y_4}{y_3})$};]{{4, {0, 0}, 0}, {4, {0, 0}, 1}, {1, {0, 1}, 0}, {3, {0, 1}, 
   1}, {0, {0, 2}, 0}, {2, {0, 2}, 1}, {2, {0, 3}, 0}, {0, {0, 3}, 
   1}, {0, {1, 0}, 0}, {0, {1, 0}, 1}, {3, {1, 1}, 0}, {1, {1, 1}, 
   1}, {1, {1, 2}, 0}, {3, {1, 2}, 1}, {0, {1, 3}, 0}, {5, {1, 3}, 
   1}, {4, {2, 0}, 0}, {4, {2, 0}, 1}, {1, {2, 1}, 0}, {3, {2, 1}, 
   1}, {3, {2, 2}, 0}, {1, {2, 2}, 1}, {5, {2, 3}, 0}, {5, {2, 3}, 
   1}, {0, {3, 0}, 0}, {0, {3, 0}, 1}, {3, {3, 1}, 0}, {1, {3, 1}, 
   1}, {1, {3, 2}, 0}, {1, {3, 2}, 1}, {5, {3, 3}, 0}, {0, {3, 3}, 
   1}}
\end{center}
with a total coefficient $(1-y_4/y_2)^2$ of $G^{\stableau{&\\}}$ in the expansion of
$G^{\stableau{\\}} G^{\stableau{\\}} G^{\stableau{&\\}}$.

Actually, one feature that the integrable model does not display is the commutativity
of $K_T(Gr(k,n))$, \ie
$
c^{\lambda,\mu}_{\,\,\nu}
=
c^{\mu,\lambda}_{\,\,\nu}
$
and
$c_{\lambda,\mu}^{\,\,\nu}
=
c_{\mu,\lambda}^{\,\,\nu}$. This means that the planarity of the trees is actually
irrelevant; it would be interesting to see this at the integrable level, as
well as to find an interpretation for Verlinde-type formulae.


\subsubsection{Non-equivariant setting}
Finally, we consider the non-equivariant case, \ie the content of Theorems~\ref{thm:nonequiv} and \ref{thm:nonequivdual}. Now the vertices are completely 120 degree rotationally invariant, and one may choose orientations of edges arbitrarily, since their only use is to keep track of the way to read partitions
(\ie the orientation of an edge marked $\lambda$ can be inverted at the cost of replacing $\lambda$ with $\lambda^\ast$). In particular, the orientation
of internal edges is irrelevant and may be omitted.

Each tree made of such vertices then computes the expectation value
\[
\left<
\prod_{a=1}^r G^{\lambda_a}
\left(\prod_{i=1}^k x_i\right)^{1+\#\{\text{white vertices}\}}
\right>
\]
where the $\lambda_a$ are the partitions associated to external edges
of the tree, and we have assumed for simplicity that all external arrows are incoming.

Many diagrams may therefore lead to the same computation.
For example, $\left<(G^{\stableau{\\}})^6 x_1x_2\right>$ in $K(Gr(2,5))$ (answering
the Schubert calculus question ``how many lines intersect 6 given planes in 4-space?''; the power of $x_1x_2$ is actually irrelevant, since this
is really a question in cohomology) is given by
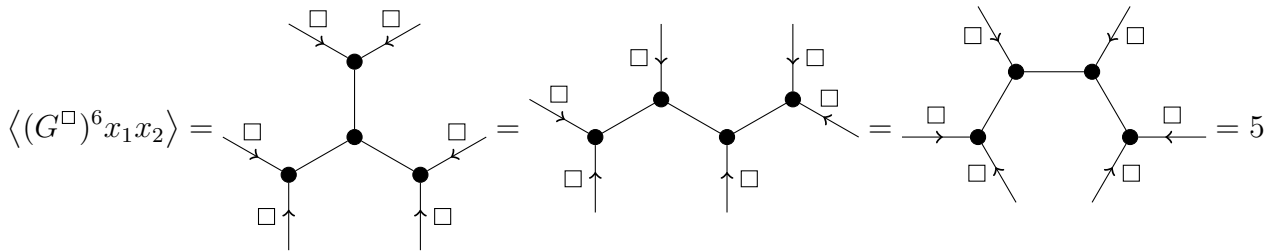
\[
\left<(G^{\stableau{\\}})^6 x_1x_2\right>
=
\begin{tikzpicture}[baseline=0]
\node[mydot=black] (a) at (0,0) {};
\node[mydot=black] (b) at (0,1) {};
\node[mydot=black] (c) at (-30:1) {};
\node[mydot=black] (d) at (-150:1) {};
\draw (a) -- (b) (a) -- (c) (a) -- (d);
\draw[invarrow] (b) -- node[above] {\fund} ++(30:1);
\draw[invarrow] (b) -- node[above] {\fund} ++(150:1);
\draw[invarrow] (c) -- node[above] {\fund} ++(30:1);
\draw[invarrow] (c) -- node[right] {\fund} ++(-90:1);
\draw[invarrow] (d) -- node[above] {\fund} ++(150:1);
\draw[invarrow] (d) -- node[left] {\fund} ++(-90:1);
\end{tikzpicture}
=
\begin{tikzpicture}[baseline=0]
\draw (0,0) node[mydot=black] (a) {}
-- ++(30:1) node[mydot=black] (b) {}
-- ++(-30:1) node[mydot=black] (c) {}
-- ++(30:1) node[mydot=black] (d) {};
\draw[invarrow] (a) -- node[above] {\fund} ++(150:1);
\draw[invarrow] (a) -- node[left] {\fund} ++(-90:1);
\draw[invarrow] (b) -- node[left] {\fund} ++(90:1);
\draw[invarrow] (c) -- node[right] {\fund} ++(-90:1);
\draw[invarrow] (d) -- node[right] {\fund} ++(90:1);
\draw[invarrow] (d) -- node[above] {\fund} ++(-30:1);
\end{tikzpicture}
=
\begin{tikzpicture}[baseline=0]
\draw (0,0) node[mydot=black] (a) {}
-- ++(60:1) node[mydot=black] (b) {}
-- ++(0:1) node[mydot=black] (c) {}
-- ++(-60:1) node[mydot=black] (d) {};
\draw[invarrow] (a) -- node[above] {\fund} ++(180:1);
\draw[invarrow] (a) -- node[left] {\fund} ++(-60:1);
\draw[invarrow] (b) -- node[left] {\fund} ++(120:1);
\draw[invarrow] (c) -- node[right] {\fund} ++(60:1);
\draw[invarrow] (d) -- node[above] {\fund} ++(0:1);
\draw[invarrow] (d) -- node[right] {\fund} ++(-120:1);
\end{tikzpicture}
=5
\]
(note that the first diagram is nothing but $\left<\big(G^{\stableau{&&&\\&&\\\\}}\big)^3 \prod_{i=1}^4 x_i\right>$ in $K(Gr(4,10))$).


\subsection{Mosaics}
As in \cite{Pur,z-j} for puzzles without $K$-tiles, one can reformulate our puzzles
as tilings of certain regions of the plane with a set of elementary
tiles. The $K$-tile then corresponds to a ``shield'' tile, shown last below:
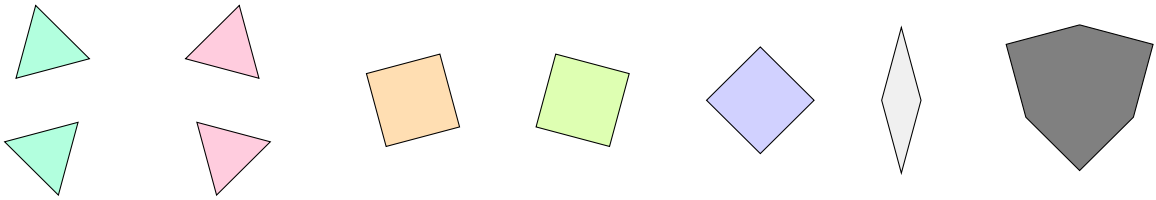
\begin{center}\definecolor{c9}{rgb}{0.5,0.5,0.5}
\begin{tikzpicture}
\matrix[column sep=0.7cm] {
\draw[fill=c0,yshift=0.7cm,rotate=15,scale=0.577] (90:1) -- node[left] {} (210:1) -- node[below] {} (330:1) -- node[right] {} cycle;
\draw[fill=c0,yshift=-0.7cm,rotate=15,scale=0.577] (270:1) -- node[right] {} (30:1) -- node[above] {} (150:1) -- node[left] {} cycle;
&
\draw[fill=c1,yshift=0.7cm,rotate=-15,scale=0.577] (90:1) -- node[left] {} (210:1) -- node[below] {} (330:1) -- node[right] {} cycle;
\draw[fill=c1,yshift=-0.7cm,rotate=-15,scale=0.577] (270:1) -- node[right] {} (30:1) -- node[above] {} (150:1) -- node[left] {} cycle;
&
\draw[fill=c2,rotate=15] (-0.5,-0.5) -- node[left] {} ++(0,1) -- node[above] {} ++(1,0) -- node[right] {} ++(0,-1) -- node[below] {} cycle;
&
\draw[fill=c3,rotate=-15] (-0.5,-0.5) -- node[left] {} ++(0,1) -- node[above] {} ++(1,0) -- node[right] {} ++(0,-1) -- node[below] {} cycle;
&
\draw[fill=c4,rotate=45] (-0.5,-0.5) -- node[below left=-1mm] {} ++(0,1) -- node[above left=-1mm] {} ++(1,0) -- node[above right=-1mm] {} ++(0,-1) -- node[below right=-1mm] {} cycle;
&
\draw[fill=c7]  (0,0) -- node[below] {} ++(75:1) -- node[above] {} ++(-75:1) -- node[above] {} ++(-105:1) -- node[below] {} cycle;
&
\draw[fill=c9] (0,1) -- node[below] {} ++(-15:1) -- node[right] {} ++(-105:1) -- node[right] {} ++(-135:1) -- node[left] {} ++(-225:1) -- node[left] {} ++(-255:1) -- node[below] {} ++(-345:1);
\\
};
\end{tikzpicture}
\end{center}
We refer to the website~\cite{puzzle} for a demonstration of this representation.

\section*{Acknowledgments}

MW is supported by the ARC grant DE160100958 and the ARC Centre of Excellence for Mathematical and Statistical Frontiers (ACEMS). PZJ is supported by ERC grant ``LIC'' 278124 and ARC grant DP140102201.

\gdef\MRshorten#1 #2MRend{#1}%
\gdef\MRfirsttwo#1#2{\if#1M%
MR\else MR#1#2\fi}
\def\MRfix#1{\MRshorten\MRfirsttwo#1 MRend}
\renewcommand\MR[1]{\relax\ifhmode\unskip\spacefactor3000 \space\fi
\MRhref{\MRfix{#1}}{{\scriptsize \MRfix{#1}}}}
\renewcommand{\MRhref}[2]{%
\href{http://www.ams.org/mathscinet-getitem?mr=#1}{#2}}
\bibliographystyle{amsplainhyper}
\bibliography{refs}

\end{document}